\documentclass[11pt]{amsart}
\usepackage{times}
\usepackage[T1]{fontenc}
\usepackage{amssymb, amsthm, amsmath}
\usepackage{mathrsfs}

\usepackage[active]{srcltx}

\usepackage{todonotes}

\usepackage{setspace}
\usepackage{geometry}

\usepackage{hyperref}

\DeclareMathOperator{\acl}{acl}
\DeclareMathOperator{\dcl}{dcl} 
\DeclareMathOperator{\gl}{GL} 
\DeclareMathOperator{\aut}{Aut} \DeclareMathOperator{\id}{Id}
 \DeclareMathOperator{\fr}{Fr}

 \DeclareMathOperator{\dom}{dom}

\DeclareMathOperator{\eq}{eq}

\DeclareMathOperator{\cl}{cl}

\newtheorem{introtheorem}{Theorem}

\newtheorem{theorem}{Theorem}[section]

\newtheorem{claim}{Claim}[theorem]

\newtheorem{corollary}[theorem]{Corollary}

\newtheorem{fact}[theorem]{Fact}
\newtheorem{lemma}[theorem]{Lemma}

\newtheorem{proposition}[theorem]{Proposition}

\newtheorem*{gen-dif}{\fbox{{\large A}} \hypertarget{Agen-dif}{Gen-Dif}}

\newtheorem*{min-balln}{\fbox{{\large A}} \hypertarget{Amin-ball}{Cballs}}

\theoremstyle{definition}
\newtheorem{definition}[theorem]{Definition}
\newtheorem{example}[theorem]{Example}
\newtheorem{remark}[theorem]{Remark}

\newcommand{\CA}{\mathcal A}

\newcommand{\Nn}{{\mathbb{N}}}
\newcommand{\Qq}{{\mathbb{Q}}}

\newcommand{\Zz}{{\mathbb {Z}}}

\newcommand{\m}{\textbf{m}}
\newcommand{\bk}{\textbf{k}}

\newcommand{\CL}{{\mathcal L}}
\newcommand{\CK}{{\mathcal K}}
\newcommand{\CN}{{\mathcal N}}

\newcommand{\CR}{{\mathcal R}}
\newcommand{\CM}{{\mathcal M}}

\newcommand{\CC}{{\mathcal C}}

\newcommand{\CO}{{\mathcal O}}

\newcommand{\CF}{{\mathscr F}}
\newcommand{\CG}{{\mathcal G}}

\newcommand{\0}{\emptyset}

\newcommand{\rest}{\upharpoonright}

\renewcommand{\phi}{\varphi}

\newcommand{\ad}{\mathrm{Ad}}

\def\bm{\mathfrak m}

\def\qp{\mathbb Q_p}

\def\dpr{\mathrm{dp\text{-}rk}}
\def\sub{\subseteq}

\newenvironment{claimproof}[1][\proofname]
  {%
    \proof[#1]%
  }
  {%
    \endproof%
  }

\date{\today}
\title{Semisimple groups  interpretable in various valued fields}
\author{Yatir Halevi}
\address{Faculty of Natural Sciences, Department of Mathematics\\ University of Haifa, Haifa, Israel}
\email{yatirh@gmail.com}

\author{Assaf Hasson}
\address{Department of Mathematics, Ben Gurion University of the Negev, Be'er-Sheva, Israel}
\email{hassonas@math.bgu.ac.il}

\author{Ya'acov Peterzil}
\address{Faculty of Natural Sciences, Department of Mathematics, University of Haifa, Haifa, Israel}
\email{kobi@math.haifa.ac.il}

\date{February 2025}

\begin{document}

\thanks{The first author was partially supported by ISF grant No. 555/21 and 290/19. The second author was supported by ISF grant No. 555/21. The third author was supported by ISF grant No. 290/19.}

\begin{abstract}
     We study infinite groups interpretable in power bounded $T$-convex, $V$-minimal or $p$-adically closed fields. We show that if $G$ is an interpretable definably semisimple group (i.e.,  has no definable infinite normal abelian subgroups) then, up to a finite index subgroup,  it is definably isogenous 
to a group $G_1\times G_2$, where $G_1$ is a $K$-linear group and $G_2$ is a $\mathbf{k}$-linear group. The analysis is carried out by studying the interaction of $G$ with four distinguished sorts: the valued field $K$, the residue field $\mathbf{k}$, the value group $\Gamma$,  and the closed $0$-balls  $K/\mathcal{O}$.  
\end{abstract}

\maketitle

\tableofcontents

\section{Introduction}

We continue the study of groups interpretable in three classes of tame valued fields: $p$-adically closed fields (and their analytic expansions), power bounded $T$-convex expansions of o-minimal real closed fields, and $V$-minimal expansions of algebraically closed valued fields of equi-characteristic $0$.  

The tameness conditions in each of these classes have significant geometric implications on definable sets. For example, they imply a well behaved notion of dimension, generic differentiability of definable functions $f: K^n\to K$ with corresponding versions of Taylor's approximation theorem, and more (see, e.g., \cite{hensel-min}). For definable groups, expanding on Pillay's work in the o-minimal context \cite{Pi5} (and see also \cite{PilQp}), this gives rise to a rudimentary Lie theory (\cite{AcHa}). 

 \vspace{.2cm}

A group $G$ is  {\em interpretable} in a structure $\CK$ if its universe is the quotient of a definable set by a definable equivalence relation and multiplication is part of the induced structure. The powerful geometric tools described above are not directly available for the study of interpretable groups. 
Our general program aims, therefore, to exploit those tools (as well as tameness of the value group $\Gamma$,  and the residue field $\bk$) to give structure theorems for interpretable groups using groups that are better understood by virtue of being definable in a small collection of well studied sorts.    

In our previous works, \cite{HaHaPeGps} and \cite{HaHaPeVF}, we showed that any group $G$ interpretable in $\CK$ has "infinitesimal" type-definable subgroups definably isomorphic to groups that are (type)-definable in one of the four {\em distinguished} sorts: the valued field sort $K$, the value group, the residue field (when infinite) and the sort of closed $0$-balls $K/\CO$.  Our strategy here is to understand interpretable groups using these type-definable groups and their construction.

In \cite{HaHaPeVF} we used this analysis to describe all interpretable fields in those families of structures. 
Here we use it to study {\em definably semisimple groups}, namely groups which contain no infinite definable normal abelian subgroups. Our main theorem (Theorem \ref{T: main} below) is:

\begin{introtheorem}\label{T: intro}
    Let $\CK$ be either a power bounded $T$-convex field, a $V$-minimal field or a $p$-adically closed field. Let $G$ be an interpretable definably semisimple group in $\CK$. Then there exists a finite normal subgroup $N\trianglelefteq G$ and two normal subgroups $H_1,H_2\trianglelefteq G/N$, such that 
    \begin{enumerate}
        \item $H_1\cap H_2=\{e\}$, $H_1$ and $H_2$ centralize each other and $H_1\cdot H_2$ has finite index in $G/N$.
               \item $H_1$ is definably isomorphic to a subgroup of $\gl_n(\bk)$.
     \item $H_2$ is definably semisimple and definably isomorphic to a subgroup of $\gl_n(K)$.
    \end{enumerate}
\end{introtheorem}

It may be worth pointing out, with regard to the formulation of the above theorem, that in our setting,  definable semisimplicity is preserved under finite quotients (Corollary \ref{C: semisimple in quotients}). we make use of this several times in the proof of the theorem. 

We have been informed by J. Gismatullin,  I. Halupczok and D. Macpherson that in a recent unpublished work \cite{GisHalMac} they characterize simple groups \emph{definable}  in certain Henselian valued fields of characteristic $0$ (covering the classes of fields discussed in the present paper).  Their work seems to combine with the present one to characterize definably simple groups interpretable in our settings.

Our proof goes through a case by case reduction to one of the four distinguished sorts. This is based on \cite{HaHaPeGps}, where we showed that after modding out by a finite subgroup, $G$ is {\em locally strongly internal} to one of the distinguished sorts $D$, namely there exists 
 an infinite definable set $X\sub G$ and a definable injection $f:X\to D^k$, for some $k$.

The main obstacle is to eliminate the cases when  $D=\Gamma,K/\CO$. In Proposition \ref{P: Gamma} we show that if 
 $G$ is locally strongly internal to $\Gamma$ then it contains a definable normal finite index subgroup whose center is infinite, which prohibits $G$ from being definably semisimple. 
A more intricate result, Proposition \ref{P: K/O}, allows us to conclude that a definably semisimple group $G$ cannot be locally strongly internal to $K/\CO$. 

When $G$ is locally strongly internal to $K$ we use local differentiability of definable functions with respect to  $K$, and basic Lie theory over $K$, to associate to $G$ an adjoint representation over $K$. When $D=\bk$ we either use similar methods, in the $T$-convex case, or use the theory of groups of finite Morley rank, in the V-minimal case, to complete the proof.

Though the statement of Theorem \ref{T: intro} and some of the   auxiliary results  often hold in all settings regardless of whether $\CK$ is $p$-adically closed, power bounded $T$-convex or V-minimal, some of the proofs depend  on the specific context. E.g., o-minimality of the value group plays a crucial in our analysis of $\Gamma$-groups in the V-minimal and power bounded $T$-convex setting, and a rather different analysis --  albeit with a similar conclusion -- is needed for the $p$-adic case. 

\begin{remark}
We note that a-priori the notion of definable semisimplicity (more precisely,  the existence of an infinite definable normal abelian subgroup) need not be elementary.
Indeed, while the valued field sort in our settings is a geometric structure, so in particular has uniform finiteness (sometimes called ``elimination of $\exists^\infty$'') for definable families of subsets of $K^n$, 
the same  might not be true in $\CK^{eq}$.

Johnson, \cite{JohCminimalexist}, shows, in the V-minimal case, that $\CK^{eq}$ does eliminate $\exists^\infty$ and using his methods we show the same for power bounded $T$-convex structures (see Section \ref{P:exists-infty V-min or T-conv}). However, 
in the $p$-adically closed case this fails in $\CK^{eq}$, as neither  $\Gamma$  nor  $K/\CO$ have uniform finiteness. Nevertheless,  one of the consequences of the present work is that definable semisimplicity is indeed an elementary property in all cases.
\end{remark}

\begin{remark}
    In the power bounded $T$-convex  case, our work makes use of results from James Tyne's PhD thesis, \cite{tynephd}, which as far as we know, have not been published elsewhere. These results, together with the work of van den Dries, \cite{vdDries-Tconvex}, imply that every definable subset of $K$ is a boolean combination of balls and intervals (first proven by Holly, \cite{holly}, for real closed valued fields). In order to make the results available in print, we include in the appendix direct proofs.
\end{remark}

\vspace{.2cm}

\noindent{\bf Previous work} 
We note recent work on interpretable groups in $p$-adically closed fields, by Johnson, \cite{JohnTopQp}, also together with  Yao,  \cite{JohnYao}, \cite{JohnYaoAbelian},  and with Guerrero, \cite{JohnGue}.  Further work is needed in order to understand the relation between our methods and the model theoretic tools studied there,  such as definable compactness, finitely satisfiable generics (fsg), definable $f$-generics (dfg), etc.

\vspace{.2cm}

\noindent\emph{Acknowledgement} We would like to thank J. Gismatullin, I. Halupczok and D. Macpherson  for sharing with us their unpublished work on simple groups definable in certain henselian fields. We also thank D. Macpherson for several conversations and useful suggestions, and E. Sayag for directing us to some useful references.  Finally, we thank the referee for a careful reading of the paper and for noticing several errors which required fixing.

\section{Preliminaries and Notation}
We set up some notation and terminology, and review some of the basic facts concerning the main objects of interest in the present paper.  Throughout, structures are denoted by  calligraphic capital letters, $\CM$, $\CN$, $\CK$ etc., and their respective universes by the corresponding Latin letters, $M$, $N$ and $K$. 

Tuples from a structure $\CM$ are always assumed to be finite, and are denoted by small Roman characters $a,b,c,\dots$. We apply the standard model theoretic abuse of notation writing $a\in M$ for $a\in M^{|a|}$. Variables will be denoted $x,y,z,\dots$ with the same conventions as above. We do not distinguish notationally between tuples and variables belonging to different sort, unless some ambiguity can arise. Capital Roman letters $A,B,C,\ldots$ usually denote small subsets of parameters from $\CM$. As is standard in model theory, we write $Ab$ as a shorthand for $A\cup \{b\}$. In the context of definable groups we will, whenever confusion can arise, distinguish between, e.g., $Agh:=A\cup\{g,h\}$ and $A\,g\!\cdot \! h:=A\cup \{g\!\cdot\! h\}$. 

By a partial type we mean a consistent collection of formulas. Two partial types $\rho_1, \rho_2$ are equal, denoted $\rho_1=\rho_2$, if they are logically equivalent, i.e., if they have the same realizations in some sufficiently saturated elementary extension.

All the definable sets we shall consider here have finite dp-rank, whose properties  (such as sub-additivity, invariance under finite-to-finite correspondences, invariance under automorphisms etc.) we use freely. See the preliminaries sections of \cite{HaHaPeVF},\cite{HaHaPeGps} for a more detailed discussion.

\subsection{Valued fields}
Throughout $\CK$ denotes an expansion of a  valued  field of characteristic $0$ in a language $\CL$ expanding the language of valued rings. We assume $\CK$ to be  $(|\CL|+2^{\aleph_0})^+$-saturated.

Unless specifically written otherwise, we will always work in $\CK^{\eq}$.  
{\bf Henceforth, by ``definable'' we mean ``definable in $\CK^{\eq}$ using parameters'', unless specifically mentioned otherwise}. In particular, we shall not use ``interpretable'' anymore. A more detailed review of standard definitions and notation can be found in \cite[\S 2]{HaHaPeGps}.

For any valued field $(K,v)$, we let $\CO$ denote its valuation ring,  $\m$ its maximal ideal and $\bk:=\CO/\m$ the residue field. The value group is denoted $\Gamma$. In case of possible ambiguity, we may, for the sake of clarity, add a subscript (e.g., $\CO_K$) to the above notation.

A closed ball in $K$ is a set of the form $B_{\geq \gamma}(a):=\{x\in K: v(x-a)\geq \gamma\}$ and similarly $B_{>\gamma}(a)$ denotes the open ball of (valuative) radius $\gamma$ around $a$. We will use the fact that $v$ descends naturally to $K/\CO\setminus \{0\}$ (by $v(a+\CO):=v(a)$ for any $a\notin \CO$), and use the same notation $B_{>\gamma}(x)$ and $B_{\ge \gamma}(x)$ for $x\in K/\CO$ in the obvious way. We will, however, reserve \textbf{the term ``ball'' in $K/\CO$, when $\CK$ is $p$-adically closed, only to  such sets where $\gamma<\Zz$}.   For $a=(a_1,\dots,a_n)\in K$ (or in $(K/\CO)^n$) we set $v(a)=\min_i\{v(a_i)\}$. A ball in $K^n$ (or in $(K/\CO)^n$) is an $n$-fold  product of $K$-balls (or $(K/\CO)$-balls) of {\bf equal radii}. 

When $\CK$ is $p$-adically closed, it is elementarily equivalent to some finite  extension  $\mathbb{F}$ of $\mathbb{Q}_p$.  By saturation, we may assume that $(K,v)$ is an elementary extension of $(\mathbb{F},v)$.  Since its value group $\Gamma_{\mathbb{F}}$ is isomorphic to $\mathbb{Z}$, as ordered abelian groups,  we identify $\Gamma_{\mathbb{F}}$ with $\mathbb{Z}$ and view it as a prime (and minimal) model for $\Gamma$. We denote $\Zz_{Pres}$ the structure $(\Zz, +, <)$.

\subsection{The setting}
Unless otherwise stated, $\CK$ is a saturated expansion of a valued field of one of  three types
 (see \cite{HaHaPeGps} for definitions and more details):

\begin{itemize}
    \item A $V$-minimal expansion of an algebraically closed valued field of residue characteristic $0$. 
    \item A $T$-convex expansion of a real closed valued field, for an o-minimal power bounded theory $T$.
    \item A $p$-adically closed field.
\end{itemize}

\begin{remark}
    Our proof for the $p$-adically closed case works, as written, in the context of $P$-minimal 1-h-minimal fields with definable Skolem functions in the valued field sort. These include models of the theory of $\qp^{an}$, the expansion of $\qp$ (or a finite extension thereof) by all convergent power series $f: \CO^n\to \qp$ (any $n$). For the sake of clarity of exposition, we stick to the $p$-adically closed case. 
\end{remark}

 There are important similarities between the three settings. E.g.,  in all cases the structure $\CK$ is  dp-minimal, namely $\dpr(\CK)=1$, so definable sets in $\CK^{eq}$ have finite dp-rank. Also, in all cases the valued field sort is a geometric structure, carrying, moreover,  the structure of an SW-uniformity. The latter introduced (without the name) by Simon and Walsberg, \cite{SimWal}: 
 \begin{definition}
    A dp-minimal expansion of a topological group $G$ is {\em 
    an SW-uniformity} if it supports a definable  group topology, with no isolated points and such that every infinite definable subset has non-empty interior.    
 \end{definition}
 In \cite{SimWal} the underlying setting is that of a definable uniformity inducing the topology. The existence of such a uniformity is automatic in the context of topological groups with a definable basis for the topology.

There are, however,  also obvious differences between the three settings. For example, the residue field is stable in the $V$-minimal case, o-minimal in the $T$-convex case and finite in the $p$-adic case. 
Thus, while the main theorems can be stated uniformly in all settings, some of the  proofs will require us to specialize to the particular cases. 

\subsection{The distinguished sorts}
As in our previous work, the analysis of definable quotients is carried out via a reduction to four {\em distinguished sorts},  $K, \Gamma, \bk$ and $K/\CO$. They are all dp-minimal, except the finite $\bk$ in the $p$-adic case. Note that in all cases the sorts $K$, $\Gamma$ and $K/\CO$ are partially ordered and therefore unstable. However, the residue field sort is unstable only in the $T$-convex case (in the $V$-minimal case it is a pure algebraically closed field, and in the $p$-adic case it is finite). Thus, when proofs mention the  ``unstable sorts'' they refer to the distinguished sorts in all three cases except for $\bk$ in the $V$-minimal and $p$-adically closed settings.

As noted above, in all settings the sort $K$ is an SW-uniformity, as is $\Gamma$ in the $V$-minimal and $T$-convex cases (it is in fact an ordered vector spaces so o-minimal) and $K/\CO$ in the $T$-convex setting (it is weakly o-minimal). However, in all cases $K/\CO$ is neither  a geometric structure ($\acl(\cdot)$ in $K/\CO$ does not satisfy the Steinitz Exchange Principle) nor is it  stably embedded, leading to certain complications in some  proofs.

\begin{remark}\label{R: assuming named F}
In \cite[\S3]{HaHaPeGps} we study the structure of $K/\CO$ in $p$-adically closed fields. In this context, it was helpful to work in a saturated model, expanding the language by constants for all  elements of (a copy of) $\mathbb{F}$.

Although the saturation assumption on $\CK$ plays an important role in many of our proofs here, the main theorems  of the present paper do not assume saturation. Thus,  a copy of $\mathbb F$ cannot be expected to exist in all our models (let alone be named). Whenever needed, as part of the proof,  we bridge this gap in the assumptions.
\end{remark}

\subsection{Some specialized terminology} 
 We remind some terminology from \cite{HaHaPeGps} that is used throughout the paper: 

Assume that  $S$ is definable in $\CK$ and $D$ is one of the distinguished sorts. We say that $S$ is \emph{locally almost strongly internal to $D$} if {\bf in a sufficiently saturated elementary extension}
  there is a definable infinite set $X\sub S$ and a definable $m$-to-one map $f: X\to D^n$, for some $m,n\in \Nn$. The set $X$ is then called \emph{almost strongly internal to $D$}. If we can find a definable injection $f:X\to D^n$ then $S$ is \emph{locally strongly internal to $D$} and $X$ is \emph{strongly internal to $D$}.  We add ``over $A$'' to all the notions above if $S,X$ and the map $f$ are defined over a parameter set $A$.

  The starting point of our analysis is the following (\cite[Lemma 7.3, Lemma 7.6, Lemma 7.10]{HaHaPeGps}):
  \begin{fact}\label{F: Reduction to sorts}
   Every definable infinite set $S$ in $\CK$ is locally almost strongly internal to $K$, $\bk$, $\Gamma$ or $K/\CO$.    
  \end{fact}

A {\em $D$-critical subset of $S$} is a definable $X\sub S$ of maximal dp-rank that is strongly internal to $D$.  The \emph{$D$-rank\footnote{In \cite{HaHaPeGps} this was called the $D$-critical rank of $S$.}} of $S$ is the dp-rank of any $D$-critical $X\sub S$. The \emph{almost $D$-rank} of $S$ is the maximal dp-rank of a definable set $X\sub S$ almost strongly internal to $D$. A set $X\sub S$ is \emph{almost $D$-critical} if $\dpr(X)$ is the almost $D$-rank of $S$, and the size of the fibers of some function witnessing almost strong internality of $X$ is minimal possible, among all sets of the same dp-rank.



The set  $S$ is \emph{$D$-pure}  if it is locally almost strongly internal to $D$ but not to any other distinguished sort.  

\begin{definition}
    Let $X$ be an $A$-definable set in $\CK$, $a\in X$ and $B\supseteq A$ a set of parameters.
    \begin{enumerate}
        \item The point $a$ is $B$-\emph{generic}  in $X$ (or, {\em generic in $X$ over $B$}) if $\dpr(a/B)=\dpr(X)$.
        \item For an  $A$-generic  $a\in X$, a set $U\sub X$ is \emph{a $B$-generic vicinity of $a$ in $X$} if $a\in U$, $U$ is $B$-definable, and $\dpr(a/B)=\dpr(X)$ (in particular, $\dpr(U)=\dpr(X)$).
    \end{enumerate}
\end{definition}

In order to overcome the failure of additivity of dp-rank, we introduced in \cite{HaHaPeGps} the notion of a $D$-group. In the present paper this notion can be used as a black box allowing us to seamlessly refer to results from \cite{HaHaPeGps}. However, for the sake of completeness, we give the definition:   For $D$ one of the unstable distinguished sorts, an $A$-definable group $G$ is a \emph{ $D$-group} if it is locally strongly internal to $D$ and for every $X_1, X_2\subseteq G$  strongly internal to $D$, with $X_2$ $D$-critical in $G$,  both defined  over some $B\supseteq A$, and for every $(g,h)$ $B$-generic in $X_1\times X_2$, we have \[\dpr(g/B,g\cdot h)=\dpr(g/B).\]

We stress that, by definition, the notion of a $D$-group refers only to unstable $D$, namely all infinite sorts in our setting except $\bk$ in the $V$-minimal case. 
 The following fact shows that a group $G$  almost strongly internal to an unstable sort $D$ is close to being a $D$-group.
 
 \begin{fact}\cite[Fact 4.25, Proposition 4.35]{HaHaPeGps}\label{F: existence of finite normla to get D-group}
     Let $G$ be an infinite $A$-definable group in $\CK$ locally almost strongly internal to an unstable distinguished sort $D$. Then there is an $A$-definable finite  normal abelian subgroup $H\trianglelefteq G$ such that $G/H$ is a $D$-group. Moreover, 
     \begin{enumerate}
         \item The almost $D$-rank and the $D$-rank of $G/H$ are equal (and equal to the almost $D$-rank of $G$).
         \item $H$ is invariant under any definable automorphism of $G$ and is contained in any definable finite index subgroup of $G$. 
     \end{enumerate}
 \end{fact}
 
Recall that every definable group in $\CK$ is almost locally strongly internal to one of  the distinguished sorts, hence the above fact applies whenever that sort is unstable.

 \subsection{Vicinities and  infinitesimal subgroups}\label{ss:infint and vicin}

 In this section we recall the notion of a vicinic set and that of an infinitesimal group from \cite{HaHaPeGps}. Before proceeding, we  clarify the relation between several  $\acl$-related notions of dimension.

 \begin{definition}
For $D$ a definable set, a parameter set $A$, and $a\in D^n$, denote:
\begin{enumerate}
    \item $\dim_{\acl}(a/A)$ the minimal length of a sub-tuple $a'\sub a$ such that $\acl(a'A)=\acl(aA)$ and 
    \item $\dim_{\mathrm{ind}}(a/A)$  the maximal size of a sub-tuple $a'\sub a$ which is $\acl$-independent over $A$  (namely, no $a_i\in a'$ is in $\acl(A\cup a'\setminus \{a_i\})$).
\end{enumerate}
 \end{definition}

If $\acl$ satisfies Exchange on $D$ it is well known and easy to see that $\dim_{\acl}=\dim_{
  \mathrm{ind}}$. In general, we only have  $\dim_{\mathrm{ind}}(a/A)\geq \dim_{\acl}(a/A)$. In our setting, however, more is true: 

\begin{lemma}
    For $D$ a dp-minimal definable set, the following are equivalent:
    \begin{enumerate}
        \item For every tuple $a\in D^n$ and set $A$, $\dim_{\acl}(a/A)=\dpr(a/A)$.
        \item For every tuple $a\in D^n$ and set $A$, $\dim_{\mathrm{ind}}(a/A)=\dpr(a/A)$.
    \end{enumerate}
\end{lemma}
\begin{proof}
    By dp-minimality and sub-additivity of dp-rank  $\dpr(a/A)\leq \dim_{\acl}(a/A)$, proving  $(2)\Rightarrow (1)$. For the other direction,  assume $(1)$.

Let $a'\sub a$ be $\acl$-independent over $A$ of maximal length $d$, namely, $d=\dim_{\mathrm{ind}}(a/A)$. Since $a'$ is $\acl$-independent over $A$, $\dim_{\acl}(a'/A)=d$, which by assumption equals $\dpr(a'/A)$. Thus, $\dpr(a/A)\geq \dpr(a'/A)=d=\dim_{\mathrm{ind}}(a/A)$, and equality of $\dpr$ and $\dim_{\mathrm{ind}}$ follows.
\end{proof}

\begin{remark}

In \cite{HaHaPeGps}, we used a slightly different definition of $\dim_{\acl}$, that we assumed throughout, to be equal to $\dpr$. It follows immediately from the lemma  that under this  assumption this notion of dimension is also equal to $\dim_{\acl}$ as defined here (and thus also to $\dim_{\mathrm{ind}}$).
\end{remark}





We recall the following from \cite{HaHaPeGps}:

\begin{definition}
    A dp-minimal set $D$ is \emph{vicinic} if it satisfies the following axioms: 
\begin{enumerate}
    \item[(A1)] $\dim_{\acl}=\dpr$; i.e. for any tuple $a\in D^n$ and set $A$, $\dim_{\acl}(a/A)=\dpr(a/A)$.
    \item[(A2)] For any sets of parameters  $A$ and $B$,  for every $A$-generic elements $b\in D^n$, $c\in D^m$ and  any $B$-generic vicinity $X$  of $b$ in $D^n$,  there exists $C\supseteq A$ and a $C$-generic vicinity of $b$ in $X$ such that $\dpr(b,c/A)=\dpr(b,c/C)$.
\end{enumerate}
\end{definition}

By \cite[Fact 4.7]{HaHaPeGps}, all the unstable distinguished sorts in our settings are vicinic.
Throughout this subsection, unless specifically stated otherwise, we let $D$ be one of them. Given a definable $D$-group $G$ in $\CK$ the main technical result of \cite{HaHaPeGps} is the construction of the infinitesimal type-definable subgroup $\nu_D$. To achieve this, we introduce the notion of $D$-sets (in $G$). For completeness, we remind the somewhat technical definition. Note, however, that we do not give the original definition, we switch the original formulation  of ``minimal fibers'' with an equivalent one, see \cite[Remark 4.12]{HaHaPeGps}.  The fine details of the definition are unimportant for us here: 

\begin{definition}\cite[Definition 4.16]{HaHaPeGps}
    A definable set $X\subseteq G$ is a \emph{$D$-set over $A$ in $G$} if it is $D$-critical in $G$, witnessed by some $A$-definable function $f:X\to D^m$ and there exists a coordinate projection $\pi:f(X)\to D^n$, with $n=\dpr(X)$, such that for every $B\supseteq A$ and $B$-generic $a\in f(X)$, all elements of $\pi^{-1}(\pi(f(a)))$ have the same type over $B\pi(f(a))$.
\end{definition}

\begin{remark}\label{R: D-sets}
\begin{enumerate}
    \item If $G$ is a definable group locally strongly internal to $D$ then it always contains a $D$-set. See \cite[Remark 4.18]{HaHaPeGps}.
    \item Note the following special case: if $X$ is $D$-critical, $f: X\to D^n$ a definable injection witnessing it, and \underline{$n=\dpr(X)$} then $X$ is a $D$-set.  As we shall see, such an $X$ can always be found when $G$ is locally strongly internal to $D$. If $D$ is an SW-uniformity this follows from \cite[Proposition 4.6]{SimWal} and in the $p$-adically closed case this follows from Proposition \ref{P:local homeo K/O} when $D=K/\CO$ and cell decomposition when $D=\Gamma$. See Lemma \ref{L: Dsets} for more information.    
\end{enumerate}

\end{remark}

\begin{definition} 
Let $G$ be a $D$-group,  $Z\sub G$ a $D$-set over $A$ and $d\in Z$ an $A$-generic point.
The {\em infinitesimal vicinity of $d$ in $Z$},  denoted $\nu_Z(d)$, is the partial type consisting of all $B$-generic vicinities of $d$ in $Z$, as $B$ varies over all small parameter subsets of $\CK$.
\end{definition}

By \cite[Lemma 4.20]{HaHaPeGps}, the type $\nu_Z(d)$ is a filter-base, namely the intersection of any two generic vicinities of $d$ contains another. It follows that $\dpr(\nu_Z(d))$ equals the $D$-rank of $G$.

We think of $\nu_Z(d)$ (and the type definable group $\nu_D$ defined below) both as a collection of formulas over $\CK$ and a set of realization of the partial type in some monster model extending $\CK$. We say that two such types are equal if they are logically equivalent. For a definable set $X$ we denote $\nu_Z(d)\vdash X$  if there is $Y\in \nu_Z(d)$ such that $Y\sub X$. By writing $\nu_Z(d)\vdash \nu_W(d')$ we mean that for all $X\in \nu_W(d')$ we have $\nu_Z(d)\vdash X$.

\begin{fact}\cite[Proposition 5.8]{HaHaPeGps}\label{F: properties of nu}
Let $D$ be an unstable distinguished sort and  let $G$ be a $D$-group.
\begin{enumerate}
\item Assume that $X\subseteq G$ is a $D$-set over $A$, then for every $A$-generic $a,b\in X$ the set $\nu_X(a)a^{-1}$ is a (type-definable) subgroup of $G$ and  
$\nu_X(a)a^{-1}=\nu_X(b)b^{-1}=a^{-1}\nu_X(a)$. We denote this group $\nu_X$.

\item If $X,Y\sub G$ are $D$-sets over $A$ then $\nu_X=\nu_Y$, and we can call it $\nu_D(G)$, \emph{the infinitesimal type-definable subgroup of $G$ with respect to $D$}.
\item For every $g\in G(\CK)$, we have $g\nu_D(G) g^{-1}=\nu_D(G)$. In fact, $\nu_D$ is invariant under any $\CM$-definable automorphism of $G$. 
\end{enumerate}
\end{fact}

Whenever the group $G$ is understood from the context and there is no ambiguity, we denote $\nu_D(G)$ by $\nu_D$.

\begin{remark}\label{R: nu lives on any definable subgroup witnessing}
Note that if $X\subseteq G$ is a $D$-set which  happens to be a subgroup, then $\nu_D\vdash X$. 
\end{remark}

\begin{lemma}\label{L:nu of subgroup} 
Let $H\leq G$ be two definable $D$-groups, locally strongly internal to an unstable distinguished sort $D$. Then 

\begin{enumerate}
    \item $\nu_D(H)\vdash \nu_D(G)$.
    \item If $H$ and $G$ have the same $D$-rank then $\nu_D(H)=\nu_D(G)$. In particular, this holds if $H$ has finite index in $G$.
\end{enumerate}
\end{lemma} 
\begin{proof} 
Let $H\leq G$ be any subgroup, as in the statement.

(1) Let $X_G\subseteq G$ be a $D$-set in $G$  and $X_H\subseteq H$ a $D$-set in $H$,  all definable over a parameter set $A$. Let $(g,h)\in X_G\times X_H$ be  generic over $A$, so $\nu_D(G)=g^{-1}\nu_{X_G}(g)$ and $\nu_D(H)=h^{-1}\nu_{X_H}(h)$.

Let $V$ be a generic vicinity of $g$ and $U$ a generic vicinity of $h$.  By \cite[Lemma 4.26]{HaHaPeGps}, $U\cap hg^{-1} V$ is a generic vicinity of $h$, hence
\[\nu_D(H)\vdash h^{-1}(U\cap hg^{-1}V)=h^{-1}U\cap g^{-1}V\sub g^{-1}V.\]

(2) Assume that $H$ and $G$ have the same $D$-rank, hence any $D$-set in $H$ is automatically a $D$-set in $G$. It now follows by definition that $\nu_D(H)=\nu_D(G)$.

If $H$ has finite index in $G$ then it is easy to see that they have the same $D$-rank. 
\end{proof} 

The next lemma supports the intuition that the type-definable coset $g\cdot \nu_D(G)$ is an infinitesimal neighborhood of $g$, for $g$ generic in a set locally strongly internal to $D$: 

\begin{lemma}\label{L:intersects largely} 
Let $G$ be a $D$-group,  $X\subseteq G$ an $A$-definable set strongly internal to $D$ over $A$, and  $g\in X$ generic over $A$.  Then $\dpr(X\cap g\cdot \nu_D )=\dpr(X)$. 
\end{lemma}
\begin{proof}
 Let $Z'$ be any $D$-set, definable over some parameter set $B'$. Find an element $g'\equiv_A g$ such that $\dpr(g'/AB')=\dpr(g/A)$. Applying an automorphism over $A$ we can move $g'$ to $g$ and $B'$ to some $B$.  The image, $Z$, of $Z'$ under this automorphism, is  definable over $B$ and $\dpr(g/AB)=\dpr(g/A)$. Renaming, we assume from now on, that  $A=AB$.

Fix an $A$-generic $h\in Z$ with $\dpr(g,h/A)=\dpr(X)+\dpr(Z)$. Thus, as $\nu_D=h^{-1}\nu_Z(h)$, we have to show that $\dpr(X\cap gh^{-1}\nu_Z(h))=\dpr(X)$.

Let $Y\subseteq Z$ be some $B$-generic vicinity of $h$ (i.e. $Y\in \nu_Z(h)$), for some $B$; so it will suffice to prove that $\dpr(X\cap gh^{-1}Y)=\dpr(X)$.

By \cite[Lemma 4.13]{HaHaPeGps}, there exists $C\supseteq A$ and a $C$-generic vicinity $Y'\subseteq Y$ of $h$ such that $\dpr(g,h/A)=\dpr(g,h/C)$. So $(g,h)$ is $C$-generic in $X\times Y'$. It will be sufficient to prove that $\dpr(X\cap gh^{-1}Y')=\dpr(X)$; this is exactly \cite[Lemma 4.26]{HaHaPeGps}.
 \end{proof}

\begin{lemma}\label{L:passage of D-group under finite-to-one}
Let $G$ be a definable group in $\CK$, $H$ a finite normal subgroup and $f:G\to G/H$ the quotient map. Let $D$ be any of the distinguished sorts.

\begin{enumerate}
    \item The almost $D$-ranks of $G$ and $G/H$ are equal.

\end{enumerate}

 For the following assume that $D$ is not $K/\CO$ in the $p$-adically closed case.
\begin{enumerate}
    \item[(2)] The $D$-rank of $G$ is at most the $D$-rank of $G/H$. 
    
    \item[(3)] If, furthermore, $G$ is $D$-group (so $D$ is unstable) then so is $G/H$, and then $f(\nu_{D}(G))=\nu_{D}(G/H)$.
    \item[(4)] If the $D$-critical rank  and the almost $D$-critical ranks of $G$ coincide, then the same is true for $G/H$.
\end{enumerate}
\end{lemma}
\begin{proof} 

For (1) and  (2) we first note that  for any (almost) $D$-critical set $X\subseteq G$, there exists an (almost) $D$-critical  $Y\subseteq f(X)$ (with respect to $G/H$), with $\dpr(Y)=\dpr(X)$. Indeed, if $D$ is an SW-uniformity then this is \cite[Lemma 2.9]{HaHaPeGps} and if $D=\bk$ in the $V$-minimal case then it is \cite[Lemma 4.3]{HaHaPeGps}.  This implies (1) and (2)  for $D$ other than  $K/\CO$ in the $p$-adically closed case.  For (1) in that latter case use \cite[Lemma 3.9]{HaHaPeGps}.

We now assume that $D$ is not $K/\CO$ in the $p$-adically closed case.

(3) If $G$ is a $D$-group then $G/H$ is also locally strongly internal to $D$ by (2). Combined with (the proof of) \cite[Fact 4.25]{HaHaPeGps}  it  follows that $G/H$ is also a $D$-group.

To show that $f(\nu_{D}(G))=\nu_{D}(G/H)$, let $X_0\subseteq G$ be a $D$-set. By the above, we may find a $D$-critical subset $Y_0\subseteq f(X_0)$. By \cite[Remark 4.18]{HaHaPeGps} there exists a $D$-set $Y\subseteq Y_0\subseteq G/H$. Setting $X=f^{-1}(Y)\subseteq X_0$,  and since $X_0$ is a $D$-set so is $X_0$. We are now in the situation where $X$ and $Y=f(X)$ are both $D$-sets, with respect to $G$ and $G/H$, respectively. Assume everything is defined over some parameters set $A$.

Let $a\in X$ be an $A$-generic in $X$, so $f(a)$ is an $A$-generic in $Y$. It suffices to prove that $f(\nu_X(a))=\nu_X(f(a))$.

For this first note that if $U\sub X$ is a $B$-generic vicinity of $a$, for some $B\supseteq A$, then $f(U)$ is a $B$-generic vicinity of $f(a)$ since $f(a)\in \dcl(Aa)$ and $\dpr(U)=\dpr(f(U))$ as $f$ is finite-to-one.

To show the other direction, let $V$ be a $B$-generic vicinity of $f(a)$ for some $B\supseteq A$, then $f^{-1}(V)$ is a $B$-generic vicinity of $a$ since $a\in\acl(Af(a))$ and $f(f^{-1}(V))=V$ because $f$ is surjective.

(4) Follows directly from (1) and (2),
%
\end{proof}

\subsection{Some basic group theoretic facts in our setting} 

Before the next corollary, we note the following application of Baldwin-Saxl (\cite[Lemma 1.3]{PoiGroups}).
\begin{fact}\label{F: Baldwin saxl}
    Let $G$ be a group definable in a sufficiently saturated NIP structure and $\{H_i:i\in T\}$ a definable family of finite index subgroups of $G$. Then $\bigcap_{i\in T} H_i$ is a definable subgroup of finite index.
\end{fact}
\begin{proof}
    By Baldwin-Saxl, there is a finite bound on the index of finite intersections of the $H_i$.
\end{proof}

\begin{corollary}\label{Baldwin Saxl} 
    Let $G$ be a definable group in a sufficiently saturated NIP structure,  $\{\lambda_t:t\in T\}$ a definable family of group automorphisms of $G$, and $X\sub G$, all definable over a parameter set $A$. Assume that  for every $a\in X$, $C_G(a)$ has finite index in $G$. Then there exists an $A$-definable subgroup $G_1\sub C_G(X)$ of finite index in $G$ that is invariant under $\lambda_t$, for all $t\in T$.
\end{corollary}
\begin{proof}
    By Fact \ref{F: Baldwin saxl},  $C_G(X)$ has finite index in $G$. Applying this fact again to the intersection of the family $\{\lambda_t(C_G(X)):t\in T\}$ gives the desired conclusion.
\end{proof}

We need a couple of group theoretic observations on definable groups in our setting. We note for future reference  that Lemma \ref{L:groups 1} and Corollary \ref{C: semisimple in quotients} below do not require saturation of $\CK$.  
\begin{lemma}\label{L:groups 1}
Let $N$ be a definable group in $\CK$ and $H\trianglelefteq N$ a definable normal subgroup, such that $N/H$ is abelian. 
For $k\in \mathbb N$, let $N^k=\{g^k:g\in N\}$. Then: 
\begin{enumerate}
    \item For every $k\in \mathbb N$, $N^k H$ is a normal subgroup of $N$ and $N/N^kH$ is finite.
    
    \item If $H$ is finite and central in $N$,  and $k=|H|$ then the set $N^k$ is contained in  $Z(N)$ and $Z(N)$ has finite index in $N$.
    
\end{enumerate}
\end{lemma}

\begin{proof} (1) Since $N/H$ is abelian, for every $a,b\in N$, $ab=bah$ for some $h\in H$. Because $H$ is normal, for all $g\in G$ and $h\in H$ there is $h'\in H$ such that $hg=gh'$. It follows that $a^2b^2=(ab)^2h_1$, for $h_1\in H$, and by induction, $a^kb^k=(ab)^k h_0$, for some $h_0\in H$. Thus $N^kH$ is a subgroup, clearly normal in $N$.

The order of every  $g\in N/N^kH$ is at most $k$, thus $N/N^kH$ has bounded exponent.
The group $N/N^kH$ is clearly also definable in $\CK$, and by  \cite[Theorem 7.4, Theorem 7.7 and Theorem 7.11]{HaHaPeGps} a definable group of bounded exponent must be finite. 
Thus, $N/N^kH$ must be finite.

(2) Assume now that $k=|H|$ and $H$ is central. Since $G/H$ is abelian,  for every $g, x\in N$ we have $g^{-1}xg=x h$ for some $h\in H$, and hence, since $H$ is central,  $g^{-1}x^k g=(xh)^k=x^kh^k=x^k$. Thus $N^k\sub Z(N)$. It follows that $N^kH\sub Z(N)$, so by (1), $Z(N)$ has finite index in $N$.
\end{proof}

The proof of the next corollary is simpler when $H$ is central, but we need the more general statement: 
\begin{corollary}\label{C: semisimple in quotients}
 Let $G$ be a definable group in $\CK$ and $H$ a finite normal subgroup of $G$, both defined over a parameter set $A$. Let $\{\lambda_t: t\in T\}$ be a definable family of group automorphisms of $G$ fixing $H$ setwise.

 If for some $B\supseteq A$ the group $G/H$ contains a $B$-definable normal abelian subgroup of dp-rank $k$  invariant under all the $\lambda_t$ then so does  $G$. In particular, if $G$ is definably semisimple, then so is $G/H$.
\end{corollary}
\begin{proof} 
 For simplicity, let us call a set invariant under all the $\lambda_t$ $\Lambda$-invariant.  By Lemma \ref{Baldwin Saxl}, there exists a definable $\Lambda$-invariant $G_1\trianglelefteq G$ of finite index  such that $G_1\sub C_G(H)$. In particular, $G_1\cap H$ is central in $G_1$. We fix such $G_1$.

 Assume that $G/H$ has an infinite $\Lambda$-invariant definable abelian normal subgroup of the form $N/H$ for $N\trianglelefteq G$.  It follows that $N$ is $\Lambda$-invariant. Let  $N_1:=N\cap G_1$, an infinite normal subgroup of $G$ of finite index in $N$ and $H_1:=H\cap N_1$, a central subgroup of $N_1$. The quotient $N_1/H_1$ is isomorphic to $N_1H/H\sub N/H$ so is  abelian. Note that $N_1$ is also $\Lambda$-invariant. 

By Lemma \ref{L:groups 1} (2), $Z(N_1)$ has finite index in $N_1$ and therefore $\dpr(Z(N_1))=\dpr(N_1)=\dpr(N)=\dpr(N/H)$. Because $N_1$ is $\Lambda$-invariant and normal in $G$ so is $Z(N_1)$. Hence, $Z(N_1)$ is a $\Lambda$-invariant definable  normal abelian subgroup of $G$ of the same rank as $N_1/H$. Clearly, if $N/H$ is $B$-definable for some $B\supseteq A$ then so are $N_1$ and $Z(N_1)$. 
\end{proof}

\section{Definable subgroups of $((K/\CO)^n,+)$}
Let $\CK$ be one of our valued fields. The purpose of this section is to describe the definable subgroups of $(K/\CO)^n$. When $\CK$ is either power bounded $T$-convex or $V$-minimal those turn out to be definably isomorphic to a product of balls in $K/\CO$. In this case we can also describe all their definable endomorphisms. When $\CK$ is $p$-adically closed, the existence  of finite subgroups creates obstructions (see Example \ref{E:counter to K/O p-adic}), nonetheless we will show that definable subgroups project injectively onto subgroups of full dp-rank.

\subsection{$\CK$ power-bounded $T$-convex or $V$-minimal}\label{ss:groups in K/O}
We assume that  $\CK$ is either power bounded $T$-convex or $V$-minimal. Recall that for $a\in K\setminus \CO$, $v(a+\CO)$ is well-defined, allowing us to refer to definable balls in $K/\CO$. Below, we use the term {\em trivial ball} to refer to either $K$ (or $K/\CO$) or $\{0\}$.

We start with the following basic observation.

\begin{lemma}\label{L:definable subgroup of K or K/O is a ball}
Every  definable subgroup $G$ of $(K,+)$ is a ball, possibly trivial.  As a result, every definable subgroup of $K/\CO$ is a (possibly trivial) ball.
\end{lemma}
\begin{proof}
    Since $\pi: K\to K/\CO$ is a group homomorphism, and the image of a ball (centered at $0$) under $\pi$ is again a ball, it suffices to show that the claim is true for definable subgroups of $(K,+)$.  So let $G$ be a subgroup of $(K,+)$.
   Since $(K,+)$ is torsion-free, if $G$ is finite it is trivial. So we assume may $G$ is infinite.
    Let $B$ be the union of all sub-balls of $G$ containing $0$. If $B=K$ then $G=K$ and we are done, so assume $B\neq K$. Because $\Gamma$ is definably complete, $B$ is a ball itself, possibly $\{0\}$.   Since every infinite definable subset of $K$ has an interior, and $G$ is a group $B\neq \{0\}$. We will show that $G=B$.

Assume for contradiction that $G\neq B$. In our settings,  $B$ is a divisible group (indeed, the maps $x\mapsto nx$ send $B$ onto itself for all non-zero $n\in \Nn$), and since $(K,+)$ is torsion-free, it must be that $[G:B]=\infty$. This means that $G$ contains infinitely many disjoint maximal balls, cosets of $B$.

Assume that $B$ is a closed ball. By the so-called (Cballs) property introduced in \cite{HaHaPeVF}, which holds in our settings \cite[Proposition 5.6, Lemma 5.10]{HaHaPeVF}, only finitely many translates of $B$ intersect $G$, so $G$ contains only finitely many cosets of $B$, contradiction.

Assume then that $B$ is open. After re-scaling $G$, we may assume that $B=\m$. Again, by (Cballs), $G$ intersects only finitely many closed $0$-balls. Consequently, $\CO\cap G$ is an additive subgroup of $K$ containing infinitely many cosets of $\m$. The image of $\CO\cap G$ is, therefore, an infinite definable subgroup of $(\bk,+)$. However, under our assumptions $\bk$ has no infinite definable proper subgroups, thus  $G\cap \CO=\CO$ contradicting the maximality of the ball $B=\m$. Thus, $G=B$, with the desired conclusion. 
\end{proof}

\begin{example}\label{E:counter to K/O p-adic}
    The lemma above does not hold in the $p$-adically closed case. For example, consider a finite residual extension $K$ of $\mathbb{Q}_p$. Let $H$ be a non-trivial finite proper subgroup of $(\bk_K,+)$, then $G=\{g\in K:\mathrm{res}(g)\in H\}$ is a  subgroup of $K$ that is not a ball.
\end{example}

The following computation should be well known.

\begin{fact}\label{F:prod-balls} 
Let $B_1,B_2\subseteq K$ be  balls (possibly the whole of $K$).
\begin{enumerate} 
\item Every ball containing $1$ but not $0$  is a multiplicative subgroup    of $K^\times$.
\item The point-set product $B_1\cdot B_2$ is also a ball.

\item If $0\notin B_2$ then their point-set quotient $B_1\cdot (B_2)^{-1}$ is also a ball.
\end{enumerate}
\end{fact}
\begin{proof} 
We assume both $B_1$ and $B_2$ are not equal to $K$. The proof can be easily adapted to include this case as well.

(1) Well known.

(2) Let $B_1$ and $B_2$ be balls. It will suffice to show that $cB_1B_2$ is a ball for some $c\neq 0$. So, as we proceed,  we may freely replace $B_i$ with $cB_i$ for any such constant $c$. 

Assume, first, that $0\in B_1$ but $0\notin B_2$, thus $B_1B_2=\bigcup \{B_1b: b\in B_2\}$ is a chain of balls centered at $0$. After multiplying by a suitable element, we may assume that $v(b)=0$ for all $b\in B_2$ and so $B_1b=B_1$ for all $b\in B_2$, which gives $B_1B_2=B_1$. If $0\in B_1\cap B_2$ then after multiplying by suitable elements we may assume that $B_1,B_2\in \{\CO,\m\}$; in any of these cases $B_1B_2$ is obviously a ball.

Assume, now,  that $0\neq B_1\cup B_2$. By multiplying by appropriate elements, we may assume that $1\in B_1\cap B_2$, so both are multiplicative subgroups of $K^\times$. Without loss of generality,  $B_1\sub B_2$. Then $B_2\sub B_1B_2\sub B_2B_2=B_2$.

(3) If $0\notin  B_2$ then after possibly multiplying by an appropriate element, we get that $B_2$ is a multiplicative subgroup of $K^\times$. Thus $B_2^{-1}=B_2$ and (2) applies.
\end{proof}

\begin{lemma}\label{a general lemma}
Let $I,J, H\subseteq K$ be definable subgroups, $I\subseteq H\cap  J$, and let $T:H/I\to K/J$ be a definable homomorphism.
Then there is $d\in K$ such that, $d\cdot I\subseteq J$ and for every $x\in H$,  $T(x+I)=d\cdot (x+I)+J$.
\end{lemma}
\begin{proof}
Since $I,J, H$ are definable subgroups of $K$, they are balls and so are their cosets, and because $T$ is a group homomorphism, the image under $T$ of a coset of $I$ is also a coset of a subgroup, so viewed as a subset of $K$ it is a ball. 
Given $x\in H\setminus I$, let
\[S_x=\{w/z\in K:z\in x+I \wedge w\in T(x+I)\}\] 
As a quotient of two balls  $S_x$ is a ball, too (note that $0\notin x+I$ so  Fact \ref{F:prod-balls} applies).  For $d\in K$, let
\[H_d=I\cup \{x\in H\setminus I :d\in S_x\}.\]

We claim that each $H_d$ is a subgroup of $K$ (and when $I=0$, possibly a singleton). To see this, let $H_d'=\{x\in H\setminus I :d\in S_x\}$; by definition $H_d'\cap I=\emptyset$. It follows directly from the definition of $H_d'$ that if $x_1\in I$ and $x_2\in H_d'$, then $x_1\pm x_2\in H_d'$. So it remains to show that if $x_1,x_2\in H_d'$ then   $x_1-x_2\in H_d$. By assumption, $d\in S_{x_1}\cap S_{x_2}$, so  we can write,  $d=w_1/z_1=w_2/z_2$ with $w_i\in T(x_i+I)$ and $z_i\in x_i+I$. So $d(z_1-z_2)=w_1-w_2$. If $z_1-z_2\in I$ then $x_1+I=x_2+I$ so obviously $x_1-x_2\in H_d$. Otherwise, $d=(w_1-w_2)/(z_1-z_2)$,  $z_1-z_2\in x_1-x_2+I$ and $w_1-w_2\in T(x_1+I)-T(x_2+I)=T(x_1-x_2+I)$.

Hence, by Lemma \ref{L:definable subgroup of K or K/O is a ball}, $H_d$ is a ball around $0$. We use this fact now to show that the family $\{S_x:x\in K\}$ forms a chain of balls with respect to inclusion. Namely, we show that for $x_1,x_2\in H\setminus I$, if $v(x_1)\leq v(x_2)$ then $S_{x_1}\sub S_{x_2}$.  Let $d\in S_{x_1}$. Since $H_d$ is a ball and $v(x_1)\leq v(x_2)$ then $x_1\in H_d$ implies that  $x_2\in H_d$, i.e., $d\in S_{x_2}$.

Since $V$-minimal and power bounded $T$-convex valued fields are $1$-h-minimal (see \cite[Section 6]{hensel-min}) they are  definably spherically complete (\cite[Lemma 2.7.1]{hensel-min}, namely the intersection of a definable chain of non-empty balls is non-empty. 
Thus, $\bigcap\limits_{x\in H\setminus I} S_x\neq \0$,  and we let $d$ be an element in the intersection.

Let $\hat H_d=\{ z\in H: d\cdot z\in T(z+I)\}$. Since $T: H/I \to K/J$ is a homomorphism, $\hat H_d$  is a subgroup of $(K,+)$. By definition $H_d'\subseteq \hat H_d$ and as both $\hat H_d$ and $I$ are balls, either $I\subseteq \hat H_d$ or $\hat H_d\subseteq I$. Since $H_d'\cap I=\emptyset$ necessarily, $I\subseteq \hat H_d$ and thus $H_d\subseteq \hat H_d$. On the other hand, by the choice of $d$, for all $x\in H\setminus I$, $d\in S_x$,  so $H=H_d=\hat H_d$.

Finally, as $I\subseteq \hat H_d$, $d\cdot I\subseteq T(I)=J$. Thus $T(x+I)=d\cdot (x+I)+J$ for any $x\in H$.
\end{proof}

We are now ready to describe all definable subgroups of $K^n$ and the associated homomorphisms.
\begin{lemma} \label{end-groups2}
The following holds for all  $n$:

$(1)_n$ If $H\subseteq K^n$ is a definable subgroup of $K^n$ then there is $g\in \gl_n(\CO)$ such that $g(H)$ is a cartesian product  of balls, possibly trivial. 

$(2)_n$ If $H\subseteq K^n$ and $J\subseteq K$ are definable subgroups and $T:H\to K/J$ is a definable homomorphism then there are elements $\alpha_1,\ldots, \alpha_n\in K$ such that for all $x=(x_1,\ldots, x_n)\in H$, 
\[T(x_1,\ldots,x_n)=\alpha_1x_1+\cdots+\alpha_nx_n+J.\]
\end{lemma}
\begin{proof}
$(1)_1$ By Lemma \ref{L:definable subgroup of K or K/O is a ball}, every definable subgroup of $K$ is a ball, possibly trivial.

$(2)_1$ This is  Lemma \ref{a general lemma} for $I=\{0\}$.

We now proceed with the induction step, assuming $(1)_{n-1}, (2)_{n-1}$ and prove $(1)_{n}$: 

Let $\pi:K^n\to K^{n-1}$  be the projection onto the first $n-1$ coordinates. By $(1)_{n-1}$, we may assume that $\pi(H)=H_1\times \cdots \times H_{n-1}$, for balls $H_i\subseteq K$.   Also, write $\ker(\pi)=H\cap (\{0\}^{n-1}\times K)$ as $\{0\}^{n-1}\times J$, for a definable subgroup $J\subseteq K$. 

 Notice that for every $(a,b), (a,c)\in H\sub K^{n-1}\times K$ we have $b-c\in J$ and hence $H$ can be viewed as the graph of a function $T:\pi(H) \to K/J$, mapping $a$ to $b+J$, i.e.
 \[H=\{(a,b)\in K^n:a\in \pi(H)\, \wedge b\in T(a)\}. \]
 By $(2)_{n-1}$, there are $\alpha_1,\ldots, \alpha_{n-1}\in K$, such that $T(x)=\sum_{i=1}^{n-1}\alpha_ix_i+J$.

Hence, 
\[H=\{(x_1,\ldots, x_n)\in K^n: (x_1,\ldots, x_{n-1})\in \pi(H) \wedge x_n-\sum_{i=1}^{n-1} \alpha_i x_i\in J\}. \]
 
 The groups $J$ and $\alpha_iH_i$, for $i=1,\ldots, n-1$, are subgroups of $(K,+)$, hence they are balls. Thus, for every $i=1,\ldots, n-1$,  either $J\sub \alpha_i H_i$ or $\alpha_i H_i\sub J$.  Note that if $\alpha_{i_0} H_{i_0}\sub J$ for some $i_0$  and $(x_1,\dots, x_{n-1})\in \pi(H)$ then $x_n-\sum\limits_{i\neq i_0} \alpha_i x_i\in J$ iff $x_n-\sum\limits_i \alpha_ix_i\in J$.  So there is no harm assuming that $\alpha_i=0$ whenever $J\supseteq \alpha_i H_i$  and that $J\sub \alpha_i H_i$ whenever $\alpha_i\neq 0$. 
Also, we may assume that for some $i$,  $\alpha_i\neq 0$, for otherwise $H=\pi(H)\times J$, and we are done.

Fix $\alpha_1, \dots, \alpha_{n-1}$ as above. Permuting the coordinates, if needed, we may assume that $v(\alpha_1)\leq v(\alpha_j)$, for all $j=2,\ldots, n-1$. Thus, we can write 
\[H=\{(x_1,\ldots, x_n): (x_1,\ldots, x_{n-1})\in \pi(H)\,\wedge \frac{1}{\alpha_1}x_n-(x_1+\sum_{i=2}^{n-1} \frac{\alpha_i}{\alpha_1} x_i)\in \frac{1}{\alpha_1}J\}.\]

Let $S(x_2,\ldots, ,x_n)=\frac{1}{\alpha_1}x_n-\sum_{i=2}^{n-1} \frac{\alpha_i}{\alpha_1}x_i.$ Then $S:K^{n-1}\to K$ is a linear map defined over $\CO$ and we have,

\begin{equation}\label{eq.1}H=\{(x_1,\ldots, x_n): (x_1,\ldots, x_{n-1})\in \pi(H)\,\wedge \, x_1
-S(x_2,\ldots, x_n) \in \frac{1}{\alpha_1}J\}\end{equation}

Let $\hat \pi(x_1,x_2,\ldots, x_n)=(x_2,\ldots, x_n)$ be the projection onto the last $n-1$ coordinates. 
\begin{claim}\label{claim1}
    For every $\hat x=(x_2,\ldots, x_n)\in \hat \pi (H)$, we have $(S(\hat x),\hat x)\in H$.
\end{claim}
\begin{claimproof}
    Let $\hat x=(x_2,\dots,x_n)\in \hat \pi(H)$ and let $x_1=S(\hat x)$, then clearly $x_1-S(\hat x)=0\in \frac{1}{\alpha}J$, so by (\ref{eq.1}), it is sufficient to see that $(x_1,x_2,\ldots, x_{n-1})\in \pi(H)$. Since $\hat x\in \hat \pi(H)$,  there exists $x_1'$ such that $(x_1',x_2,\ldots, x_n)\in H$. In particular, $x_2\in H_2,\ldots, x_{n-1}\in H_{n-1}$, so for $(x_1,\ldots, x_{n-1})$ to be in $\pi(H)$, we only need to verify that $x_1=S(\hat x)\in H_1$. By assumption, $(x_1',x_2,\ldots, x_{n-1}, x_{n})\in H$, so by (\ref{eq.1}), $x_1'\in H_1$ and 
$x_1'-S(\hat x)\in \frac{1}{\alpha_1} J$, so $S(\hat x)\in \frac{1}{\alpha_1}J+x_1'$. However, we assumed that  $J\sub \alpha_1 H_1$ so $\frac{1}{\alpha_1}J\sub H_1$, and therefore $S(\hat x)\in H_1$, hence $(S(\hat x),\hat x)\in H$. 
\end{claimproof}

We get that 
\[H=\{(x_1,x_2,\ldots, x_n):(x_2,\ldots, x_n)\in \hat \pi(H) \,\wedge\, x_1-S(x_2,\ldots, x_n)\in \frac{1}{\alpha_1}J\}.\]

So $H\cap (K\times  \{0\}^{n-1})=\frac{1}{\alpha_1}J\times \{0\}^{n-1}$ and, in particular, the  map $(x_1,\dots, x_n)\mapsto x_1-S(x_2,\dots, x_n)$ from $H$ to $\frac{1}{\alpha_1}J$ is surjective. 
We now define $F:K^n\to K^n$ by \[F(x_1,x_2,\ldots, x_n)=(x_1-S(x_2,\ldots, x_n),x_2, \ldots x_n).\]

Then $F$ is over $\CO$, and by a direct computation one sees that it has determinant $1$, hence $F\in \gl_n(\CO)$. It follows from the definition of $F$ and the observation above that the restriction $F\restriction H$ is definable, injective and onto $\frac{1}{\alpha_1}J\times \hat \pi(H)$.

By induction, there is $h\in \gl_{n-1}(\CO)$ such that $h(\hat \pi (H))$ is a product of balls. Hence, there is $g\in \gl_n(\CO)$ sending $H$ to a product of balls. This ends the proof of $(1)_n$. 

For $(2)_n$, we start with $T:H\to K/J$. with $H\subseteq K^n$, By $(1)_n$, we may assume that $H=V_1\times \cdots \times V_n$, for definable subgroups $V_i\subseteq K$. Thus, 
\[T(x_1,\ldots, x_n)=T(x_1,0,\ldots, 0)+\cdots +T(0,\ldots, 0,x_n),\] with all elements still in $H$. The result follows from  the case $n=1$.
\end{proof}

\begin{remark}
     Lemma \ref{end-groups2}(1) is inspired by the work of Hrushovski-Haskell-Macpherson on  definable $\CO$-submodules of $K^n$ in algebraically closed valued fields,  \cite[Lemma 2.2.4]{HaHrMac1}. In that work the authors prove that up to an automorphism in $\gl_n(K)$ every definable $\CO$-submodule is a finite cartesian product of $K$, $\CO$, $\bm$ and $\{0\}$. 

    In our setting, if $G\sub K^n$ is a definable subgroup then it is an  $\CO$-submodule (the converse is clearly true), since  $\{d\in \CO:dG\sub G\}$ is a definable subgroup of $(K,+)$ containing $1$, so by Lemma \ref{L:definable subgroup of K or K/O is a ball}, it must be the whole of $\CO$.

    Thus Lemma \ref{end-groups2} (1) can be seen as a strengthening of
\cite[Lemma 2.2.4]{HaHrMac1} even in the ACVF$_{0,0}$ setting.
\end{remark} 

We may now conclude:
\begin{lemma}\label{K/O end-groups2}
Let $H\sub (K/\CO)^n$ be a definable subgroup.

$(1)$ There is a definable automorphism $T$ of $(K/\CO)^n$ such that $T(H)=H_1\times \dots \times H_n$, where each $H_i$ is a, possibly trivial,  ball.  

$(2)$ If $T:H\to K/\CO$ is a  definable homomorphism then there are scalars $d_1,\ldots, d_n\in \CO$ such that for all $x=(x_1+\CO,\ldots, x_n+\CO)\in H$, 
\[T(x_1+\CO,\ldots,x_n+\CO)=d_1x_1+\cdots+d_nx_n+\CO.\]
\end{lemma}
\begin{proof} 
(1) Consider $\hat H\sub K^n$ the preimage of $H$ in $K^n.$ By Lemma  \ref{end-groups2}, there is $g\in \gl_n(\CO)$ such that $g(\hat H) $ is a product of (possibly trivial) balls in $K$. Since $g\in \gl_n(\CO)$, it descends to an automorphism of $(K/\CO)^n$ sending $H$ to a product of balls in $(K/\CO)$ (possibly trivial ones).

For (2), we may assume that $H=V_1\times \cdots \times V_n$ for $V_i\sub K/\CO$ and then 
\[T(x_1+\CO,\ldots, x_n+\CO)=T(x_1+\CO,0,\ldots,0)+\cdots+T(0,\ldots,0,x_n+\CO),\] with each element on the right inside $H$. We apply Lemma \ref{a general lemma} with $I=J=\CO$, so there are $d_1,\ldots, d_n\in \CO$ (because $d_i\CO\subseteq \CO$), such that $T(x_1+\CO,\ldots, x_n+\CO)=d\cdot x_1+\cdots+d_n\cdot x_n+\CO$. 
\end{proof}

Finally, we want:
\begin{lemma} \label{automorphism} 
Let $H\sub (K/\CO)^n$ be a definable group and $T:H\to (K/\CO)^n$
a definable homomorphism. Then $T$ can be extended definably to an endomorphism of $(K/\CO)^n$.

In addition, if $T$ is injective, then we can choose the extension to be an automorphism of $(K/\CO)^n$.
\end{lemma}
\begin{proof}
    For the first part, we may think of $T$ in coordinates and apply Lemma \ref{end-groups2}$(2)_n$ to each coordinate map, obtaining $L\in \mathrm{End}((K/\CO)^n)$ extending $T$.

Assume now that $T$ is injective, and we shall see that so is $L$. By Lemma \ref{K/O end-groups2}$(1)_n$, after composing with a definable automorphism of $(K/\CO)^n$ we may assume that $H=B_1\times\dots\times B_n$,  where each $B_i\sub K/\CO$ is a ball around $0$ (possibly trivial). 

Assume first that, for all $i$, $B_i$ is not the zero ball. If $L$, the extension of $T$ provided above, were not injective then, after permutation of the coordinates, we may assume the projection of $\ker(L)$ into $B_1$ is infinite.  But then, $\ker(L)\cap B_1\times \{0_{n-1}\}$ is nontrivial, contradicting the injectivity of $T$.  

  So without loss of generality, we assume that $H=B_1\times\dots\times B_m\times \{0\}^{n-m}$ and that $B_i$ is non-trivial for $i\leq m$. Since $T$ is injective, $\dpr(T(H))=m=\dpr(H)$ and hence, after a definable automorphism of $(K/\CO)^n$ (the range) we may assume that $T(H)=C_1\times\dots\times C_m\times \{0\}^{n-m}$, where the $C_i\sub K/\CO$ are balls with $r(C_i)<0$ (possibly $C_i=K/\CO$). Setting $H_1=B_1\times\dots \times B_m$ and $H_2=C_1\times\dots\times C_m$, the map $T$ thus induces an injective isomorphism of $H_1$ and $H_2$, that, by what we have already noted, can be extended to a definable automorphism $L_1$ of $(K/\CO)^m$.  

  Now, for $(x,y)\in (K/\CO)^m\times (K/\CO)^{n-m}$, let $S(x,y)=(L_1(x),y)$. This is an extension of $T$ to an automorphism of $(K/\CO)^n$.
\end{proof}
As a corollary, we obtain:

\begin{corollary}\label{C: same valuation}
    Assume that  $f:(K/\CO)^n\to (K/\CO)^n$ is a definable group automorphism.
    Then there is $g\in \gl_n(\CO)$ such that for all $x\in K^n$, $f(x+\CO^n)=gx+\CO^n$. In particular, $T$ preserves the valuation.
\end{corollary}
\begin{proof}
    By Lemma \ref{K/O end-groups2}(2), there exist $L_1,L_2\in M_n(\CO)$ such that for every $x\in K^n$, 
    $$f(x+\CO^n)=L_1(x)+\CO^n,\,\,\,\, f^{-1}(x+\CO^n)=L_2(x)+\CO^n.$$
    
    It follows that for all $x\in K^n$, we have $L_1\circ L_2(x)-x\in \CO^n$. It is easy to see that this forces the $K$-linear map $L_1\circ L_2(x)-x$ to be $0$. Thus, $L_2=L_1^{-1}$ and both belong to $\gl_n(\CO)$.
\end{proof}

 \subsection{$\CK$ $p$-adically closed}\label{ss: p-closed}
In the present subsection, we assume that $\CK$ is $p$-adically closed. As we have already seen, definable subgroups of $K/\CO$ need not be balls, so the analysis of definable subgroups of $(K/\CO)^n$ is more subtle than in the $V$-minimal and the power-bounded $T$-convex settings. Our aim in this section is to prove the result below, a weak version of Lemma \ref{K/O end-groups2}(1) that will suffice for our needs. Recall that balls in $\CK/\CO$ are by definition infinite, and we call $\CK$ a {\em trivial ball}.

	\begin{proposition}\label{P:local homeo K/O}		
    For any infinite definable subgroup $H\leq (K/\CO)^n$ 
		there exist $k\in \mathbb N$ and a coordinate projection $\pi_0:(K/\CO)^n\to (K/\CO)^m$, with $m=\dpr(H)$, such that $\pi_0\rest p^kH$ is injective.
	\end{proposition}
\begin{remark}
   For any natural number $k$, since $H/p^kH$ is an interpretable group in $\CK$ with bounded exponent it must be finite, \cite[Theorem 7.12(4b)]{HaHaPeGps}.
\end{remark}

	Let us fix some notation for the rest of Section \ref{ss: p-closed}. Let $J\supseteq \CO$ be a subgroup of $(K/\CO, +)$ with $J/\CO$ finite and  $\rho:K\to K/\CO$ the quotient map. Let $B_J$ be the smallest closed ball around $0$ containing $J$. 
	
	Recall that since $\CK$ has definable Skolem functions, each (partial) definable function $f: K/\CO\to K/J$ lifts to a (partial) definable function $\widehat f: K\to K$. Namely, $\dom(\widehat f)+\CO=\dom(\widehat f)$ and for every $a\in \dom(\widehat f)$, $\widehat f(a)+J= f(\rho(a))$. In particular, for $a,b\in \dom(\widehat f)$, if  $a-b\in \CO$, then $\widehat f(a)-\widehat f(b)\in J$.

	We break the proof into several lemmas. The first is an adaptation of \cite[Proposition 3.21]{HaHaPeGps}, so we may be terse at times.

	\begin{lemma}\label{L:locally scalar}
		Let $H,J\leq K$ be definable subgroups containing $\CO$ with $J/\CO$ finite and $H/\CO$ a ball in $K/\CO$. Let $\widehat T:H\to K$ be a definable function lifting a definable homomorphism $T:H/\CO\to K/J$. Then  there exists a non-trivial ball $U$ in $ K$, $0\in U\leq H$,  and $c\in B_J$ such that $\widehat T(x)-cx\in J$ for all $x\in U$.
	\end{lemma}
	\begin{proof}
		
		Assume everything is defined over some parameter set  $A$ and let $p$ be a complete type over $A$ which is concentrated on  $H/\CO$ with $\dpr(p)=1$. As in \cite[Section 3.2]{HaHaPeGps}, there exists a unique complete  type $\widehat p$ over $A$ concentrated on $H$ such that  $\rho_*\widehat p=p$.  In particular, for any $a\models \widehat p$ also   $a+\CO\models \widehat p$ .
		
		By generic differentiability,  $\widehat T$ and $\widehat T'$ are both  differentiable on $\widehat p$ (see \cite[Lemma 3.17(1)]{HaHaPeGps}). A similar proof to that of \cite[Lemma 3.17(2)]{HaHaPeGps} gives, for any $b\models \widehat p$, that  $\widehat T'(b)\in B_J$.

		\begin{claim}
			For every $a\models \widehat p$ there exists a $\widehat \CK$-definable ball $B\ni a$ contained in $\widehat p(\widehat \CK)$  of valuative radius $r(B)<\mathbb{Z}$  such that for all $b\in B$, $v(\widehat T''(b))+2r(B)>0$. 
		\end{claim}
		\begin{claimproof}
			The proof mimics \cite[Lemma 3.17(3)]{HaHaPeGps}. Since there is one delicate adjustment towards the end, we give the whole argument. The reader may refer to \cite[Section 3.2]{HaHaPeGps}  for the relevant definitions and notions. 
			
			By saturation of $\CK$ and the definition of $\widehat p$, there exists a ball $B_0\sub \widehat p(\CK)$ around $a$ with $r(B_0)< \mathbb{Z}$ (see \cite[Section 3.2]{HaHaPeGps}) and let $r_0:=r(B_0)$. Note that $B_{>r_0+m}(a)\subseteq \widehat p(\CK)$ for any natural number $m$. 
			
			By \cite[Fact 3.13]{HaHaPeGps} applied to the function $\widehat T'$ there are an $A$-definable finite set  $C$  and $m\in \Nn$  such that  \[\tag{$\dagger$}
			v(\widehat T'(a)-\widehat T'(x))=v((\widehat T''(a))+v(a-x)
			\]
			for all $x$ in any ball $m$-next to $C$ around a, and $v(\widehat T''(x))$ is constant on that ball.  
			By definition,  the ball $B_{> r_0+m}(a)$ is contained in a ball $m$-next to $C$, so after possibly shrinking $B_0$, we may assume that $v(\widehat T''(x))$ is constant on $B_0$ and that $(\dagger)$ holds on $B_0$ (see also \cite[Lemma 3.14]{HaHaPeGps}).

			If $\widehat T''(t)\equiv 0$, the claim holds trivially. Otherwise, by \cite[Fact 3.13]{HaHaPeGps}, $\widehat T'(B_{>r'}(a))$ is an open ball of radius $v(\widehat T''(a))+r'$ for any $r'\geq r_0$. 
			
			As $B_{>r_0}(a) \sub \widehat p(\widehat \CK)$, we have  $\widehat T'(B_{>r_0}(a))\sub B_J$. Since $J/\CO$ is finite, we deduce that  $v(\widehat T''(a))+r_0$ is either positive or a finite negative integer. Either way, for any $r'>\mathbb{Z}$ satisfying that for any $n\in \mathbb{Z}$, $r'-n>r_0$, we get that  $\widehat T'(B_{>r'}(a))$ is an open ball of radius $v(\widehat T''(a))+r'>0$.
			
			So let $r$ be such an element. Since $r/2$ also satisfies the same requirements, we deduce that $\widehat T'(B_{>r/2}(a))$ is an open ball of radius $v(\widehat T''(a))+r/2>v(\widehat T''(a))+r>0$.
			
			We conclude that for any $b\in B:= B_{>r/2}(a)$, $v(\widehat T''(b))+r>0$.
		\end{claimproof}
		
		Now, the proof of \cite[Lemma 3.18]{HaHaPeGps} is applicable word-for-word and we get that  for every $a\models \widehat p$ there is a  ball $B$, $a\in B\sub \widehat p(\CK)$, such that for all $y\in B$,
		\[
		v(\widehat T(y)-\widehat T(a)-\widehat T'(a)(y-a))>0.
		\]
		
		Setting $c:=\widehat T'(a)\in B_J$, we get that for all $y\in B$,  $\widehat T(y)-\widehat T(a)-c(y-a)\in \m \sub J$. 
		
		Let $U=B-a$; it is a subgroup of $H$. Let $x=y-a$ be an element of $U$ (so $y\in B$).  Since $\widehat T $ is a lift of a homomorphism, $\widehat T(x)+J=\widehat T(y)-\widehat T(a)+J=c(y-a)+J=cx+J$.
	\end{proof}

We note that for groups definable in $K/\CO$ injectivity of definable homomorphisms can be detected locally: 
	
	\begin{lemma}\label{L:subgroup intersects ball}\begin{enumerate}
			\item    Let $N\leq (K/\CO)^n$ be a non-trivial definable subgroup and $B\ni 0$ a ball in $(K/\CO)^n$. Then $N\cap B$ is non-trivial.
			\item Let $H\sub (K/\CO)^n$ be a definable group, $f:H\to (K/\CO)^m$
			a definable homomorphism and $B\ni 0$  ball in $(K/\CO)^n$. Then $f$ is injective if and only if $f\rest(B\cap H)$ is injective.
		\end{enumerate}
	\end{lemma}
	\begin{proof} 
		(1) By \cite[Lemmas 3.1(3), 3.10 (1)]{HaHaPeGps}, the ball $B$ contains all torsions points in $(K/\CO)^n$. 
		By  by \cite[Lemma 3.10 (2)]{HaHaPeGps}, 
		$N$ has non-trivial torsion. Thus $N\cap B$ contains a non-trivial torsion point.
		
		(2) Apply (1)  to $N=\ker(f)$.
	\end{proof}
	
The following is the technical core of the proof: 	
	
	\begin{lemma}\label{L:private case}
		Let $J\supseteq \CO$ be a group with $J/\CO$ finite, $T:B\to (K/\CO)/J$ be a group homomorphism and let $H\subseteq (K/\CO)^n$  be a definable subgroup of the form \[\{(h_1,\dots,h_n)\in (K/\CO)^n:(h_1,\dots,h_{n-1})\in N\, \wedge h_n+J= T(h_1,\dots,h_{n-1})\},\] where $N\leq (K/\CO)^{n-1}$ is some subgroup of dp-rank $n-1$.
		
		Then there exists a natural number $k$ 
		such that the projection of $p^{k}H$ on some $n-1$ coordinates is injective.
	\end{lemma}
	\begin{proof}
		Since $\dpr(N)=n-1$, there exists a ball $B\subseteq N$ around $0$. If there exists a coordinate projection $\pi$ and a natural number $k$ for which $\pi\rest p^k (H\cap (B\times K/\CO))$ is injective then as $p^k(H\cap (B\times K/\CO))=p^kH\cap (p^kB\times K/\CO)$ we may apply Lemma \ref{L:subgroup intersects ball} (2) and deduce that it is injective on $p^kH$ as well. Consequently, we may assume that $N=B=H_1\times \dots\times  H_{n-1}$ is a product of balls.
		
				Recall that  $\rho:K\to K/\CO$ is the  quotient map.  	
Since \[T(x_1,\ldots, x_{n-1})=T(x_1,0,\ldots, 0)+\cdots +T(0,\ldots, 0,x_{n-1}),\] and denoting  $\widehat T_i$  for a lift of $T(0,\dots,x_i,\dots 0)$ to a partial map from $K$ to $K$, we obtain
		\[\rho^{-1}(H)=\{(a_1,\dots,a_{n-1},a_n)\in K^n:(a_1,\dots,a_{n-1})\in \rho^{-1}(B)\, \wedge a_n+J= \sum_{i=1}^{n-1} \widehat T_i(a_i)+J\}.\]
		
		Applying Lemma  \ref{L:locally scalar} to the $\widehat T_i$,  for each $i$ we find $c_i\in K$  and sub-balls $H_i'\leq H_i$ such that $\widehat T_i(x)-c_ix\in J$ for elements of $H_i'$. Letting $B'=H_1'\times\dots \times H_{n-1}'$, we may, as above,  replace $B$ by $B'$ and $H$ by $H\cap (B'\times K/\CO)$. So we may assume that
		\[\rho^{-1}(H)=\{(a_1,\dots,a_{n-1},a_n)\in K^n:(a_1,\dots,a_{n-1})\in \rho^{-1}(B)\, \wedge a_n+J= \sum_{i=1}^{n-1} c_i\cdot a_i+J\}.\]
		If $c_i=0$ for all $1\leq i\leq n-1$ then $H$ is equal to a product of $n-1$ balls together with $J/\CO$. If we choose $p^k$ large enough so that $p^kJ \sub \CO$, then $p^kH \sub (K/\CO)^{n-1}\times \{0\},$ so projects injectively into the first $n-1$ coordinates.
		
		We thus assume that  $c_i\neq 0$  for some $i$.
		Setting $c_n=1$, assume, without loss of generality,  that $v(c_1)=\min_{1\leq i\leq n}\{v(c_i)\}$.
		
		\begin{claim}
			$\rho^{-1}(H)$ is equal to
			\[X:=\{(a_1,\dots,a_n)\in K^n:(a_2,\dots,a_n)\in P\, \wedge a_n+J= \sum_{i=1}^{n-1} c_i\cdot a_i+J\},\] where $P$ is the projection of $\widehat E$ on the last $n-1$ coordinates.
		\end{claim}
		\begin{claimproof}
			Obviously $\widehat E$ is contained in $X$, so we show the reverse inclusion.  Let $(a_1,\dots,a_n)\in X$. As $(a_2,\dots,a_n)\in P$, there exists $t$ such that $(t,a_2,\dots,a_n)\in  \widehat E$ so $a_n- c_1t -\sum_{i=2}^{n-1}c_ia_i\in J$. On the other hand $(a_1,\dots,a_n)\in X$ so $a_n-\sum_{i=1}^n c_ia_i\in J$ implying that $c_1t-c_1a_1\in J$.
			So in order to show that $(a_1,\dots,a_n)\in \rho^{-1}(H)$ we only have to verify that if $t\in \rho^{-1}(H_1)$ then also $a_1\in \rho^{-1}(H_1)$. But $a_1-t\in c_1^{-1}J$ which is a finite subgroup of $K/\CO$. As $\rho^{-1}(H_1)$  is a ball it contains all torsion elements (\cite[Fact 3.1, Lemma 3.10]{HaHaPeGps}) so it contains $a_1-t$ as well and the conclusion follows.
		\end{claimproof}
		
		We get 
		\[\rho^{-1}(H) =\{(a_1,\dots,a_n)\in K^n:(a_2,\dots,a_n)\in P\, \wedge a_1-\sum_{i=2}^n e_ia_i\in c_1^{-1}J\},\] for some $e_i\in \CO$.
		
		As $c_1^{-1}J/\CO$ is finite as well, we can find some  $k\in \Nn$ large enough so that $p^k(c_1^{-1}J)\subseteq \CO$. We claim that for any $(h_1,\dots,h_n)\in p^kH$, $h_1$ is uniquely determined by $(h_2,\dots,h_{n})$. We will show that for a tuple in $\rho^{-1}(p^kH)$ the first coordinate is  determined, up to  $\CO$-equivalence, by the last $n-1$ coordinates .
		
		To simplify the notation we give the argument for $n=2$,  the general case is similar. Let $(a,b),(c,d)\in \rho^{-1}(p^kH)$, with  $b-d\in \CO$. We want to prove that $a-c\in \CO$. As $\rho^{-1}(p^kH)=p^k\widehat H+\CO$, we can write $(a,b)=(p^ka'+o_1,p^kb'+o_2)$
		and $(c,d)=(p^kc'+o_3,p^kd'+o_4)$, with $(a',b'),(c',d')\in \rho^{-1}(H)$.
		
		We thus have 
		$a'-e_2(b'+o_2), c'-e_2(d'+o_4) \in c_1^{-1}J$. Since $p^k(c_1^{-1}J)\sub \CO$, we get that  
		\[p^k(a'-c')-p^k(e_2(b'-d'))=p^k(a'-c')-e_2p^k(b'-d')\in \CO.\]  
		By our assumption that $b-d\in \CO$ (and since $p^k(b'-d')+\CO=(b-d)+\CO$), it follows that $p^k(b'-d')\in \CO$. and since $e_2 $ is assumed to be in $\CO$, it follows from the above that $p^k(a'-c')\in \CO$. By our assumptions, $p^k(a'-c')+\CO=(a-c)+\CO$, and therefore $a-c\in \CO$ as claimed.
	\end{proof}

	We can finally prove Proposition \ref{P:local homeo K/O}. 
	
		
	\begin{proof}[Proof of Proposition \ref{P:local homeo K/O}]
	
		We proceed by induction. The case $n=1$ is trivially true (take $k=0$ and $\pi_0=\id$). 
		
		Let $\pi:(K/\CO)^n\to (K/\CO)^{n-1}$  be the projection onto the first $n-1$ coordinates. We may assume that the kernel of this projection is finite: Indeed, let $H^i:=\ker(\pi^i\rest H)$  for $\pi^i$ the projection dropping the $i$-the coordinate. If all $H^i$ were infinite then, since $H\supseteq H^1\times\dots\times H^n$, we would conclude that $\dpr(H)=n$ and there is nothing to prove. Thus, we may assume that one of the $H^i$ is finite, and after permuting coordinates, assume that $i=n$.
		
		Write $\ker(\pi\rest H)$ as $\{0\}^{n-1}\times J$, for a finite subgroup $J\subseteq K/\CO$. Since $\pi\rest H$ is finite-to-one,  $\dpr(H)=\dpr(\pi(H))$. Notice that for every $(a,b), (a,c)\in H\sub (K/\CO)^{n-1}\times K/\CO$ we have $b-c\in J$ and hence $H$ can be viewed as the graph of a function $T:\pi(H) \to (K/\CO)/J$, mapping $a$ to $b+J$, i.e.
		\[H=\{(a,b)\in (K/\CO)^n:a\in \pi(H)\, \wedge b+J= T(a)\}. \]
		
		By the induction hypothesis applied to $\pi(H)\sub (K/\CO)^{n-1}$,  there exists $\ell\in \mathbb N$, and a coordinate projection $\pi_1:(K/\CO)^{n-1}\to (K/\CO)^m$ such that $\pi_1\rest p^\ell \pi(H)$ is injective and $m=\dpr(\pi(H))$. Without loss of generality, assume that $\pi_1$ is the projection onto  the last $m$-coordinates $n-m,\dots, n-1$. Let
		\[H_2=
		\{(a_1,\dots,a_{n-1},a_n)\in (K/\CO)^n:(a_1,\dots,a_{n-1})\in p^\ell \pi(H)\, \wedge a_n+J= T(a_1,\dots,a_{n-1})\}\]
		and note that $p^\ell H\subseteq H_2$.
		
		By assumption, $H_2$ is definably isomorphic via $(\pi_1,\mathrm{id})$ to
		\[H_3=
		\{(a_{n-m},\dots,a_{n-1},a_n)\in (K/\CO)^{m+1}:(a_{n-m},\dots,a_{n-1})\in \pi_1(p^\ell \pi(H))\,\]\[ \wedge a_n+J= S(a_{n-m},\dots,a_{n-1})\},\] 
		for $S=T\circ (\pi_1\rest p^\ell\pi(H))^{-1}$.
		
		Since $\dpr(\pi_1(p^\ell\pi(H)))=m$, we may apply Lemma \ref{L:private case} to $H_3$ and find $r\in\mathbb{N}$ and a coordinate projection $\pi_2:(K/\CO)^{m+1}\to (K/\CO)^m$ (on some $m$ coordinates) such that $\pi_2\rest p^rH_3$ is injective. As $H_2$ is isomorphic to $H_3$ via $(\pi_1,\mathrm{id})$, by composing the coordinate projections, we get that $\pi_0=\pi_2\circ (\pi_1,id)$ is injective on $p^rH_2$. Hence it is also injective on $p^{r\ell}H\subseteq p^rH_2 $.
	\end{proof}

\section{Topology and dimension}
If $D$ is a distinguished sort which is an SW-uniformity, it follows from \cite{HaHaPeGps} (see below for details) that definable $D$-groups inherit a group topology, $\tau_D$, from $\nu_D$. On the other hand, since $\CK$ is geometric, $\CK^{eq}$ inherits a notion of dimension (that turns out to be non-trivial for $K$-groups). In the present section, we first recall the basic properties of the dimension induced from $K$ to $\CK^{eq}$, and then study its relation with the topology $\tau_G$ in $K$-groups. 

\subsection{Geometric dimension and equivalence relations}
A sufficiently saturated (one sorted) structure is \emph{geometric} if $\acl(\cdot)$ satisfies Steinitz Exchange and the quantifier  $\exists^\infty$ can be eliminated. Elimination of $\exists^\infty$, sometimes referred to as \emph{uniform finiteness}, means that in definable families there is a uniform bound on the size of finite sets. 

In \cite{Gagelman}, Gagelman shows that for geometric structures, the dimension associated with the $\acl(\cdot)$-pre-geometry can be extended naturally to imaginary sorts. In the present section, we review this extension of dimension and exploit it to show that in $\CK$ the  $K$-rank and the almost $K$-rank of definable sets coincide (compare with \cite[Corollary 4.37]{JohnTopQp}).

Given a geometric structure $\CM$, we remind Gagelman's extension of $\dim_{\acl}$ to $\CM^{eq}$:  Given a definable equivalence relation $E$ on $M^n$ set, and $A\sub \CM^{eq}$ 

\[\dim^{eq}(a_E/A)=\max\{\dim(b/A)-\dim[a]:b\in [a]\},\] where $\dim:=\dim_{\acl}$,
the $E$-equivalence class of $a$ is $[a]\sub K^n$, $a_E:=a/E\in M^n/E$. 
For $Y\sub X/E$ defined over $A$, we define
\[\dim^{eq}(Y)=\max\{\dim^{eq}(a_E/A):a_E\in Y\}.\]

For a concise summary of the properties of $\dim^{eq}$ we refer to \cite[\S 2]{JohnTopQp}. In the present text we will mostly use additivity of $\dim^{eq}$: For $a,b\in \CM^{eq}$, 
\[
\dim^{eq}(a,b/A)=\dim^{eq}(a/Ab)+\dim^{eq} (b/A).
\]
Note that $\dim^{eq}$ coincides with $\dim_{\acl}$ on definable subsets of $M^n$, and on tuples in $M$, over parameters from $M$.  There is, therefore, no ambiguity in extending the notation $\dim$ (instead of $\dim^{eq}$) to imaginary elements and definable sets. Note, however, that in this notation for a definable set $Y$, $\dim(Y)=0$ does not imply that $Y$ is finite, unless $Y\sub M^n$. E.g., $\dim(K/\CO)=\dim(\Gamma)=0$.

Whenever $\CM$ is in addition dp-minimal, dp-rank coincides with dimension on  definable subsets of $M^n$  (\cite[Theorem 0.3]{Simdp}), a fact that we use without further mention. In our setting, as $\CK$ is a geometric structure, this implies directly from the definitions that $\dim(X)\leq \dpr(X)$ for any definable set $X$ in $\CK^{eq}$.  

Since dimension is preserved under definable finite-to-one functions, and infinite definable subsets of $K^n$ have positive dimension, it follows that if $X$ is locally almost strongly internal to $K$ then $\dim(X)>0$. \\

The above observation allows us to show that, in our setting, the $K$-critical and the almost $K$-critical ranks coincide. 
We start with the following result \cite[Lemma 3.8]{PePiSt}.

\begin{fact}
    \label{lemma1.1}
Let $\CM$ be a geometric structure and let $E$ be a definable equivalence relation on $M^n$. Then there exists a definable $S\sub M^n$ such that for every $x\in S$, $[x]\cap S$ is finite and $\dim(S)=dim(S/E)=\dim(M^n/E)$.
\end{fact}

In the setting where $\CM=\CK$ we can conclude the following:
\begin{corollary}\label{C:every thing is strongly internal to K}
Let $Y$ be a definable set in $\CK$ (so possibly in $\CK^{eq}$). If $Y_0\sub Y$ is almost strongly internal to $K$ then there exists a definable subset $Y'\subseteq Y_0$ with $\dpr(Y')=\dpr(Y_0)$ that is strongly internal to $K$. Moreover, the following are equal
\begin{enumerate}
    \item $\dim(Y)$
    \item The $K$-rank of $Y$
    \item The almost $K$-rank of $Y$.
\end{enumerate}
\end{corollary}
\begin{proof}
We use the fact that, in our setting,  the sort $K$ is a geometric SW-uniformity. The proof relies on the following claim.
\begin{claim}
    For any $Z\subseteq Y$, there exists $Z_0\subseteq Z$ strongly internal to $K$ with $\dpr(Z_0)=\dim(Z)$.
\end{claim}
\begin{claimproof}
    Assume that $Z=X'/E$ for some $X'$. Let $S\subseteq X'$ be a definable set, as provided by Fact \ref{lemma1.1}.  I.e. $\dim(S)=\dim(Z)$ and $S$ intersects every $E$-class in a finite (possibly empty) set. Let $\pi:S\to S/E$ be the finite-to-one projection map; note that $S/E\subseteq Z$ and by \cite[Theorem 0.3(1)]{Simdp}, $\dpr(S/E)=\dpr(S)=\dim(S)=\dim(X'/E)$. 

By \cite[Lemma 2.6(1)]{HaHaPeGps}, as $K$ is an SW-uniformity,  there exists a definable subset $Z_0\subseteq S/E\subseteq Z$ strongly internal to $M$ and satisfying  $\dpr(Z_0)=\dpr(S/E)=\dim(Z)$.
\end{claimproof}
We now apply this claim to prove the statements of the corollary.
First, let $Y_0$ be as in the statement; applying the claim for $Z=Y_0$, we get $Y'\subseteq Y_0$ strongly internal to $K$ with $\dpr(Y')=\dim(Y_0)$. But since $Y_0$ is almost strongly internal to $K$, $\dpr$ and $\dim$ also coincide on $Y_0$ so $\dpr(Y')=\dim(Y_0)=\dpr(Y_0)$.

This result, assures that the $K$-rank and the almost $K$-rank of $Y$ are equal. To conclude, note that, since $\dim(Y)$ is obviously bounded below by the $K$-rank of $Y$, we only need to verify the other inequality. This is immediate by applying the claim to $Z=Y$. 
\end{proof}

\subsection{Topology}\label{section: topology} 
  Let $G$ be a definable group in $\CK$, locally strongly internal to a fixed definable SW-uniformity $D$ (for example  $D=K$). In particular, $D$ admits a definable basis for a topology. 
In this section, we review results from \cite{HaHaPeGps} on how to topologize $G$ using the $D$-topology. For $p$-adically closed fields, this was done using different techniques in \cite{JohnTopQp} for the case $D=K$.

 The group $G$ is automatically a $D$-group by \cite[Fact 4.25(1)]{HaHaPeGps}. By \cite[Proposition 5.8]{HaHaPeGps}, there is a type-definable subgroup $\nu_D:=\nu_D(G)$ of $G$ definably isomorphic to an infinitesimal type-definable group in $D$.  Specifically, given any $D$-critical set $X\subseteq G$, any definable injection $f:X\to D^n$ (for $n=\dpr(X)$)  and  any $c\in X$ generic over all the data we have  (recalling that we identify partial types with collections of definable sets):

\begin{equation}
    \label{eq: nuK}
\nu_D=\{f^{-1}(U)c^{-1}: U\subseteq D^n \text{ definable open containing $f(c)$}\}.\end{equation}

Before proceeding with the description of $\nu_D$ we give the proof of the statement in Remark \ref{R: D-sets}(2),  assuring that such an $X$ can always be found.

\begin{lemma}\label{L: Dsets}
   Let $D$ be an unstable distinguished sort in $\CK$ and  $G$ a $\CK$-definable $D$-group. Then  there exists a $D$-critical subset $X\sub G$ and a definable injection $f: X\to D^m$ for $m=\dpr(X)$. In particular, $X$ is a $D$-set.
\end{lemma}
\begin{proof}
    If $D$ is an SW-uniformity this follows from \cite[Proposition 4.6]{SimWal}, so we may assume that $\CK$ is $p$-adically closed and $D$ is either $\Gamma$ or $K/\CO$. If $D=\Gamma$ this follows from cell-decomposition in Presburger arithmetic (as referred to in the proof of Fact \ref{F:minimal fibers in Gamma}). If $D=K/\CO$ then by \cite[Theorem 7.11(3)]{HaHaPeGps}, there exists a definable subgroup $H\subseteq G$ with $\dpr(H)=n$, the $K/\CO$-rank of $G$, definably isomorphic to a subgroup of $((K/\CO)^r,+)$ for some $r>0$. By Proposition \ref{P:local homeo K/O}, we may assume, replacing $H$ with a subgroup of the same dp-rank that $r=n$.
\end{proof}

We now return to the assumption that $D$ is an SW uniformity. Note that $\nu_D$ is given as a definable collection of sets $\{U_t:t\in T\}$ which forms a filter-base: for every $t_1,t_2\in T$ there is $t_3\in T$ such that  $U_{t_3}\sub U_{t_1}\cap U_{t_2}$.
By \cite[Corollary 5.14]{HaHaPeGps}, $G$ has a definable basis for a topology $\tau_D=\tau_{D}(G)$,  making $G$ a non-discrete Hausdorff topological group. {\bf For the rest of this section, all topological notions in $G$  refer to $\tau_D$}. 

A definable subset $X\subseteq G$ is open in this topology if and only if for all $a\in X$  $a\cdot \nu_D\sub X$. In particular, $\dpr(X)\ge \dpr(\nu_D)$, i.e., the dp-rank of any open definable subset of $G$ is at least the $D$-rank of $G$. Of course, it could be, for example, that $\dpr(G)>\dpr(\nu_D)$, so that definable open subsets need not all have the same dp-rank (but they all have the same $D$-rank). 

The next lemma shows that the topology $G$ inherits from $D$ shares some of its good properties. Toward that end,  recall that the $D$-rank of a set $Z$ is the maximal dp-rank of a definable subset strongly internal to $D$. We let $\fr(X)$, the frontier of $X$, denote   $\cl(X)\setminus X$.

\begin{lemma}\label{frontier}
If $X\subseteq G$ is definable, then the $D$-rank of $\fr(X)$ is strictly smaller than the $D$-rank of $X$.
\end{lemma}
\begin{proof} 
Let $d$ denote  the $D$-rank of $\fr(X)$ and let $X_1\subseteq \fr(X)$ be $D$-critical over $A$.  Fix an $A$-generic $g\in X_1$  and $Y\ni g$ a definable basic open set. In particular, we can choose $Y$ to be strongly internal to $D$.

By definition of $\fr(X)$, it follows that  $\fr(X)\cap Y= \fr(X\cap Y)$. By Lemma \ref{L:intersects largely}, $\dpr(X_1\cap Y)=\dpr(X_1)$. By the properties of SW-uniformities, (\cite[Proposition 4.3, Lemma 2.3]{SimWal}), and since $X\cap Y$ can be identified with a subset of some $D^n$, $\dpr(\fr(X\cap Y))<\dpr(X\cap Y)$. Thus, as $X_1\cap Y\subseteq \fr(X\cap Y)$,
\[d=\dpr(X_1)\leq \dpr(\fr(X\cap Y))<\dpr(X\cap Y).\]

Since $X\cap Y$ is strongly internal to $D$ (as $Y$ was), its dp-rank is at most the $D$-rank of $X$, as needed.
\end{proof}

\begin{lemma}\label{L:groups are closed, open iff full D-rank}
If $H$ is a definable subgroup of $G$ then $H$ is closed in $G$ and
the following are equivalent:
\begin{enumerate}
    \item $H$ is open,
    \item the $D$-ranks of $H$ and $G$ are equal, 
    \item $\nu_D\vdash H$.
\end{enumerate}
\end{lemma}
\begin{proof}
Because $G$ is a topological group, and a basis for the topology is definable, the closure of $H$, call it $H_1$ is also a definable subgroup. Therefore, If $H_1\setminus H\neq \emptyset$ then $H_1$ must contain a coset of $H$ thus the $D$-rank of $H_1\setminus H$ is at least that of $H$ contradicting  Lemma \ref{frontier}. So $H$ is closed in $G$.

Now, assume that the $D$-ranks of $H$ and $G$ are equal. This implies (by definition of $\nu_D$),  that $\nu_D\vdash H$. Since $\nu_D$ is open, and $H$ is a group, this implies that  $H$ is open. Finally, as we have seen, if $H$ is open, then it contains $\nu_D$ as a subgroup, and therefore they have the same $D$-rank (since the $D$-rank of $\nu_D$ is maximal in $G$). 
\end{proof}

\begin{definition} \label{D: centralizer}
For $G$ locally strongly internal to $D$,  we let {\em the centralizer of the type $\nu_D$}, denoted by $C_G(\nu_D)$, be the set of all $g\in G$ such that for some definable $Y$ with $\nu_D\vdash Y$, $g$ commutes with all elements of $Y$.
\end{definition}

Since, as we noted, $\nu_D$ is given as a definable collection of sets $\{U_t:t\in T\}$, it 
follows that $C_G(\nu_D)$ is definable: $g\in C(\nu_D)$ if there exists $t\in T$ such that $g\in C_G(U_t)$. Moreover, by the filter-base property of the family, it is a subgroup of $G$. 

\begin{remark}
 Let us note that,  despite its name, if $\CK\prec \hat \CK$, and $g\in C_G(\nu_D)(\hat \CK)$ then $g$ does not necessarily centralize the set $\nu_D(\hat \CK)$. What we know is that  there exists $t\in T(\hat \CK)$ such that $g$ centralizes $U_t(\hat \CK)$. with possibly $U_t\vdash \nu_D(\CK)$.
\end{remark}

 Recall that definable sets in o-minimal structures can be decomposed into finitely many definably  connected sets (i.e. sets containing no non-trivial definable clopen sets). Thus, the same is true if $X\sub G$ is strongly internal to an o-minimal sort $D$.
The result below will be useful in the sequel.
\begin{lemma}\label{L:def connected}
 Assume that $D$ is one of the o-minimal distinguished sorts.   Assume that $X\subseteq G$ is definable,  strongly internal to $D$ and $e\in X$. If $X$ is  definably connected, then every $g\in C_G(\nu_D)$ centralizes $X$.
\end{lemma}
\begin{proof}
    Let $g\in C_G(\nu_D)$.  By definition $\nu_D\vdash C_G(g)$, so by Lemma \ref{L:groups are closed, open iff full D-rank}, $C_G(g)$ is a clopen subgroup of $G$. Now, $C_G(g)\cap X\neq \0$ (as $e$ is in the intersection), so definable connectedness of $X$ implies $X\subseteq C_G(g)$.
\end{proof}

 For the rest of this section  we focus our attention on the case $D=K$ (so, in particular, it is an SW-uniformity),  and the topology we discuss below is the one coming from $K$. \\

We start with an immediate  corollary of Lemma \ref{L:groups are closed, open iff full D-rank} and Corollary \ref{C:every thing is strongly internal to K}.
\begin{corollary}\label{C: open iff full dim}
    Let $G$ be a definable group and $H$ a definable subgroup. Then $H$ is open in $G$ if and only if $\dim(G)=\dim(H)$.
\end{corollary}

As the distinguished sorts,  $\Gamma$, $\bk$ and $K/\CO$ are $0$-dimensional, we get:

\begin{lemma}\label{l: pure 0dim}
    A definable set $S$ is $K$-pure if and only if every definable $0$-dimensional  $X\sub S$ is finite. 
\end{lemma}
\begin{proof} 
    Assume that  $X\sub S$ is infinite and $0$-dimensional. By Fact \ref{F: Reduction to sorts},
    $X$ (and hence also $S$) is locally almost strongly internal to some distinguished sort $D$. Namely, there is a definable infinite $X_1\sub X$ and a definable finite to one function $f:X_1\to f(X_1)\sub D^n$.  Since $\dim(X_1)\ge \dim (f(X_1))$,  necessarily $\dim (f(X_1))=0$ with $f(X_1)$ infinite. Hence,  $D\neq K$, so $S$ is not $K$-pure. 
    
    For the converse, assume that $S$ is not $K$-pure, witnessed by a definable infinite $X\sub S$ and a definable finite to one function  $f:X\to D^n$ for some $D\neq K$.  Since $\dim(D)=0$ for  $D\neq K$,  it follows that  $\dim(f(X))=0$ and hence $\dim(X)=0$.   So, $X$ is infinite and $0$-dimensional.
\end{proof}

For the sake of completeness, we  note that the $\tau_K$-topology on $G$ is {\em locally Euclidean}, in the following sense: for every $g\in G$ there exists a definable open $U\ni g$,  which is definably homeomorphic to an open subset of $K^{\dim(G)}$. Moreover, it is the unique such  group topology on $G$. 

We prove:
\begin{lemma} The $\tau_K$-topology on $G$ (taken to be discrete if $\dim G=0$) is locally Euclidean and if $\tau$ is any other locally Euclidean group topology on $G$ then $\tau=\tau_K$.

In particular, if $K$ is a $p$-adically closed field, $\tau_K$ equals Johnson's admissible topology from \cite{JohnTopQp}.
\end{lemma}
\begin{proof} A non-discrete locally Euclidean topological group is, by definition, a $K$-group, so (by Corollary \ref{C:every thing is strongly internal to K} )  $\dim(G)>0$ and since discrete  groups are trivially locally Euclidean,  we assume $\dim(G)>0$.  Since the topology is invariant under translations, it is sufficient to find a single $g\in G$ at which the topology is locally Euclidean.  If $ n=\dim(G)>0$ then, by Lemma \ref{C:every thing is strongly internal to K}, there exists a definable $X\sub G$, $\dim(X)=\dim(G)$, such that $X$ is strongly internal to $K$, over some $A$, and $\dim(G)$ is the $K$-rank of $G$. Given $g_1$ generic in $X$ over $A$, it follows from Equation (\ref{eq: nuK})  at the beginning of Section \ref{section: topology}  that there exists a definable $\tau_K$-open set $U$, $g_1\in U\sub X$, which is definably homeomorphic to an open set in $K^n$.

Now, assume that $\tau_1,\tau_2$ are  two locally Euclidean group topologies on $G$. Then for $g\in G$, there are definable $U_1,U_2\ni g$, $U_i$ a  $\tau_i$-open set, and definable $f_i:U_i\to V_i\sub   K^n$, such that each $f_i$ is a homeomorphism between $U_i$ with the  $\tau_i$-topology and open $V_i$ with the  $K^n$-topology.

The map $f_2f_1^{-1}:f_1(U_1\cap U_2)\to V_2$, is a definable injection. However, in SW-uniformities, definable bijections are homeomorphisms at generic points, \cite[Corollary 3.8]{SimWal}. Thus,  there is some $g_1\in U_1\cap U_2$ such that on $\tau_1, \tau_2$ open neighborhood of $g_1$, the two topologies agree. Thus, $\tau_1=\tau_2$.

Since Johnson's admissible topology is locally Euclidean, the two topologies are the same.
\end{proof}

Using the exact same proof as above, one can show that for any distinguished sort $D$ which is an SW-uniformity, if $G$ is locally strongly internal to $D$ then every $g\in G$ has a $\tau_D$-open neighborhood which is definably homeomorphic to an open set in $D^m$, where $m$ is the $D$-rank of $G$.

\section{The infinitesimal group $\nu_D$ and local (differentiable) groups}\label{S:infnit and local}

In Section \ref{ss:infint and vicin} we gave an abstract description of $\nu_D(G)$ for an infinite definable $D$-group  $G$ and  an unstable distinguished sort $D$. In the present section, we collect -- for later use -- more specific information on the construction of $\nu_D(G)$, as $D$ ranges over the various distinguished sorts in the different settings we are interested in. Throughout, we fix an infinite group $G$ definable in $\CK$. 

\subsection{The sort of closed $0$-balls $K/\CO$}\label{ss:nu in K/O}
Let $G$ be an infinite definable $K/\CO$-group. In each of our three settings, there exists a definable  subgroup $H\subseteq G$ definably isomorphic to a subgroup of $((K/\CO)^m,+)$ for some $m>0$, such that $\dpr(H)$ is the  $K/\CO$-rank of $G$ \cite[Theorems 7.4(4), 7.7(4), 7.11(3)]{HaHaPeGps}. By Lemma \ref{end-groups2} (if $\CK$ is $V$-minimal or power-bounded $T$-convex), or by Proposition \ref{P:local homeo K/O} (if $\CK$ is $p$-adically closed), we can, after possibly shrinking $H$  but not its dp-rank, choose $m=\dpr(H)$. 




Recall that the valuation descends to $K/\CO$ and $(K/\CO)^n$, hence, so does the notion of a ball. However, we reserve the term ``ball'' for an infinite set, thus in the $p$-adically closed case we require the valuative radius to be infinitely negative, i.e., smaller than $n$ for all  $n\in  \mathbb{Z}$. 

We may now further assume that $H$ is definably isomorphic to a definable ball (of the same rank) centered at $0$:
\begin{fact}\label{F:interior in K/O}
    For any $A$-definable set $X\subseteq (K/\CO)^n$ with $\dpr(X)=n$ and any $A$-generic $a\in X$, there exists a ball $B\subseteq X$ with $a\in B$.
\end{fact}
\begin{proof}
    If $\CK$ is power-bounded $T$-convex or $V$-minimal then this is \cite[Corollary 2.7]{SimWal}, and if $\CK$ is $p$-adically closed this is \cite[Lemma 3.6]{HaHaPeGps}.
\end{proof}

We can now give, keeping the above notation and assumptions,  a more specific description of the construction of $\nu_{K/\CO}$: 
\begin{lemma}\label{F: nu in K/O}
Let $f:H\to (K/\CO)^n$ be an $A$-definable injective homomorphism,  $\dpr(H)=n$  the $K/\CO$-rank of $G$. Then 

 \[\nu_{K/\CO}=\{f^{-1}(U): U\subseteq (K/\CO)^n \text{ is an open ball in $(K/\CO)^n$ centered at $0$}\}.\]
\end{lemma}
\begin{proof} 

Let  $\nu_1:=\{f^{-1}(U): U\subseteq (K/\CO)^n \text{ is an open ball in $(K/\CO)^n$ centered at $0$}\}$.

By definition,  $\nu_{K/\CO}=\nu_H(c)c^{-1}$ for any $A$-generic $c\in H$. Let $H_1 :=f(H)\le  (K/\CO)^n$.
Since $\dpr(H_1)=n$, by Fact \ref{F:interior in K/O}, we may assume, shrinking $H$ (but not its rank) if needed, that  $H_1$ is a ball in $(K/\CO)^n$. 

We first show that $\nu_{K/\CO}\vdash \nu_1$.
Let $U\subseteq H_1$ be an open ball, $0\in U$. By \cite[Proposition 3.12]{HaHaPeVF} (if $\CK$ is power-bounded $T$-convex or $V$-minimal) or \cite[Proposition 3.8]{HaHaPeGps} (if $\CK$ is $p$-adically closed), there exists a ball $Y\subseteq U+f(c)$, $f(c)\in Y$, definable over some $B\supseteq A$ such that $\dpr(f(c)/B)=n$. Since $H_1$ is a subgroup, we have $Y\subseteq H_1$. Now, as $f$ is a group homomorphism, $f^{-1}(Y-f(c))=f^{-1}(Y)c^{-1}\subseteq f^{-1}(U)$, $c\in f^{-1}(Y)$,  and $\dpr(c/B)=n$. Thus, by the definition of $\nu_{K/\CO}$, we have $\nu_{K/\CO}\vdash f^{-1}(U)$, so $\nu_{K/\CO}\vdash \nu_1$. 

Similarly,  $f(\nu_1)c\vdash \nu_{H_1}(c)$, so we conclude that $\nu_1\vdash \nu_{K/\CO}$.
\end{proof}

\subsection{The valuation group $\Gamma$.}\label{ss:infnit of gamma}
When $\CK$ is either power bounded $T$-convex our $V$-minimal, the valuation group $\Gamma$ is o-minimal, when it is $p$-adically closed, it is a model of Presburger arithmetic.  In order to get a uniform treatment (and formulation of results) we make the following definition:

\begin{definition}\label{D:box}
A subset $B\subseteq \Gamma^n$ is called a $\Gamma$-box (around $a=(a_1,\dots,a_n)$) if it is of the following form:
\begin{enumerate}
    \item (In the non $p$-adic case) $\prod_{i=1}^n (b_i,c_i)$ for some $b_i<a_i<c_i$ in $\Gamma$.
    \item  (In the $p$-adic case)   A cartesian product of $n$-many sets of the form $(b_i,c_i)\cap \{x_i: x_i-a_i\in P_{m_i}\}$ where both intervals $(b_i,a_i)$ and $(a_i,c_i)$ are infinite and $P_{m_i}$ is the predicate for $m_i$-divisibility.
\end{enumerate}
\end{definition}

\begin{fact}\label{F:minimal fibers in Gamma} 
Let $Y\subseteq \Gamma^m$ be a definable set with $\dpr(Y)=n\leq m$. Then there exists a definable $Z\subseteq Y$ with $\dpr(Z)=n$ projecting injectively onto a $\Gamma$-box in $\Gamma^n$. 
\end{fact}
\begin{proof}
    If $\CK$ is power-bounded $T$-convex or $V$-minimal, $\Gamma$ is o-minimal and the result follows by cell-decomposition.

    If $\CK$ is $p$-adically closed, then $\Gamma$ is a model of Presburger arithmetic. It also admits a cell-decomposition  \cite{ClucPresburger} (see also \cite[Fact 2.4]{OnVi} for a more explicit formulation), and thus the result follows from the fact that dimension coincides with dp-rank (\cite[Theorem 0.3]{Simdp}).
\end{proof}

Using Fact \ref{F:minimal fibers in Gamma} and \cite[Lemma 4.2]{HaHaPeGps} repeatedly (as in the proof of Lemma \ref{F: nu in K/O} above) we get the following.

\begin{lemma}\label{L: inf neigh in Gamma}
    Let $G$ be a definable $\Gamma$-group and $g:Y\to \Gamma^n$ be a definable injection with $\dpr(Y)=n$ the $\Gamma$-rank of $G$. Assume everything is defined over some parameter set $A$, and $c\in Y$ is $A$-generic. Then
     
    \[\nu_Y(c)=\{g^{-1}(U): U\subseteq \Gamma^n \text{ a $\Gamma$-box around $g(c)$}\}.\]
\end{lemma}

We can now conclude:

\begin{lemma}\label{L: nu in Gamma}
Let $G$ be a definable $\Gamma$-group. There exists $X\subseteq G$, a $\Gamma$-critical set with $\nu_{\Gamma}\vdash X$, and $f:X\to \Gamma^n$ a definable injection satisfying:
\begin{enumerate}
    \item $f(X)$ is a $\Gamma$-box around $0$,
    \item $f(xy^{\pm 1})=f(x)\pm f(y)$ for any $x,y\in X$ with $xy^{\pm 1}\in X$ and
    \item $\nu_\Gamma=\{f^{-1}(U): U\subseteq \Gamma^n \text{ a $\Gamma$-box around $0$}\}.$
\end{enumerate}
\end{lemma}
\begin{proof}
    By \cite[Theorems 7.4(3), 7.7(3), 7.11(2)]{HaHaPeGps}, $\nu_\Gamma$ is definably isomorphic (as groups) to a type-definable subgroup of $(\Gamma^r,+)$ for some $r>0$, and using Fact \ref{F:minimal fibers in Gamma}, we may further assume that $r=n$.
    As this isomorphism is witnessed by an isomorphism of groups, the result follows by compactness and Lemma \ref{L: inf neigh in Gamma}.
\end{proof}

\subsection{The valued field and the residue field}\label{ss: valued field and residue field}
For this section $D$ is either the valued field $K$ or the residue field $\bk$ when $\CK$ is power bounded $T$-convex. We first describe the infinitesimal group $\nu_D$ and then show how in these situations the type-definable group $\nu_D$  gives rise to a definable, differentiable local group with respect to either $K$ or $\bk$.

\subsubsection{Local differential groups}\label{sss:local differ group}
Let $\CF$ be an expansion of either a real closed field or a valued field with valuation $v$. Let us recall some standard definitions.  We later apply them for when $\CF=D$.

\begin{definition}
Given $U\subseteq \CF^n$ open, a map $f:U\to \CF^m$ is \emph{differentiable} at $x_0\in U$ if there exists a linear map $D_{x_0}f:\CF^n\to \CF^m$ such that:

\noindent In the ordered case:
\[\lim_{x\to x_0}\frac{|f(x)-f(x_0)-(D_{x_0}f)\cdot (x-x_0)|}{|x-x_0|}=0,\]
and in the valued case:
\[\lim_{x\to x_0}\left[ v(f(x)-f(x_0)-(D_{x_0}f)\cdot (x-x_0))-v(x-x_0)\right]=\infty.\]

Also, in the valued setting, $f$ is called {\em strictly differentiable at $x_0$} if there exists a linear map $D_{x_0}f$ which satisfies: for all $\epsilon\in \Gamma$ there exists $\delta\in \Gamma$, such that for all $x_1,x_2\in B_{>\delta}(x_0)$, 
\[v(f(x_1)-f(x_2)-(D_{x_0}f)\cdot(x_1-x_2))-v(x_1-x_2)>\epsilon.\]

 \end{definition}

We are going to work extensively with the notion of {\em a local group}, so we first recall some additional  definitions: 

\begin{definition} \label{D: loc gp}
A \emph{local group with respect to $\CF$} is a tuple $\mathcal{G}=(X,m,\iota,e)$ such that 
\begin{itemize}
   \item $X$ is a topological space and there exists a homeomorphism  $\phi: U\to V$  between an  open neighborhood of $e$ in $X$ and  an open $V\sub \CF^n$, for some $n$. 
    \item  the maps $m:X\times X\dashrightarrow X$ and $\iota:X\dashrightarrow X$ are continuous \textit{partial} functions, with open domains containing $(e,e)$ and $e$, respectively.
     
\end{itemize}
such that the following equalities hold \textbf{whenever both sides of the equations are defined}:
\begin{enumerate}
    \item For any $x\in X$, $x=m(x,e)=m(e,x)$
    \item For any $x\in X$, $e=m(x,\iota(x))=m(\iota(x),x)$.
    \item For all $x,y,z\in X$, $m(x,m(y,z))=m(m(x,y),z)$.
\end{enumerate}

The local group $\mathcal{G}$ is  {\em differentiable}  if  $\phi( m(\phi^{-1}(x), \phi^{-1}(y))$  and $\phi (\iota(\phi^{-1}(x))$  are differentiable. {\em Strictly differentiable} local groups are defined analogously.  

The local group  $\mathcal{G}$ is \emph{definable} in $\CF$, if $X,m,\iota$ and $\varphi$ are definable. 
\vspace{.2cm}

For $G$  a definable group,  a \emph{definable local subgroup with respect to $\CF$} is a local subgroup with respect to $\CF$ whose universe is a definable subset of $G$ and whose multiplication agrees with $G$-multiplication.
\end{definition}

\begin{definition}
    Let $\mathcal{G}=(X,m, e, \iota)$ and $\mathcal{G}'=(X',m', e', \iota')$ be local groups. A \emph{ homomorphism}  of local groups $f:\mathcal{G}\to\mathcal{G}'$  is a continuous function $f: U\to X'$, where $U\sub X$ is an open neighborhood of $e$, such that $f(e)=e'$ and $f(m(x,y))=m(f(x), f(y))$ in a neighborhood of $e$. Such an $f$ is a \emph{local isomorphism} if, in addition, it is a homeomorphism onto its image.     If $\CG, \CG'$ are (strictly) differentiable local groups, then such an $f$ is \emph{(strictly) differentiable} if $\varphi'\circ f\circ \varphi^{-1}$ is (strictly) differentiable. 

    For $G$ a definable group, a local subgroup $\CG$ is called \emph{normal in $G$}  if for every $g\in G$, the map $x\mapsto x^g$ restricts to  a local automorphism of $\CG$. In particular -- in the notation of local subgroups --  for any definable open neighborhood $U\sub X$ of $e$ there exists an open neighborhood $V\subseteq X$ of $e$ such that $x\mapsto x^g$ maps $V$ into $U$.
\end{definition}

Assume further that  every definable function in $\CF$  is (strictly) \emph{generically differentiable}, i.e. for every definable open $U\subseteq \CF^n$, and definable $f:U\to \CF$,  the set of points $x\in U$ such that $f$ is not (strictly) differentiable at $x$ has empty interior. See \cite[Section 4.3]{HaHaPeVF} for more information.

Now, if $\CG, \CG'$ as above are (strictly) differentiable local groups and $f:\CG\to \CG'$ is a definable homomorphism of local groups then $f$  is also (strictly) differentiable. Indeed, since definable functions are generically (strictly) differentiable with respect to $\CF$, the corresponding map $\phi'\circ f\circ \phi$ is $\CF$-(strictly) differentiable at a generic point, and then, using the local group structure, it is (strictly) differentiable  on an open neighborhood of $e$.
\begin{definition}
    
\label{N:Ad}
Let $G$ be a definable group in $\CM$ and let  $\mathcal{G}=(X,\cdot,^{-1})$  be a differentiable normal local  subgroup of $G$ with respect to  $\CF$, witnessed by a map $\varphi:X\to \CF^n$.   The {\em Adjoint map with respect to $\CF$} is the map $\ad^{\CG}_\CF:G\to \gl_n(\CF)$, which assigns to every $g\in G$ the Jacobian matrix of the map $D_e(\varphi\circ \tau_g\circ\varphi^{-1})$.
\end{definition}

By the chain rule in $\CF$, $\ad^\CG_\CF$ is a group homomorphism. 

Note that while the matrix $D_e(\tau_g)$ may depend on the choice of $\phi$ (up to conjugation), the definable group  $\ker(\ad^\CG_\CF)$ does not.

\subsubsection{The infinitesimal group}
Under the assumptions of this section, the sort $D$ is an SW-uniformity expanding a field. Therefore, if $X\sub D^k$ is definable, $f: X\to D^m$ is a definable injection, then by possibly shrinking $X$, but not its rank, we may compose $f$ with a projection $\pi: X\to D^{\dpr(X)}$ such that $ \pi\circ f (X)$  is a basic open set. 
 
 Furthermore, every definable function in $D$ is generically differentiable with respect to $D$ in the o-minimal case and generically strictly differentiable in the valued case. Indeed, if $D=\bk$ in the power bounded $T$-convex case, then $\bk$ is a o-minimal so every definable function is generically differentiable. In the other cases, it follows from 1-$h$-minimality (\cite[Proposition 3.12]{AcHa}).

\begin{fact}\label{F: nu in K}
Let $G$ be a definable $D$-group, locally strongly internal to $D$ over $A$, witnessed by the definable injection $f:X\to D^n$, with $\dpr(X)=n$, the $D$-rank of $G$. Given $c\in X$, generic over $A$, 
\[\nu_D(G)=\{f^{-1}(U)c^{-1}:U\subseteq D^n \text{ open  containing } f(c)\}.\]
\end{fact}
\begin{proof}
 By \cite[Proposition 5.6]{HaHaPeGps}, for $c\in X$ $A$-generic $\nu_X(c)=f^{-1}(\mu(f(c))$, where $\mu(f(c))$ is the infinitesimal neighborhood of $f(c)$ in the topology on $D$. The result now follows. 
\end{proof}

\begin{lemma}\label{L:existence of local group}
    Let $G$ be a definable $D$-group locally strongly internal to $D$.
    
    Then there exists a definable  differentiable local normal subgroup $\CG=(X,\cdot, ^{-1},e)$ with respect to the field $D$, with $\nu_D(G)\vdash X$. When $D=K$ the local group is strictly differentiable.

    If $G$ is definable over some $\CK_0\prec \CK$ then the local group and the map $\varphi:X\to D^n$ witnessing it can be found definable over $\CK_0$.
\end{lemma}
\begin{proof}
     Let $\nu_D=\nu_D(G)$. By Fact \ref{F: nu in K},  $\nu_D\vdash X$, for some definable $\nu_D$-open set $X\sub G$, and there exists a definable injection $\varphi:  X\to D^n$, with $\varphi(X)$ a definable open subset of $D^n$ and $n$ the $D$-rank of $G$ (indeed, in the notation of the above Fact, replace $Xc^{-1}$ by $X$).
     
     Let $\widehat \CK\succ \CK$ be a $|\CK|^+$ saturated elementary extension. By \cite[Theorem 7.4(1,2), Theorem 7.7(1), Theorem 7.11(1)]{HaHaPeGps}, $\nu_D(\widehat \CK)$ is a (differentiable) Lie group with respect to the structure induced  by $\varphi$. Furthermore, since every definable function in the valued field case is generically strictly differentiable, a similar proof shows in this case that $\nu_D(\widehat \CK)$ is a strictly differentiable Lie group. Furthermore,  $g\nu_D g^{-1}=\nu_D$ for any $g\in G(\CK)$ (Fact \ref{F: properties of nu}).

     Using  compactness,  we can endow $X$ with the structure of a (strictly) differentiable local normal subgroup of $G$ with respect to the field $D$.

     Lastly, if $G$ is definable over $\CK_0$ then since the existence of $X$ and $\varphi$ with the desired properties is first order, such can be found over $\CK_0$ as well.
\end{proof}

Combining the last lemma with Definition \ref{N:Ad} we can find a definable group representation $\ad_D^\CG: G\to \gl_n(D)$,  for $n$ the $D$-rank of $G$. As noted after Definition \ref{N:Ad},  the map $\ad_D^\CG$ depends on $\CG$ (i.e. on $X$ and $\varphi$), only up to a change of coordinates. In particular, the group $\ker(\ad^\CG_D)$  does not depend on the choice of $\CG$ and the image $\ad_D^\CG(G)$ is independent of  $\CG$, up to conjugation.  As for the latter, since we do not care about the particular embedding in $\gl_n(D)$,  the choice of $\CG$ is unimportant, and \textbf{we will write, from now on, $\ad_D(G)$ without specifying any choice of local subgroup $\CG$.}    

For future reference we single out the following corollary of Lemma \ref{L:existence of local group} and the above discussion: 
\begin{remark}\label{R: H1H2 over the same parameters}
    Given a $D$-group $G$ defined over a model $\CK_0$ the subgroup $\ker(\ad_D(G))$ is definable over $\CK_0$. 
\end{remark}

\section{Groups locally strongly internal to  $\Gamma$}

As above, $\CK$ denotes a saturated model of one of our valued fields, $\Gamma$ its valued group. Since  $\Gamma^n$ and $(K/\CO)^n$ are commutative, so are their (local) subgroups. In the present and the next section, we show that this is reflected in a strong sense in definable $\Gamma$-groups or $K/\CO$-groups. For $\Gamma$-groups, we get a clean result:  definable $\Gamma$-groups contain infinite definable normal abelian subgroups. We prove (keeping the notation and conventions of the previous sections): 
\begin{proposition}\label{P: Gamma}
Assume that $G$ is a definable group  locally strongly internal to $\Gamma$. Then $G$ contains a  definable normal subgroup $G_1$ of finite index, defined over the same parameters as $G$, such that  $\nu_\Gamma\vdash Z(G_1)$. In particular, $G$ contains a definable (over the same parameter set) infinite normal abelian subgroup.
\end{proposition}

The proof splits between the $p$-adic case (where $\Gamma$ is discrete) and the remaining cases (where $\Gamma$ is dense and o-minimal). 

\subsection{$\CK$ $p$-adically closed}
We assume that $\CK$ is $p$-adically closed and thus $\Gamma$ is a model of Presburger arithmetic. Let $\mathbb{Z}$ be a prime (and minimal) model for $\Gamma$. We denote by $\mathbb{Z}_{Pres}$ the structure $(\mathbb{Z},+,<)$. \\

Before proceeding to the proof of  Proposition \ref{P: Gamma} in this setting, we need some preparatory results: 

\begin{lemma}\label{F: def over Z}
For any definable family, $\{X_t\}_{t\in T}$, of subsets of $\Gamma^n$ the family $\{X_t\cap \Zz^n\}_{t\in T}$  is  definable in $\Zz_{Pres}$.
\end{lemma}
\begin{proof}
Because $\CK$ is $p$-adically closed, $\Gamma$ is stably embedded, so we may assume that $T\sub \Gamma^k$ for some $k$. Since in  Presburger arithmetic types over $\mathbb{Z}$ are (uniformly) definable, the family $\{X_t\cap \Zz^n:t\in T\}$ is definable in $\Zz_{Pres}$. See  \cite[Theorem 0.7]{CoVo} (and also \cite{delon-def}).
\end{proof}

\begin{lemma}\label{L:bound on index}
    Let $\{X_t:t\in T\}$ be a definable family of subsets of $\Gamma^n$ and assume that for all $t\in T$, $X_t\cap \mathbb{Z}^n$ contains a subgroup of $\mathbb{Z}^n$ of finite index. Then there is a uniform upper bound on $l(t)$,  the minimal $l\in \Nn$ such that $X_t\cap \Zz^n$ contains a subgroup $\Zz^n$ of index $l$.  
\end{lemma}
\begin{proof}
    Assume towards a contradiction  that there is no bound on $l(t)$ for $t\in T$.  So the following type is consistent:
    \[\rho(t):=\{D\not\subseteq X_t: D\subseteq \mathbb{Z}^n \text{ finite, generating a definable subgroup of finite index}\},\] contradicting the assumption.
\end{proof}

\begin{lemma}\label{F: full dp-rank, presburger}
\begin{enumerate}
    \item Let $Y\subseteq \Gamma^n$ be a definable set. If $Y\cap \mathbb{Z}^n$ contains a subgroup of $\mathbb{Z}^n$ of finite index, then $\dpr(Y)=n$.
    \item Every finite index subgroup $H\leq \Gamma^n$ is definable. 
\end{enumerate}
\end{lemma}
\begin{proof}
By Fact \ref{F: def over Z}, $Y\cap \mathbb{Z}^n$ is definable in $\Zz_{Pres}$, as a subset of $\mathbb{Z}^n$. Since it contains a finite index subgroup, it has dp-rank $n$. Thus, we have $\mathbb Z_{pres}\prec \Gamma$ and $\dpr(Y\cap \mathbb Z^n)=n$.  It follows by \cite[Lemma 3.10]{HaHaPeVF} that $\dpr(Y)=n$. For Clause (2) let $H\le G$ be a definable subgroup of finite index, and note that since $H$ has finite index, there is $k\in \mathbb N$ such that the map $x\mapsto kx$ sends $\Gamma^n$ into $H$.  Because $k\Gamma^n$ has finite index in $\Gamma^n$,  it follows that $H$ is a union of finitely many cosets of $k\Gamma^n$, $H$ is  definable.
\end{proof}

Recall Definition \ref{D:box} of a $\Gamma$-box.

\begin{lemma}\label{L: family of functions, presburger}
Let $Y\sub \Gamma^n$ be a definable set such that $Y\cap \Zz^n$ contains a subgroup $H$ of $\Zz^n$ of  finite index. Assume that $\{f_t\}_{t\in T}$ is a definable family of definable functions $f_t:Y \to Y$ such that for all $a,b\in Y$ with $a+b\in Y$,  we have $f_t(a+b)=f_t(a)+f_t(b)$. Then: 

\begin{enumerate}
    \item For every $t\in T$, $f_t(H)\subseteq \mathbb{Z}^n$.
    \item The family $\{f_t\restriction H:t\in T\}$ is uniformly definable in $\Zz_{Pres}$ and therefore finite.
\end{enumerate}
\end{lemma}
\begin{proof}
Assume everything is definable over some parameter set $A$. By stable embeddedness of $\Gamma$, the family $\{f_t: t\in T\}$ is uniformly definable in $\Gamma$ so we may assume that $T\sub \Gamma^k$. Since $H$ is a subgroup of finite index of $\mathbb{Z}^n$ it is generated by some finite set $\{m_1,\dots,m_s\}\sub \mathbb Z^n$.

(1) Fix some $t\in T$.  It suffices to prove that each  coordinate function of $f_t$ sends $H$ into $\mathbb{Z}$. So we may assume $f_t:Y\to \Gamma$.  Let $c\in Y$ be $A$-generic in $Y$.

Since $\dpr(Y)=n$, it follows from  cell decomposition,  \cite[Theorem 1]{ClucPresburger}, and  \cite[Lemma 3.4] {OnVi} that there is an $A$-definable $n$-dimensional  $\Gamma$-box, $B=\prod_i J_i\sub Y$, centered at $c=(c_1,\ldots, c_n)\in B$,  such that 
\[
(f_t\restriction B)(x)=\sum_i s_i\left( \frac{x-t_i}{k_i}\right) +\gamma,
\] 
with  $\gamma\in \Gamma^n$, $s_i,t_i,k_i\in \mathbb N$  and $J_i=I_i\cap \{x-t_i\in P_{k_i}\}$, for some infinite interval $I_i$.

By shrinking $B$, if needed (over the same parameters), we may assume that $B$ is a product of boxes of the form $I_i\cap P_k(x_i-t_i)$ (i.e., that $k_i=k$ for all $i$). 

Note that for every $\bar r\in\mathbb Z^n$, we have  by the above description of $f_t$, that 
$f_t(c+k\bar r)-f_t(c)\in \mathbb{Z}$.  In particular, if  $m_i$, $1\leq i\leq s$, is any of the generators of $H$ we fixed earlier then we have $c, c+km_i$ and $km_i$ all in $Y$, so by the additivity assumptions,  
$$f_t(km_i)=f_t(c+k\bar r)-f_t(c)\in \mathbb Z.$$

However, since $f_t(km_i)=kf_t(m_i)$ this implies that $f_t( m_i)\in \mathbb{Z}$ and, as this is true of a set of generators of $H$, we see that $f_t(H)\subseteq \mathbb{Z}$, as claimed.

(2) The first part of the claim  is a consequence of Fact \ref{F: def over Z} using Lemma \ref{F: full dp-rank, presburger}. The second part follows from quantifier elimination in Presburger arithmetic, by noting that  any definable family of group homomorphisms is finite (see also \cite[Fact 2.10]{OnVi}).
\end{proof}

We can now give the proof of Proposition \ref{P: Gamma} in $p$-adic case. 
\proof[Proof of Proposition \ref{P: Gamma} in the $p$-adic case]
We assume that $G$ is locally strongly internal to $\Gamma$. By Lemma \ref{L: nu in Gamma} there are a definable  $X\subseteq G$, with $\nu_{\Gamma}\vdash X$, and a definable function, $f:X\to \Gamma^n$, with $\dpr(X)=n$ for $n$ the $\Gamma$-rank of $G$. For simplicity of notation, identify $X$ with its image in $\Gamma^n$ and $e_G$ with $0_{\Gamma^n}$. We may further assume that, restricted to $X$, $G$-multiplication coincides with addition and the same for the inverse.
By Lemma  \ref{L: nu in Gamma}, we may further assume that $\nu_{\Gamma}$ is the intersection of  $\Gamma$-boxes around $0$. We fix one such $\Gamma$-box $B\subseteq X\subseteq \Gamma^n$, $\nu_\Gamma\vdash B$. 
   
   By \cite[Proposition 5.8]{HaHaPeGps}, $g\nu_{\Gamma} g^{-1}=\nu_{\Gamma} $ for every $g\in G$ and thus $\nu_{\Gamma}\vdash B^g\cap B$. By compactness, for every $g\in G$, there exists a  $\Gamma$-box $B_0\subseteq B\cap B^g$ around $0$. By Lemma \ref{L: family of functions, presburger}(1),   $B\cap \mathbb{Z}^n$ is a subgroup of $\mathbb{Z}^n$ of finite index (though $B^g$ need not be  contained in $\Gamma^n$).
    
    By Lemma \ref{L:bound on index} there is some natural number $k$ such that for any $g\in G$, $B^g\cap B$ contains a box $B_g$ with $B_g\cap \mathbb{Z}^n$ a subgroup of index at most $k$ in $\Zz^n$. Consequently, there exists some subgroup $H\subseteq \mathbb{Z}^n$ of finite index such that $H\subseteq B\cap B^g\cap \mathbb{Z}^n$ for all $g$.

    Let $Y=\bigcap\limits_{g\in G}B^g$. It is a definable set, contained in $B\subseteq \Gamma^n$, invariant under conjugation by all elements of $G$ and containing $H$.
     Let $Y_0:=Y\cap \mathbb{Z}^n$  (note that $H\subseteq Y_0$) and for every $g\in G$ let $\tau_g:Y\to Y$  denote the restriction of conjugation by $g$ to $Y$. By Lemma \ref{L: family of functions, presburger}(1), $\tau_g(H)\subseteq \mathbb{Z}^n$. By Lemma \ref{L: family of functions, presburger}(2), $\{\tau_g\restriction H\}_{g\in G}$ is a  family of group homomorphisms uniformly definable in $\mathbb{Z}$, so it is finite. We may now replace $H$ by the (finite) intersection of all $\tau_g(H)$, and obtain another subgroup of $\Zz^n$ of finite index. Thus, we may assume that $H$ is invariant under all $\tau_g$.

    Let $R$ be a finite set of generators for $H$ and let $E(g,h)$ be the definable equivalence relation on $G$ given by $d^g=d^h$ for all $d\in R$. Since addition on $H$ coincides with $\Gamma$-multiplication, and  for all $g,h\in G$ both $\tau^g\restriction H$  and $\tau^h\restriction H$  are homomorphisms preserving $H$, it follows that  $E(g,h)$ holds if and only if $\tau_g\restriction H=\tau_h\restriction H$. The  definable quotient $G/E$ can be identified with a finite subgroup of $\aut(H)$, and the map $\sigma:G \to G/E$ is a definable group homomorphism. Its kernel, call it $G_1$, is a definable normal subgroup of $G$ of finite index, that --  by definition -- centralizes $H$, hence $H\sub Z(G_1)$. We claim that $\nu_{\Gamma}\vdash Z(G_1)$.
    
    By Lemma \ref{F: full dp-rank, presburger}(2), $H$ is definable in $\Zz_{Pres}$ and  $Z(G_1)$ contains all finite boxes of the form $[-a,a]^n\cap H$, for $a\in \mathbb N$. Since $H$ is definable, $Z(G_1)$ must contain a set of the form $I^n\cap H(\CK)$, for an infinite interval $I\sub \Gamma$, so in particular, it contains a  $\Gamma$-box. It follows that $\nu_{\Gamma}\vdash Z(G_1)$ and therefore $Z(G_1)$ is a definable infinite normal subgroup of $G$. \qed$_{(\Gamma\text{ Presburger})}$\\

    We postpone the proof that $G_1$ can be taken to be definable over the same parameters as $G$ to the next section (since the proof is similar). 

\subsection{$\CK$ is power bounded $T$-convex or $V$-minimal}
We now assume that  $\CK$ is either power bounded $T$-convex or $V$-minimal, so that $\Gamma$ is an (o-minimal) ordered vector space. Recall  Definition \ref{D:box} of a $\Gamma$-box.

\proof[Proof of Proposition \ref{P: Gamma} for o-minimal $\Gamma$]
  By the description of $\nu_\Gamma$ (Lemma \ref{L: nu in Gamma}), there exists a definable subset $X\subseteq G$, with $\nu\vdash X$, definably isomorphic to a $\Gamma$-box (around $0$) in $\Gamma^n$. 
  Identifying $X$ with its image, we assume (by compactness) that for every $x,y\in X$ with $xy^{\pm 1} \in X$ we have  $xy^{\pm 1}=x\pm y$. 
  
  Because $\Gamma$ is o-minimal, and $X$ is identified with a $\Gamma$-box in $\Gamma^n$,  there is a definable neighborhood base, $\{W_t: t\in T\}$, of $0$ in $X$. 
  
  For every $g\in G$, let $\tau_g$ denote the map $x\mapsto x^g$, and for $g,h\in G$ write $g\sim h$ if $\tau_g$ and $\tau_h$ have the same germ at $0$, namely there exists an open neighborhood $U\sub \Gamma^n$ of $0$, such that $\tau_g|U=\tau_h|U$. By the above, this is a definable equivalence relation. 
Let $\sigma$ be the definable function mapping $g\in G$ to $[g]_\sim$. It is a homomorphism of groups, with the group operation on the set of equivalence classes given by composition of germs.

  We know that for every $g\in G$, $\nu^g=\nu$ (as types), thus there is some $W_t\subseteq X$ such that $W_t^g\subseteq X$ is also a neighborhood of $0$. So $\sigma(G)$ can be viewed as a definable family of definable germs on $X$.
Since $\Gamma$ is a pure ordered vector space over a field $F$ (the field of exponents in the o-minimal $T$), it follows  that  $\sigma(G)$ is finite. Indeed, by \cite[\S1.7 Corollary 7.6]{vdDries}, each germ is the restriction of some $T\in \gl_n(F)$ to an open neighborhood of $0$. Since each such $T$ is $\emptyset$-definable, a definable family of such germs must be finite.

Hence, the  definable group $G_1:=\ker(\sigma)$ has finite index in $G$.

  By definition, for every $g\in G_1$ there exists a $\tau_\Gamma$-open neighborhood of $0$, on which $x^g=x$.
Thus, $G_1\sub C_G(\nu_\Gamma)$ (recall Definition \ref{D: centralizer}).  Since $X\sub \Gamma^n$ is a $\Gamma$-box, it is definably connected, so we may apply Lemma \ref{L:def connected} and conclude that $X\sub C_G(\nu_\Gamma)$

By Lemma \ref{L:nu of subgroup}, $\nu_\Gamma\vdash G_1$. Thus, $\nu_\Gamma\vdash X\cap G_1\sub Z(G_1)$, as claimed.
 Since $G_1$ is normal in $G$ it follows that $Z(G_1)$ is a definable infinite abelian normal subgroup of $G$. \\

Finally, let us verify that in both the current case, and in the $p$-adically closed case, we can replace $G_1$ with a subgroup defined over the same parameters as $G$. Without loss of generality, assume that $G$ is $\0$-definable and let $\{G_s:s\in S\}$ be a $\0$-definable family of normal subgroups of $G$ whose index in $G$ is $[G:G_1]$, and such that $G_1=G_s$ for some $s\in S$. We may further assume that for each $s\in S$, $Z(G_s)$ has a definable subset which is in definable bijection with a $\Gamma$-box (in $\Gamma^n$) around $0$. By Lemma \ref{L: inf neigh in Gamma}, $\nu_\Gamma \vdash Z(G_s)$. 
By Fact \ref{F: Baldwin saxl}, $\bigcap_s G_s$ has finite index in $G$. It is $\0$-definable, and its center contains $\nu_\Gamma$. 

We have thus finished the proof of Proposition \ref{P: Gamma} in all cases. 
\label{section: Gamma}

\section{Groups locally strongly internal to $K/\CO$.}
We still assume $\CK$ is a saturated model in one of our cases. 
In the present section, we extend the results of the previous section from $\Gamma$-groups to $K/\CO$-groups. The result we get is somewhat weaker. Explicitly, we prove:

\begin{proposition}\label{P: K/O}
  Let $\CK_0\prec \CK$ be an elementary substructure, $G$ a $\CK_0$-definable  $K/\CO$-group not locally strongly internal to $\bk$. Let $\CA=\{\lambda_s:s\in S\}$ be a $\CK_0$-definable family of automorphisms of $G$, fixing the partial type $\nu_{K/\CO}$. Then there is a $K_0$-definable normal abelian subgroup $N\leq G$ which is stabilized under all of the $\lambda_g$ such that $\nu_{K/\CO}\vdash N$. In particular, $\dpr(N)$ is at least the $K/\CO$-rank of $G$. 
\end{proposition}

\begin{remark}
    For convenience of presentation, we chose in Proposition \ref{P: K/O} a uniform statement for all cases. However, in fact, the results are slightly stronger in each case. For $p$-adically closed fields, the assumption that $G$ is not locally strongly internal to $\bk$ is vacuous, while in the remaining cases we obtain a group invariant under all definable automorphisms of $G$ (without the need to fix a family in advance).
\end{remark}

We say that a subgroup $H\leq G$ is $\CA$-invariant if for every $s\in S$, $\lambda_s(H)=H$. Since the proposition does not make any assumptions on $\CA$ we may assume that $\CA$ contains the family of all conjugations by elements of $G$, and thus an $\CA$-invariant subgroup will be in particular normal in $G$. 

As in Section \ref{section: Gamma} , the proof splits between the $p$-adically closed case and the remaining cases.

\subsection{$\CK$ is $p$-adically closed}
Since $\CK$ is $P$-minimal and saturated, there is a finite extension, $\mathbb F$ of $\qp$ embedding elementarily (as a valued field) into $K$. We identify the image of some fixed such embedding with a valued subfield of $K$. 

Since the value group $\Gamma_{\mathbb{F}}$ is isomorphic to $\mathbb{Z}$, as ordered abelian groups,  we identify $\Gamma_{\mathbb{F}}$ with $\mathbb{Z}$ and view it as a prime (and minimal) model for $\Gamma$. We denote $\Zz_{Pres}$ the structure $(\Zz, +, <)$.

The following fact is an easy consequence of the results of \cite{HaHaPeGps}:

\begin{fact}\label{F: torsion}
Let $\CK_0\equiv \CK$, $\CK_0$ not necessarily saturated, with $\CO_0$ its valuation ring. Let $\mathrm{Tor}(K_0/\CO_0)$ denote the torsion subgroup. Then
\begin{enumerate}
    \item  $\mathrm{Tor}(K_0/\CO_0)=\{a\in K_0/\CO_0:v(a)\in \mathbb{Z}\}$.
    \item $\mathrm{Tor}(K_0/\CO_0)$ is a finite direct sum of Pr\"ufer $p$-groups and is isomorphic to $\mathbb{F}/\CO_{\mathbb{F}}$. In particular, $\mathrm{Tor}(K_0/\CO_0)$ is a $p$-group. 
    \item Every ball in $(K_0/\CO_0)^n$ centered at $0$ contains $\mathrm{Tor}(K_0/\CO_0)^n$ and the $p^k$-torsion points are 
    exactly the points $b\in (K/\CO)^n$ with $v(b)\geq -k$.

\end{enumerate}
\end{fact}
\begin{proof}
Since, by the basic properties of Pr\"ufer groups the $p^n$-torsion is finite for all $n$, it will suffice to prove the claim in $\CK$: 

(1): If  $v(a)=n\in \mathbb Z_{<0}$ then $p^n a\in \CO$, so $a+\CO\in \mathrm{Tor}(\CK/\CO)$. The reverse inclusion follows from \cite[Lemma 3.1]{HaHaPeGps}(3).

 (2):
By \cite[Lemma 3.1]{HaHaPeGps}(3), every torsion element of $(K/\CO)^n$ is in $(\mathbb F/\CO_{\mathbb F})^n$, and with the previous clause (2) follows for $\CK$ since $\mathbb F/\CO_\mathbb F$ is isomorphic to a  of Pr\"ufer $p$-groups.

(3) follows from the structure of the Pr\"ufer group. \\

\end{proof}

\begin{lemma}\label{L:full subgroups have the same torsion}
Let $G$ be a definable $K/\CO$-group. Let $H_1, H_2\leq G$ be definable subgroups, and  $f_i:H_i\to (K/\CO)^n$ ($i=1,2$) definable group embeddings whose respective images are  open balls  in $(K/\CO)^n$, where $n$ is the $K/\CO$-rank of $G$. Then  $\dpr(H_1\cap H_2)=n$ and $$\mathrm{Tor}(H_1)=f_1^{-1}(\mathbb{F}/\mathcal{O}_{\mathbb{F}})=\mathrm{Tor}(H_2)=f_2^{-1}(\mathbb{F}/\mathcal{O}_{\mathbb{F}}).$$ 

In particular, all definable subgroups of $G$ of dp-rank $n$ that can be definably embedded into  $(K/\CO)^n$  share the same torsion subgroup.
\end{lemma}
\begin{proof}
The assumptions and the conclusions are invariant under naming new constants, so we may assume that $\mathbb{F}$ is named in $\CK$ and so we may apply the results from \cite{HaHaPeGps}.

By the construction of $\nu_{K/\CO}$ (see Lemma \ref{F: nu in K/O} and Remark \ref{R: nu lives on any definable subgroup witnessing}) we have $\nu_{K/\CO}\vdash H_i$, $i=1,2$, hence    $\nu_{K/\CO}\vdash H_1\cap H_2$. By Lemma \ref{F: nu in K/O}, this implies that $\dpr(H_1\cap H_2)=n$. 

 Since $f_i(H_i)$ is an open ball, for $i=1,2$, it follows from  Fact \ref{F: torsion} that $\mathrm{Tor}(H_i)=f_i^{-1}((\mathbb{F}/\CO_{\mathbb{F}})^n)$. As $\dpr(H_1\cap H_2)=n$ also $\dpr(f_i(H_1\cap H_2))=n$ for $i=1,2$, so  by  \cite[Lemma 3.6]{HaHaPeGps} $f_i(H_1\cap H_2)$ has non-empty interior, thus contains a sub-ball of $(K/\CO)^n$. Therefore, (since it is a group) it also contains a ball centered at $0$. Thus, $(\mathbb F/\CO_{\mathbb F})^n\sub f_i(H_1\cap H_2)$ and hence $f_i^{-1}((\mathbb F/\CO_{\mathbb F})^n)\sub H_1\cap H_2$. We conclude \[\mathrm{Tor}(H_1)=f_1^{-1}((\mathbb F/\CO_{\mathbb F})^n)=f_2^{-1}((\mathbb F/\CO_{\mathbb F})^n)=\mathrm{Tor}(H_2),\] as needed.
\end{proof}

We can now prove Proposition \ref{P: K/O} in the $p$-adic case.

\begin{proof}[Proof of proposition \ref{P: K/O} in the $p$-adic case.] Recall that $\CA=\{\lambda_s:s\in S\}$ is a definable family of automorphisms of $G$.
  First, we show that some infinite $\CA$-invariant abelian subgroup of $G$ is definable in $\CK$ and then we construct one that is definable over $K_0$ as needed.
  
  By Section \ref{ss:nu in K/O} we can find a definable subgroup $H_0$,  $\nu_{K/\CO}\vdash  H_0\leq G$,  that is definably isomorphic to  an open ball in $(K/\CO)^n$ centered at $0$, where $n$ is the $K/\CO$-rank of $G$. Let $f:H_0\to (K/\CO)^n$ be a group embedding witnessing this (note that $H_0$ and $f$ are not claimed to be $\CK_0$-definable).
  
  Let $H=\bigcap\limits_{s\in S} H_0^{\lambda_s}$, where $H_0^{\lambda_s}=\lambda_s(H_0)$. It is a definable $\CA$-invariant abelian subgroup, and by the previous lemma it is infinite, as claimed. We shall now replace $H$ by a group defined over $K_0$.

  By Lemma \ref{L:full subgroups have the same torsion}, $\mathrm{Tor}(H_0^{\lambda_s})=f^{-1}((\mathbb F/\CO_{\mathbb F})^n)$, for every $s\in S$.  It follows, using compactness and saturation, that there is   $r<\mathbb{Z}$ such that $B_{>r}(0)\sub f(H)$. Let $r_0$ be the minimal such $r$.
  
  Assume that $H$ and $f$ are definable over some $t_0\in \CK$ and let $\{(H_t,f_t):t\in T\}$ be the corresponding $K_0$-definable family of subgroups of $G$ and definable group embeddings $f_t:H_t\to (K/\CO)^n$, such that $(H,f)= (H_{t_0},f_{t_0})$.
Note that the statement that $H_{t_0}$ is $\CA$-invariant is a first order property of $t_0$, defined over $K_0$. 

 Thus we may assume that  each $H_t$ is $\CA$-invariant. 
 
 Define $\eta:T\to \Gamma$  by $$\eta(t)=\min\{r\in \Gamma:B_{>r}(0)\sub f_t(H_t)\}.$$ In particular, $\eta(t_0)\leq r_0$ and by Lemma \ref{L:full subgroups have the same torsion}, if $\eta(t), <\mathbb{Z}$ then $\hat H:=f_{t_0}^{-1}((\mathbb F/\CO_{\mathbb F})^n)\sub H_t.$

  Given  $r \in \Gamma_{<0}$, let 
  \[
    G(r):=\bigcap \{H_t:\eta(t)\leq r\}.
   \] 
     
 Because each $H_t$ is $\CA$-invariant so is $G(r)$, and as noted above,  $\hat H\sub G(r)$ for every $r\in \Gamma_{<0}$.

 The map $f_{t_0}$ restricts to an injective homomorphism from $G(r_0)$ into $(K/\CO)^n$, and since $\hat H\sub G(r_0)$, the  set $\{r\in \Gamma: f_{t_0}^{-1}(B_{>r}(0))\sub G(r_0)\}$  contains $\mathbb Z$. It follows that  there exists $r<\mathbb Z$ such that $f_{t_0}^{-1}(B_{>r}(0))\sub G(r_0) $ 
 and therefore   $\nu_{K/\CO}\vdash G(r_0)$ (by Lemma \ref{F: nu in K/O}).

  The family $\{G(r):r\in \Gamma\}$  is definable over $K_0$ and, by its definition, it is increasing as $r$ tends to $-\infty$. Hence, the directed union $$N:=\bigcup\limits_{r\in \Gamma_{<0}} G(r)$$  is   an  abelian subgroup defined over $K_0$, $\CA$-invariant and $\nu_{K/\CO}\vdash N$.  It follows that $\dpr(N)$ is at least the $K/\CO$-rank of $G$ (note however that we do not claim that $N$ is strongly internal to $K/\CO$).

This concludes the proof of Proposition \ref{P: K/O} in the $p$-adic case. \end{proof}

We now proceed to the remaining cases. 

\subsection{$\CK$ is power bounded $T$-convex or $V$-minimal}
We assume that $\CK$ is either power bounded $T$-convex or $V$-minimal. In both cases $K/\CO$ is an SW-uniformity and $\CK$ has residue characteristic $0$. 

Since $(K/\CO)^n$ is torsion-free we cannot use torsion elements as in the $p$-adic case, so we adopt a different approach. 
The key to our argument is the characterization of definable groups and endomorphisms of $(K/\CO)^n$ from  Section \ref{ss:groups in K/O}.

The conclusion of Proposition \ref{P: K/O}, in our case, will follow from the next proposition (recall that a ball containing $0$ in $K/\CO)^n$ is of the form $B^n$ for $B$ a ball in $K/\CO$): 

\begin{proposition} \label{a K/O result}
    Let $G$ be a definable group in $\CK$ and let $H\subseteq G$ be an infinite definable subgroup, definably isomorphic to a ball in $(K/\CO)^n$. Let $\sigma$ be a  definable automorphism of $G$
    and let $H^\sigma:=\sigma(H)$. Then $H^\sigma\cdot H\subseteq G$ is in definable bijection with a set of the form
    \[H\times \prod B_i\times \prod C_i,\]
    where each $B_i$ is a ball in $K/\CO$ and each $C_i$ is a  ball in $K/\m$.
    
    Furthermore,
    \begin{enumerate}
        \item If the $\bk$-rank of $G$ is $0$  then there are no $C_i$ in the above description, so $H^\sigma\cdot H$ is strongly internal to $K/\CO$.
        \item If $H^\sigma\neq H$ then $\dpr(H^\sigma\cdot H)>\dpr(H)$.
    \end{enumerate}
\end{proposition}
\begin{proof} 
We identify $H$ with its image in  $(K/\CO)^n$ (but still write the group operations multiplicatively)  and let  $H_3=\{(a,b)\in H\times H: a^\sigma b=e\}$.

\begin{claim} 
$H_3$ is a subgroup of $H\times H$ and $(H\times H)/H_3$ is in definable bijection with $H^\sigma\cdot H$.
\end{claim}
\begin{claimproof}
Note that if $a^\sigma b=e$ then $a^\sigma $ and $b$ are in $H_0:=H\cap H^\sigma$, so they commute.  To see that $H_3$ is a subgroup, assume that  $a_1^\sigma b_1=a_2^\sigma b_2=e$ then $(a_2^{-1})^\sigma a_1^\sigma b_1b_2^{-1}=(a_1a_2^{-1})^\sigma (b_1b_2^{-1})=e$, so $(a_1a_2^{-1}, b_1b_2^{-1})\in H_3$.

We claim that for $a,b\in H$,
$a_1^\sigma b_1=a_2^\sigma b_2$ if and only if $(a_1,b_1)H_3=(a_2,b_2)H_3$, and therefore the map $(a,b)\mapsto a^\sigma b$ induces a well-defined bijection between $(H\times H)/H_3$ and $H^\sigma \cdot H$. Indeed, using the commutativity of $H^\sigma$,
\[a_1^\sigma b_1=a_2^\sigma b_2\Leftrightarrow (a_2^\sigma)^{-1} a_1^\sigma b_1b_2^{-1}=e\Leftrightarrow a_1^\sigma (a_2^\sigma)^{-1}b_1b_2^{-1}=e\Leftrightarrow (a_1,b_1)H_3=(a_2,b_2)H_3.\]\qedhere \end{claimproof}

The claim implies, in particular, that in order to compute $\dpr( (H^\sigma\cdot H)$ it will suffice to compute $\dpr\left( (H\times H)/H_3\right)$, to which we now turn our attention. 

By definition, $H_3$ is the graph of a definable injective partial function $T:H^\sigma\cap H\dashrightarrow H^\sigma\cap H$, $x\mapsto (x^{\sigma})^{-1}$, in particular $\dom(T)$ is a definable group. We want to study the map $T$. To do that we may  work solely inside $(K/\CO)^n\times (K/\CO)^n \supseteq H\times H$ so we switch to additive notation.

By Lemma \ref{automorphism}, there is a definable automorphism 
$f:(K/\CO)^n\to (K/\CO)^n$ extending $T$.  By Corollary \ref{C: same valuation}, $f$ preserves the valuation, and as $H$ is a ball, we get that $f(H)=H$. Let us replace $f$ by $f\restriction H$. As $H$ is abelian, $x\mapsto -f(x)$ is again an automorphism.

Consider the definable map $F:H\times H\to H\times H$: $F(x,y)=(x,y-f(x))$.
Because $f$ is an endomorphism of $H$, $F$ is an automorphism of $H\times H$. It maps $H_3$ onto a group of the form $H_1\times \{e\}$, where $H_1=\dom(T)$. Hence
\[(H\times H)/H_3 \cong (H\times H)/(H_1\times \{e_H\}) \cong (H/H_1)\times H.\] 

By Lemma \ref{K/O end-groups2}, there is a definable automorphism of $(K/\CO)^n$ mapping $H_1$ to a direct product of closed and open balls in $K/\CO$ (or $K/\CO$ or $\{0\}$). Since $H$ of the form $B^n$, for $B\sub K/\CO$, this automorphism preserves $H$ (Corollary \ref{C: same valuation}). Consequently, we may assume that 
\[
H_1=\prod B_i \times \prod C_i \times \prod \{0\},
\]
where $B_i$ are closed balls and $C_i$ are open balls. Therefore,    $H/H_1$ is definably isomorphic to
\[\prod B/B_i\times \prod B/C_i\times \prod B. \]

Each  $B/B_i$ is definably isomorphic to a ball in $K/\CO$ (so strongly internal in $K/\CO$) and 
each $B/C_i$ is definably isomorphic to  ball in $K/\bm$ (so strongly internal to $K/\bm$. This gives the desired form.

For $(1)$, if The $\bk$-rank of $G$ is $0$ then there are no open $C_i$ in the above description; so $H^\sigma\cdot H$ is strongly internal to $K/\CO$. 

For $(2)$, if $H^\sigma\neq H$ then $H^\sigma \cap H\subsetneq H$ and in particular $H_1\subsetneq H$. Since $\Gamma$ is dense, $[H:H_1]=\infty$ so  $\dpr(H/H_1)>0$ and thus $\dpr(H^\sigma\cdot H)>\dpr (H)$.
\end{proof}

We can now complete the proof of Proposition \ref{P: K/O} when $\CK$ is either power bounded or V-minimal.  Let $G$ be an infinite $\CK_0$-definable group whose $\bk$-rank is $0$. By Section \ref{ss:nu in K/O} we can find a definable subgroup $H\sub G$ definably isomorphic to an open ball in $(K/\CO)^n$ centered at $0$, where $n$ is the $K/\CO$-rank of $G$.  It follows from Proposition \ref{a K/O result} and the choice of $H$ that $H$ is invariant under every definable automorphism of $G$.
Indeed, assume towards contradiction that $H^\sigma\neq H$. Then by (1) of the proposition, $H^\sigma\cdot H$ is strongly internal to $K/\CO$ and by (2) $\dpr(H^\sigma\cdot H)>\dpr(H)$, contradicting the fact that $\dpr(H)$ is the $K/\CO$-rank of $G$.

Thus, $H$  is infinite, normal and abelian. Since any non-zero subgroup of $(K/\CO)^n$ is infinite, the existence of such a subgroup $H$ is an elementary property, which implies that such a group exists already in $\CK_0$, as claimed. \qed \\

We end this section with an example illustrating that in Proposition \ref{P: K/O} the assumption that the $\bk$-rank of $G$ is $0$ is essential.

\begin{example}
We produce an example of a group $G$ of dp-rank $2$ that is locally strongly internal to both $K/\CO$ and  $\bk$ but has no infinite definable normal abelian subgroup 
which is locally strongly internal to $K/\CO$.

    Let $\CK$ be either a $V$-minimal valued field or a power-bounded $T$-convex valued field, and let $\gamma>0$ be some element of $\Gamma$. Let $B_{\geq \gamma}$ and be $B_{\geq -\gamma}$ the closed balls of respective radii $\gamma$ and $-\gamma$ around $0$. 
    
    Pick any $\delta\in \Gamma$ with $2\delta>\gamma>\delta>0$, then $H=(1+B_{>\delta})/(1+B_{\geq \gamma})$ is a definable multiplicative group definably isomorphic (because of our choice of $\delta$)  to the additive group $B_{>\delta}/B_{\geq \gamma}$ (via the map $a+B_{\geq \gamma}\mapsto (1+a)(1+B_{\geq \gamma})$). This latter group is obviously definably isomorphic to a subgroup of $K/\CO$. Let $N=B_{>-\gamma}/\m$  (which is strongly internal to $\bk$).

    Set $G=N\rtimes H$, where $H$ acts on $N$ by multiplication (it is well-defined) and the latter is a normal subgroup of $G$. We identify both of these groups with their obvious images in $G$, namely we identify $g=\bar g+\bm\in N$ with $(\bar g+\bm,1+B_{\geq \gamma})$, and $a=\bar a(1+B_{\geq \gamma})\in H$ with $(\bm,\bar a(1+B_{\geq \gamma}))$.

    A direct computation gives that if $a\in H$ and $g\in N$ as above,  $$a^g=g^{-1}ag=(\bar g(\bar a-1)+\m,\bar a(1+B_{\geq \gamma})).$$ 
    
  Assume now that $L$ is a definable, normal subgroup of $G$ which is locally strongly internal to $K/\CO$. We will show that $L$ is not abelian.  By assumption,  $\nu_{\CK/\CO}\vdash L$, so $L\cap H$ is infinite and in particular contains a non identity element of the form $a=\bar a(1+B_{\geq \gamma})$, with $\gamma>v(\bar a-1)=\delta_1>\delta$. We claim that for a suitable choice of $g\in G$, $a^ga\neq aa^g$, implying that $L$ is not abelian.

  Indeed, choose $g=\bar g+\bm\in N$, so that $v(\bar g)+\delta_1<0$ (we can do that since $-\gamma+\delta_1<0$), and then, by the above computation
  $$a^ga=(\bar g(a-1)+\bar a \bar g+\bm, \bar a(1+B_{\geq \gamma})),\,\,\,\,\,  aa^g=(\bar g+\bar g(a-1)+\bm, \bar a(1+B_{\geq \gamma})).$$

  In order to see that $a^ga\neq aa^g$, it is enough to see that $\bar g(a-1)+\bar a\bar g-(\bar g +\bar g(a-1))+\bm \neq \bm$, namely that $\bar g(\bar a-1)\notin \bm$. This follows directly from our choice of $g$, since $v(\bar g)+v(\bar a-1)<0$.

    We end with noting that similar computations give \[H^g\cdot H=\{(\bar a(1-\bar g)+\m,\bar b(1+B_{\geq \gamma})):\bar a,\bar b\in 1+B_{>\delta}\},\] and thus it is not hard to see that  $H^g\cdot H=B_{>\delta+v(g)}/\m\times H$ which is line with the Proposition \ref{a K/O result}(1).
\end{example}

\section{Groups locally strongly internal to the residue field}
The results of the previous sections imply, in particular, that there are no definably semisimple groups locally strongly internal to $\Gamma$ (and in the $p$-adic case, nor to $K/\CO$). This is, clearly, not the situation for groups locally strongly internal to the valued field or to the residue field. So our aim in the present and in the next section is to study such groups. We begin with the study of groups locally strongly internal to $\bk$, where $\CK$ is either power-bounded $T$-convex or $V$-minimal. 

For the statement of the main result of this section, we need a weakening of definable semisimplicity: 

\begin{definition}
    Let $G$ be a definable group. A definable normal subgroup $H\trianglelefteq G$ is $G$-\emph{semisimple} if $H$ has no infinite abelian definable subgroups normal in $G$.  
\end{definition}

Note that, in the above notation, if  either $G$ or $H$ are definably semisimple, then $H$ is $G$-semisimple. We prove:

\begin{proposition}\label{P: k}
Let $G$ be a definably semisimple group locally almost strongly internal to $\bk$. Then there exists a finite normal subgroup $N\trianglelefteq G$ and two normal subgroups $G_1,G_2\trianglelefteq G/N$, all defined over any model over which  $G$ is defined, such that 
\begin{enumerate}   
    \item $G_1\cap G_2=\{e\}$,  $G_1,G_2$ centralize each other and  $G_1\cdot G_2$ has finite index in $G/N$.
    \item The almost $\bk$-rank of $G_1$ is $0$ and it is $G/N$-semisimple, 
    \item $G_2$ is definably semisimple, and it is definably isomorphic to a subgroup of $\gl_n(\bk)$. 
\end{enumerate}
\end{proposition}

Recall that a group $G$ is  \emph{a definably connected} if it has no definable subgroups of finite index. Note that for $G$ an arbitrary definable group, if there exists a definably connected subgroup of finite index, then it is necessarily unique and denoted by $G^0$.
Clearly, if $G^0$ exists then it is {\em definably characteristic in $G$}, namely invariant under all definable automorphisms of $G$.

\begin{fact}\cite[Fact 2.11]{PePiSt}\label{F:finite center, centerless}
    Let $G$ be a definably connected group definable in some structure $\CM$.
    \begin{enumerate}
        \item If $N$ is a finite normal subgroup, then $N\subseteq Z(G)$.
        \item If $Z(G)$ is finite, then $G/Z(G)$ is centerless.
    \end{enumerate}
\end{fact}

The proof of Proposition \ref{P: k} splits into two cases.

\subsection{$\bk$ is o-minimal}

In this subsection, we assume that $\CK$ is power bounded $T$-convex, thus $\bk$ is an  o-minimal expansion of a real closed field \cite[Theorem A]{vdDries-Tconvex}.
We first need a lemma allowing us, under suitable assumptions, to transfer definable semisimplicity under definable group homomorphisms: 

\begin{lemma}\label{L:connected subgroup of defsemisimple}
  Assume that  $G$ is a definable group in $\CK$,     $B\sub \bk^n$ is a definable group, and   $f:G\to B$ a definable surjective homomorphism.
    Let $H\trianglelefteq G$ be a normal definable subgroup with $\ker(f\restriction H)$ finite.
    Then:
    \begin{enumerate}
        \item $H^0$ exists.
        \item If $H$ is $G$-definably semisimple, then  $H^0$ and $f(H^0)$ are definably semisimple.
    \end{enumerate}
\end{lemma}
\begin{proof}

    (1) $f(H)$ is a definable group in the o-minimal structure $\bk$, so $f(H)^0$ exists. Since $\ker(f\restriction H)$ is finite, $H^0$ exists as well. Indeed, if not then there exists an infinite descending chain of finite index subgroups in $H$, which would give rise to  a proper finite index subgroup of  $f(H)^0$, contradiction.

    (2) Assume that $H$ is $G$-definably semisimple.  Let $N=f(H^0)$; it is a definably connected component. If $N$ is definably semisimple then so is $H^0$, so it suffices to show that $N$ is definably semisimple. Assume towards a contradiction that $N$ contains an infinite definable abelian normal subgroup $A$. 

    Recall that the {\em definable solvable radical of $N$} is
the subgroup of $N$ generated by all definably connected solvable normal subgroups of $G$. It is itself definable because of dimension considerations, and  clearly definably characteristic in $N$.  Let $R$ be the definable solvable radical of $N$. The group $A^0$ is contained in $R$ so $R$ is infinite. By \cite[Corollary 5.6]{baro-solvconn}, $R$ contains an infinite abelian definable definably connected subgroup $R_0$ that is definably characteristic in $N$.

    Let $A_1$ be the connected component of $f^{-1}(R_0)\cap H^0$. 
    Since $R_0$ is a definably connected group, $f(A_1)=R_0$. We claim that $Z(A_1)$ is infinite. Indeed, if it were finite then, by Fact \ref{F:finite center, centerless}, the group $A_1/Z(A_1)$
   is centerless. However, because  $\ker(f\restriction A_1)$ is finite, it follows from  the same fact that $\ker(f\restriction A_1)\sub Z(A_1)$. Thus, $A_1/Z(A_1)$ can also be written as a quotient of $f(A_1)=R_0$, so must be abelian, a contradiction.

     Since $R_0$ is a characteristic subgroup of $N=f(H^0)$ and $H^0$ is normal in $G$, the group $f^{-1}(R_0)\cap H^0$ is invariant under conjugation by elements of $G$; thus so are $A_1$ and $Z(A_1)$.  Thus, $Z(A_1)$ is an infinite abelian definable subgroup of $H$ and normal in $G$, contradicting the definable $G$-semisimplicity of $H$.
\end{proof}

Assume that $G$ is locally strongly internal to $\bk$. 
Let $\ad_{\bk}:G\to \gl_n(\bk)$ be the adjoint map, as discussed at the end of Section  \ref{S:infnit and local}. 

\begin{lemma}\label{L:kerad in o-minimal}
Let $G$ be locally strongly internal to $\bk$. Then,
\begin{enumerate}
    \item $\ker(\ad_{\bk})=C_G(\nu_\bk)$
    \item $\nu_\bk\vdash C_G(\ker(\ad_{\bk}))$
\end{enumerate}
\end{lemma}
\begin{proof}
    Let $\nu=\nu_\bk$.
    
    (1) Let $g\in \ker(\ad_{\bk})$. By \cite[Lemma 3.2(ii)]{OtPePi}, for any group $H$ definable in $\bk$,  two definable automorphisms $H$ with the same differential at $e_H$ coincide on an open neighborhood of $e_H$ in $H$. While the proof is stated for groups, the analysis holds for local groups as well. Hence, if $g\in \ker(\ad_\bk)$ then $\tau_g(x)=x$ on some $\tau_{\bk}$-open neighborhood of $e$, so by definition $g \in C_G(\nu)$. The reverse inclusion is immediate from the  definitions.

    (2) Since $\nu$ is the intersection of definable sets strongly internal to $\bk$, we may choose $\nu\vdash X\sub G$  that we can identify with a definable subset of $\bk^n$.  By cell decomposition in o-minimal structures, we may further assume that $X$ is definably connected.   By Lemma \ref{L:def connected}, $X\subseteq C_G(C_G(\nu))=C_G(\ker(\ad_\bk))$, thus $\nu\vdash C_G(\ker (\ad_{\bk}))$.
\end{proof}

\begin{proposition}\label{P:general H1,H2 into o-minimal}
Let  $G$  be 
a definably semisimple group in $\CK$, locally strongly internal to $\bk$. 
Let $H_1=\ker(\ad_\bk)$ and $H_2=C_G(H_1)$. Then
\begin{enumerate}
    \item $H_1$ and $H_2$ are normal subgroups, $H_2^0$ is definably semisimple, $H_1\cap H_2$ is finite and $H_1$ and $H_2$ centralize each other. 
    \item $H_1\cdot H_2$ has finite index in $G$.
    \item If the $\bk$-rank of $G$ equals the almost $\bk$-rank then $\dpr(H_2)$ equals the $\bk$-rank of $G$.
\end{enumerate}
\end{proposition}
\begin{proof}
    Let $\nu=\nu_\bk$.
    
    By Lemma \ref{L:kerad in o-minimal}, $H_1=C_G(\nu)$ and $\nu\vdash H_2$. By definition, $H_1$ is a definable normal subgroup, and thus so is $H_2$. 
    By the semisimplicity of $G$. the intersection of any definable normal subgroup $H$ with its centralizer is finite (otherwise, $Z(H)$ is infinite and normal in $G$).
    Thus  $H_1\cap H_2$  is finite, and by definition $H_1$ and $H_2$ centralize each other. By Lemma \ref{L:connected subgroup of defsemisimple}, $H_2^0$ is definably semisimple, completing the proof of (1).

    (2)  Note that
        \[G/(H_1\cdot H_2)\cong \frac{G/H_1}{(H_1\cdot H_2)/H_1}\cong \frac{G/H_1}{H_2/(H_1\cap H_2)}\cong \ad_\bk(G)/\ad_\bk(H_2),\]
    where $\ad_\bk(G)$ is the image of $\ad_\bk$ and $\ad_\bk(H_2)$ is the image of $\ad_\bk\restriction H_2$.

    Thus, we need to see that $\ad_\bk(G)/\ad_\bk(H_2)$ is finite. Since both images are subgroups of $\gl_n(\bk)$, we may freely use properties of groups definable in o-minimal expansions of fields. By o-minimality, showing  that $\ad_\bk(G)/\ad_\bk(H_2)$ is finite  amounts to showing that $\dim_\bk(\ad_\bk(G))=\dim_\bk(\ad_\bk(H_2))$ (we use $\dim_\bk$ for the o-minimal dimension in $\bk$). So, it is sufficient  to show that $\dim_\bk(\ad_\bk(G))\leq \dim_\bk(\ad_\bk(H_2))$. 
    
    As $G$ is definably semisimple, $H_2$ is $G$-definably semisimple. Since, by (1),  $\ker(\ad_\bk\restriction H_2)$  is finite, $H_2^0$ and $\ad_\bk(H_2^0)$ are definably semisimple by Lemma \ref{L:connected subgroup of defsemisimple}. Let $\mathfrak{h}$ be the Lie algebra (in the sense of \cite{PePiSt-defsimple})  of the definably connected  group $\ad_\bk(H_2^0)$ with its $\bk$-differential structure. By \cite[Theorem 2.34]{PePiSt-defsimple}, $\mathfrak{h}$ is a semisimple Lie algebra. Thus,  by \cite[Claim 2.8]{PePiSt-defsimple},  $\dim(\mathfrak{h})=\dim_\bk(\aut(\mathfrak{h}))$ (we use  the $\bk$-vector space dimension on the left and  the fact that $\aut(\mathfrak h)$ is definable in $\bk$).

    The group $\ad_\bk(G)$ acts on $\ad_\bk(H_2^0)$ by conjugation and thus also on $\mathfrak{h}$. We claim that the kernel of this action is trivial.

Indeed, assume that for some $g\in G$, the action of $\ad_\bk(g)$ on $\mathfrak h$ is the identity. By \cite[Lemma 3.2(ii)]{OtPePi}, it follows that for all $x\in \ad_\bk(H_2^0)$, $\ad_\bk(g^{-1}xg)=\ad_\bk(x)$, and hence for all $x\in H_2^0$, $g^{-1}xgx^{-1}\in \ker(\ad_\bk\restriction H_2^0)$. Since $\ker(\ad_\bk\restriction H_2^0)$ is finite, and $H_2^0$ is definably connected, it follows that for all $x\in H_2^0$, $g^{-1}xg=x$ and hence $g\in C_G(H_2^0)$. 
Because $\nu \vdash H_2$, then $g\in C_G(\nu)$, so by Lemma \ref{L:kerad in o-minimal}, $g\in \ker(\ad_\bk)$ and hence $\ad_\bk(g)=e$.

    We can therefore conclude that  $\ad_\bk(G)$ can be definably embedded into $\aut(\mathfrak{h})$  hence we get that 
    $\dim(\ad_\bk(G))\leq \dim(\aut(\mathfrak{h}))=\dim(\mathfrak{h})=\dim(\ad_\bk(H_2^0))$, so $\dpr(\ad_\bk(G))=\dpr(\ad_{\bk}(H_2^0))=\dpr(\ad_\bk(H_2))$, as required. 

    (3)  Because $\ker(\ad_k)\cap H_2$ is finite,  $H_2$ is almost strongly internal to $\bk$. Thus, the almost $\bk$-rank of $G$ is at least that of $H_2$. However, $\nu\vdash H_2$ so  $\dpr(H_2)$ is at least the $\bk$-rank of $G$. Because of the rank assumptions, we must have that $\dpr(H_2)$ is the $\bk$-rank of $G$.
    \end{proof}

\begin{remark}
As was noted in Remark \ref{R: H1H2 over the same parameters}, the groups $H_1$ and $H_2$ appearing in the statement of Proposition \ref{P:general H1,H2 into o-minimal} are definable over the same parameters as $G$. 
\end{remark}

We isolate the following direct consequences:
\begin{corollary}
    Let $G$ be locally strongly internal to $\bk$.
    \begin{enumerate}
        \item If $\ker(\ad_\bk)=G$ then $\nu_\bk\vdash Z(G)$. In particular, if $Z(G)$ is finite, then $\ker(\ad_\bk)$ is a proper subgroup of $G$.
        \item If $G$ is definably simple (namely non-abelian and has no non-trivial definable normal subgroup) then $G$ is definably isomorphic to a definable subgroup of $\gl_n(\bk)$.
    \end{enumerate}
\end{corollary}
\begin{proof}
    (1) If $G=\ker(\ad_\bk)$ then by Lemma \ref{L:kerad in o-minimal}(2), $\nu_\bk\vdash C_G(G)=Z(G)$.
    (2) Since $G$ is definably simple, either $\ker(\ad_\bk)=G$ or $\ker(\ad_\bk)=\{e\}$ Since $G$ is non-abelian, it follows from (1) that $\ker(\ad_{\bk})$ must be equal to $\{e\}$. 
    \end{proof}

The proof of Proposition \ref{P: k} when $\bk$ is o-minimal reduces to collecting what we have done so far: 
\begin{proof}[Proof of Proposition \ref{P: k} for o-minimal $\bk$]
    Fix $G$ a definably semisimple group locally almost strongly internal to $\bk$. 
   
    To prove (1) we need to find a finite normal $N\trianglelefteq G$ and definable  $G_1, G_2\trianglelefteq G/N$ centralizing each other with $G_1\cap G_2=\{e\}$.  By  Fact \ref{F: existence of finite normla to get D-group}, there exists  a finite normal  subgroup $N_1\trianglelefteq G$ such that $G/N_1$ is a $\bk$-group and the almost $\bk$-rank and the $\bk$-rank agree in $G/N_1$. Furthermore,  $N_1$ is definable over any model over which $G$ is defined. By Corollary \ref{C: semisimple in quotients} $G/N_1$ is definably semisimple, so -- in order to  keep notation simpler -- we  denote $G/N_1$ by $G$. By Lemma \ref{L:existence of local group}, $G$  contains a definable  normal differentiable local subgroup $\CG$ with respect to $\bk$, 
    with $\nu_{\bk}\vdash \CG$. 
    
    Then Proposition \ref{P:general H1,H2 into o-minimal} provides us with two definable normal subgroups $H_1, H_2$  satisfying (1) of the proposition.
    By Remark \ref{R: H1H2 over the same parameters}, $H_1$ and $H_2$ are both definable over any model over which $G$ is defined.   The group $N=H_1\cap H_2$ is a finite normal subgroup of $G$. Replace $G$ by $G/N$ and set $G_i:=H_i/N$.  Then $G_1$ and $G_2$ satisfy (1) of the proposition.

    For (3) we need to show that $G_2$ is definably semisimple, and definably isomorphic to a $\bk$-linear group. The latter is clear, since $\ad_\bk(G)$ is $\bk$-linear.  For the first part,  note that since $H_2^0$ is definably semisimple (by Proposition \ref{P:general H1,H2 into o-minimal}), so is $H_2$ and thus so is $G_2$ by Lemma \ref{C: semisimple in quotients}. \\
    
    It remains to prove (2), i.e., that the almost $\bk$-rank of $G_1$ is $0$ and that $G_1$ is $G/N$-semisimple.
    
    The latter part follows from the fact that $G_1$ is normal in the definably semisimple group $G$. So we only need to compute its almost $\bk$-rank. 
        
        Assume toward a contradiction that  $G_1$ is locally almost strongly internal to $\bk$. By applying Fact \ref{F: existence of finite normla to get D-group} to $G_1$, we get a finite normal subgroup $H\trianglelefteq G_1$ such that $G_1/H$ is locally strongly internal to $\bk$. Note that $H$ is normal in $G_1\cdot G_2$ as well.

        By Lemma \ref{L:kerad in o-minimal}, $\nu_\bk(G)\vdash G_2$. Since $G_1\cdot G_2$ has finite index in $G$, by Lemma \ref{L:nu of subgroup}(2) $\nu_\bk(G_1\cdot G_2)=\nu_\bk(G)$, so $\nu_\bk(G_1\cdot G_2)\vdash G_2$ and thus $\nu_\bk(G_1\cdot G_2)/H\vdash G_2/H$. By Lemma \ref{L:passage of D-group under finite-to-one}(3), $\nu_\bk(G_1\cdot G_2/H)\vdash G_2/H$ and by Lemma\ref{L:nu of subgroup}(1)  $\nu_\bk(G_1/H)\vdash \nu_\bk(G_1\cdot G_2/H)\vdash G_2/H$.  On the other hand,  obviously $\nu_\bk(G_1/H)\vdash G_1/H$ thus $(G_1\cap G_2)/H$ must be infinite, contradiction.
\end{proof}

\subsection{Proof of Proposition \ref{P: k} for $\bk$ an algebraically closed field.}
Throughout this subsection $\CK$ is assumed $V$-minimal, hence $\bk$ is a stably embedded pure algebraically closed field.  In particular, $\bk$ is strongly minimal. Fix a $\CK$-definable, definably semisimple group $G$ which is locally almost strongly internal to $\bk$.   By \cite[Proposition 6.2]{HaHaPeGps}, there exist definable subgroups  $H_0\trianglelefteq H\trianglelefteq G$, with $H$ definably connected and $H_0$ finite normal in $G$ such that $H/H_0$ is strongly internal to $\bk$.


Fix $H_0\trianglelefteq G$ and $H$ as above and consider $H_1=H/H_0$. By \cite[Theorem 1]{BousHrWeil} it is a $\bk$-connected algebraic group.  
By a classical theorem of Rosenlicht \cite[Theorem 13]{rosenlicht}, as  $H_1$ is a connected algebraic group,   $H_1/Z(H_1)$ is a $\bk$-linear group. As $G/H_0$ is definably semisimple (Corollary \ref{C: semisimple in quotients})
 and $H_1$ is normal in $G/H_0$, $Z(H_1)$ is finite.  Since $H_1$ is  connected $H_1/Z(H_1)$ is centerless (Fact \ref{F:finite center, centerless}).

We now fix a finite $N\trianglelefteq G$, $H_0\sub N$,  such that 
$H/N$ is a connected centerless $\bk$-linear group. Note that $G/N$ is still definably semisimple by Corollary \ref{C: semisimple in quotients}. Below we work in $G/N$, and to simplify notation we still use $H$ for $H/N$. Note that, since $\bk$ has definable Morley Rank, the statement "$H$ is a normal subgroup of $G$ strongly internal to $\bk$ whose Morley Rank equals the $\bk$-rank of $G$" is definable in families, and we can choose $H$ to be definable over any model in which $G$ is defined. 

\begin{claim}
    $H$ has no infinite normal abelian subgroups, hence it is a semisimple algebraic group.
\end{claim}
\begin{proof} Assume towards contradiction that such a normal subgroup existed. Then its Zariski closure is an infinite normal abelian  algebraic subgroup. Its (algebraic) connected component 
is contained in the solvable radical $R$ of $H$ which is therefore infinite as well. This radical contains an infinite abelian algebraic subgroup that is definably characteristic in $H$, and therefore is normal in $G$, contradicting our assumption. 
\end{proof}

\begin{claim} 
The group $C_G(H)\cdot H$ has finite index in $G$.
\end{claim} 
\begin{proof}
The group $G$ acts on $H$ by conjugation and because $\bk$ is stably embedded, each action is $\bk$-algebraic, so the map $f:g\mapsto \tau_g\restriction H$ sends $G$ into $\aut(H)$ the group of all algebraic automorphisms of $H$ (recall that $\tau_g: (x\mapsto x^g)$).
The kernel of the map is $C_G(H)$.

Applying \cite[Theorem 27.4]{Humph}, using the fact that $\bk$ is algebraically closed, we see that $\aut(H)$ is the semi-direct product of $\mathrm{Int}(H)$, the inner automorphisms of $H$,  and a finite group (we use here the fact that $H$ is assumed centerless). Since $f(H)=\mathrm{Int}(H)$, it follows that $f(G)$ has finite index in $\aut(H)$ so $C_G(H)\cdot H$ must have finite index in $G$.
\end{proof}

We now let $G_1=C_G(H)$ and $G_2=H$. Since $G_1$ and $G_2$ centralize each other and $G_2$ is centerless, $G_1\cap G_2=\{e\}$.
This ends the proof of (1).

By construction, $G_2$ is a linear $\bk$-group. Assume towards a contradiction that $G_1$ is locally almost strongly internal to $\bk$ as well. By \cite[Proposition 6.2]{HaHaPeGps}, there exists a finite definable normal subgroup $N'\trianglelefteq G_1$ such that $G_1/N'$ has a definable normal subgroup $B_1\trianglelefteq G_1/N'$ strongly internal to $\bk$. Since $G_1$ and $G_2$ intersect trivially, we may identify $G_2$ with $G_2/N'$. Moreover, the $\bk$-rank of $G_1\cdot G_2$, which equals that of $G$ (since it has finite index in it), is at most that of $(G_1\cdot G_2)/N'$, by Lemma \ref{L:passage of D-group under finite-to-one}; so $G_2=H$ is still $\bk$-critical in $(G_1\cdot G_2)/N'$. 
But then $B_1\cdot G_2 \cong  B_1\times G_2$ is strongly internal to $\bk$, with 
$\dpr(B_1\cdot G_2)>\dpr(G_2)$, contradicting the fact that $H=G_2$ was $\bk$-critical in $(G_1\cdot G_2)/N'$.

Finally,  we already saw that $G_2$ is definably semisimple. The fact that $G_1$  is $G/N$-semisimple, is immediate since $G/N$ is definably semisimple.

This finishes the proof of Proposition \ref{P: k} in the V-minimal case, and thus the proof of the proposition is now complete.

\section{$K$-groups}
In the notation of Section \ref{ss: valued field and residue field}, for a $K$-group $G$ there exists an infinitesimal type-definable subgroup $\nu_K(G)$  inducing a definable homomorphism $\ad_K:G\to \gl_n(K)$, for $n$ the $K$-rank of $G$.

Recall that a definable group $G$ is \emph{$K$-pure} if $G$ is locally  strongly internal to $K$ but not locally almost strongly internal to $\Gamma$, to $\bk$ or to $K/\CO$.  In the present section we collect some basic facts concerning  $K$-pure groups, as those appear naturally in our later analysis. \\

For the following result, we observe that all the valued fields we consider are $1$-h-minimal. The exact definition is immaterial here. See \cite{hensel-min} and \cite[Section 4.5]{HaHaPeVF}.

\begin{fact}\cite[Theorem 2]{AcHa}\label{F:AcHa}
    Let $\CK$ be a $1$-h-minimal field, $\CG=(X,\cdot, ^{-1})$ a definable strictly differentiable local group with respect to $\CK$ and $f:\CG\dashrightarrow \CG$ a definable strictly differentiable homomorphism of local groups. If $D_e(f)=\id$ then $\{y\in \dom(f):f(y)=y\}$ contains a definable open neighborhood of $e$
\end{fact}
\begin{proof} This is a theorem of Acosta and the second author, 
    \cite[Theorem 2]{AcHa}, implying that $\dpr \{y\in \dom(f) :f(y)=y\}=\dpr \dom(f)$, so contains a definable open subset; the result follows.
\end{proof}

We still use $\dim$ to denote the $\acl$-dimension in $K$ and the induced dimension on $K^{eq}$ and $\tau_K$ for the topology on $G$.

\begin{proposition}\label{P:dim C_G(g)=dimG}
    Let $G$ be a definable group in $\CK$, locally strongly internal to $K$. If  $g\in \ker(\ad_K)$ then $\dim C_G(g)=\dim G$.
\end{proposition}
\begin{proof}
    Let $\mathcal{G}=(X,\cdot,^{-1})$ be the definable strictly differentiable local group as provided by Lemma \ref{L:existence of local group}. If $g\in \ker(\ad_K)$ then by Fact \ref{F:AcHa}, the set  $W:=\{x\in X: x^g=x\}\sub C_G(g)$ is open in $X$. Since $\dim (X)$ is  the $K$-rank of $G$ (Corollary \ref{C:every thing is strongly internal to K}), we get that
    \[
    \dim (G)=\dim(X)=\dim (W)\leq \dim C_G(g)\leq \dim(G)
    .\qedhere\]
\end{proof}

The following is based on an analogous result of \cite{GisHalMac}: 

\begin{corollary}\label{C:dugaldetal}
    Let $G$ be a definable group, locally strongly internal to $K$ and let $g\in G$. If $G$ is $K$-pure and $\dim(C_G(g))=\dim(G)$ then $[G:C_G(g)]<\infty$. In particular, $[G:C_G(g)]<\infty$ for every $g\in \ker(\ad_K)$. 
\end{corollary}
\begin{proof}
    The conjugacy class $g^G$ is in definable bijection with the imaginary sort $G/C_G(g)$. By additivity of dimension we get that $\dim(g^G)=\dim(G)-\dim(C_G(g))$. If $\dim(C_G(g))=\dim(G)$ then  $\dim(g^G)=0$.
        By Lemma \ref{l: pure 0dim}, $g^G$ is finite, hence $[G:C_G(g)]$ is finite. 
        \end{proof}

\section{Definably semisimple groups}
We can finally prove the main results of the paper. Recall, first,  that a definable group is \emph{definably simple} if it is non-abelian and  has no definable normal subgroups, it is  \emph{definably semisimple} if it has no  definable infinite normal abelian subgroups. 

We point out that definable semisimplicity is not, a priori, an elementary property of groups definable in $\CK^{eq}$, as $\CK^{eq}$ may not eliminate the quantifier $\exists^\infty$. As we will see below, one of the  corollaries of the present work  is that in our setting, definable semisimplicity, is, in fact, elementary. 
i.e., if $\CK_0\prec \CK$ and $G$  is a $K_0$-definable group, such that $G$ is     definably semisimple in $K_0$ then it remains so in $\CK$.

As before, $\CK=\CK^{eq}$ is a sufficiently saturated valued field, either power-bounded $T$-convex, $V$-minimal or $p$-adically closed. Throughout the previous sections, we were working under the assumption that our definable group $G$ is a $D$-group (for some distinguished sort $D$).  As shown in \cite{HaHaPeGps}, this need not be the case as $G$ might not be locally strongly internal to any distinguished sort. The best we can obtain, in general,  that if $G$ is locally \emph{almost} strongly internal to $D$  and then there is a finite normal subgroup $H$ such that $G/H$ is a $D$-group (so in particular, locally strongly internal to $D$), Fact \ref{F: existence of finite normla to get D-group}. Fortunately, in our setting,  Corollary \ref{C: semisimple in quotients} assures that definable semisimplicity is preserved under finite quotients and under finite extensions. 

Before  stating the first of the results, recall from \cite[\S 9.3]{JohnTopQp} that a topological group $G$ is \emph{locally abelian} if there exists $W\ni e$, an open neighborhood of $e$ in $G$,  such that $xy=yx$ for all $x,y\in W$.

The next theorem gives conditions under which a definable,  infinite, abelian normal subgroup must exist in $G$. Recall that if $\dim(G)>0$ then by Corollary \ref{C:every thing is strongly internal to K} it is locally strongly internal to $K$.

\begin{theorem}\label{T:johnson question}
    Let $G$ be an infinite group definable over some $\CK_0\prec \CK$. 
    \begin{enumerate}
        \item If $G$ is $K$-pure (so locally strongly internal to $K$) and locally abelian with respect to $\tau_K$ then there exists a definable abelian subgroup  $G_1\trianglelefteq G$ of finite index, defined over $\CK_0$. In particular, $G_1$ is open. 
        
        \item 
        \begin{enumerate}
            \item If $G$ is locally almost strongly internal to $\Gamma$ then  there exists a $\CK_0$-definable infinite normal abelian subgroup $N\trianglelefteq G$, whose dp-rank is at least the almost $\Gamma$-rank of $G$.

            \item If $G$ is locally almost strongly internal to $K/\CO$ but not to $\bk$ then  there exists a $\CK_0$-definable infinite normal abelian subgroup $N\trianglelefteq G$ whose dp-rank is at least the $K/\CO$-rank of $G$.
\end{enumerate}       
    \end{enumerate}
    
\end{theorem}
\begin{proof} 
(1)   Since $G$ is locally strongly internal to $K$, it is a topological group with respect to the $\tau_K$-topology. All topological notions below refer to $\tau_K$.

 Assume that $G$ is locally abelian.
     By Lemma \ref{L:existence of local group}, there exists a local differentiable abelian subgroup $\CG=(U,\cdot, ^{-1})$ of $G$. Let $\tau_g$ denote conjugation by $g$. As $\tau_g\rest U=\id$ for all $g\in U$, we get that $U\subseteq \ker(\ad_K)$. This gives $\dim(\ker(\ad_K))=\dim(G)$.   
     
     The proof that $G$ is abelian-by-finite is an  adaptation of \cite[Proposition 2.3]{PiYao}. By Corollary \ref{C:dugaldetal}, since $G$ is $K$-pure, $[G:C_G(a)]<\infty$ for all $a\in U$. 
     By Fact \ref{Baldwin Saxl}, there is a definable normal subgroup of finite index $H_0\trianglelefteq G$ such that $H_0\leq C_G(U)$. 
     
     For every $h\in H_0$, $U\sub C_G(h)$ hence $\dim C_G(h)=\dim G$, e.g, by Corollary \ref{C: open iff full dim}. Therefore,    by  Corollary \ref{C:dugaldetal} and $K$-purity, we have $[G:C_G(h)]<\infty$ for every $h\in H_0$. Thus, applying Fact \ref{Baldwin Saxl} again, we see that  $C_G(H_0)$ has finite index in $G$, so in particular,  $G_1=C_G(H_0)\cap H_0$ has finite index in $G$ and is commutative. It follows that $G_1$ is open by Corollary \ref{C: open iff full dim}. The fact that $G_1$ is a definable, open, normal abelian, subgroup of index $k$ (some $k\in \Nn)$, is first order, so we can find such $G_1$ defined over  $\CK_0$. 

    (2) Assume now that $G$ is locally almost strongly internal to $D$, where $D=\Gamma$ or $D=K/\CO$. By  Fact \ref{F: existence of finite normla to get D-group} there exists $H\trianglelefteq G$ a finite normal subgroup such that  $G/H$ is locally strongly internal to $D$ and a $D$-group. Moreover, the $D$-rank of $G/H$ is the almost $D$-rank of $G$, and $H$ is $\CK_0$-definable. Also, if $G$ was not almost strongly internal to $\bk$ then neither is $G/H$.

    Assume  that $D=\Gamma$. By Proposition  \ref{P: Gamma}, we have $\nu_{\Gamma}(G/H)\vdash Z(G/H)$. In particular, $G/H$ contains a normal abelian subgroup whose dp-rank is at least the $\Gamma$-rank of $G/H$ (equivalently, the almost $\Gamma$-rank of $G$). By Corollary \ref{C: semisimple in quotients}, $G$ contains a definable normal abelian subgroup of the same dp-rank.
    
    Assume that $G$ is locally almost strongly internal to $K/\CO$ but not to $\bk$, so $G/H$ is locally strongly internal to $K/\CO$ (but not to $\bk$) and its $K/\CO$-rank equals the almost  $K/\CO$-rank of $G$. By Proposition \ref{P: K/O}, as $G$ and $H$ are both $\CK_0$-definable, there exists a $\CK_0$-definable infinite normal abelian subgroup of  $G/H$ whose dp-rank is at least the almost $\Gamma$-rank of $G/H$.
    By Corollary \ref{C: semisimple in quotients}, $G$ contains a definable normal abelian group of the same rank.
\end{proof}

The following example shows that the assumption of $K$-purity is needed in Theorem \ref{T:johnson question}(1), in order for local commutativity to imply the existence of a definable open normal abelian subgroup:

\begin{example}\label{E:counter}
Let $\CK$ be a $p$-adically closed field. Let $\CO^\times$ denote the multiplicative group of $\CO$. Consider the semi-direct product $G=\CO^\times \ltimes K/\CO$, where $(a,b+\CO)\cdot (c,d+\CO)=(ac, b+ad+\CO)$. Then $\dim(G)=1$ and $\dpr(G)=2$. It is locally abelian, as witnessed by $\CO^\times\times \{0\}$. We claim that $G$ has no definable open normal abelian subgroup. Assume, towards a contradiction, that  $H$ is such, in particular by \cite[Theorem 1.4(1)]{JohnTopQp} $\dim(H)=\dim(G)$ so $\pi_1(H)$, the projection on the first coordinate, must be infinite.

Let $(t,0)\in H$ for $t\neq 1$. Since the conjugation of $(t,0)$ by $(1,b+\CO)$ is $(t,b-bt+\CO)$, by letting $b$ vary we conclude that $\pi_2(H)$, the projection on the second coordinate, is equal to $K/\CO$. Thus, $H=U\ltimes K/\CO$ for some infinite definable subgroup $U$ of $\CO^\times$. Every element of $\CO^{\times}$ acts non-trivially on $K/\CO$, thus $U\ltimes K/\CO$ is not abelian unless $U=\{1\}$,  proving that $H$ as required does not exist. 

On the other hand, note that $\{1\}\times K/\CO$ is an infinite definable normal abelian subgroup (that is not open).
\end{example}

Theorem \ref{T:johnson question} together with the above example answers a question of Johnson's \cite[\S 9.3]{JohnTopQp} on locally abelian groups in $p$-adically closed fields.

We can now prove the main result of this paper. Note that below $\CK_0$ is not assumed to be saturated. 
\begin{theorem}\label{T: main}
    Let $\CK_0$ be either a power bounded $T$-convex field, a $V$-minimal field or a $p$-adically closed field. Let $G$ be an infinite definable, definably semisimple group in $\CK_0$. Then there exists a finite normal subgroup $N\trianglelefteq G$ and two normal subgroups $H_1,H_2\trianglelefteq G/N$, such that 
    \begin{enumerate}
        \item $H_1\cap H_2=\{e\}$, $H_1$ and $H_2$ centralize each other and $H_2$ is definably semisimple.
        \item $H_1\cdot H_2$ has finite index in $G/N$.
        \item $H_1$ is definably isomorphic to a subgroup of $\gl_n(K_0)$
        \item $H_2$ is definably isomorphic to a subgroup of $\gl_n(\bk_0)$.
    \end{enumerate}
If the almost $\bk$-rank of $G$ is $0$ (e.g., in the $p$-adically closed case) then $H_1=G/N$.
\end{theorem}

\begin{proof}
Let $\CK\succ \CK_0$ be a sufficiently saturated elementary extension. Throughout the proof below, we use $G$ to denote $G(\CK)$. As a first approximation we prove the existence of $N, H_1,H_2\sub G$ as above, all defined over $K_0$, satisfying (1), (2) and (4), such that $H_1$ is $K$-pure. We shall later show that after modding out by another finite subgroup $H_1$ becomes $K$-linear.

We divide the proof into two cases:

\underline{(a) $\CK_0$ is $V$-minimal or power bounded $T$-convex.}
\vspace{.2cm} 

In this case, either by \cite[\S 3]{JohCminimalexist} in the V-minimal case, or by Proposition \ref{P:exists-infty V-min or T-conv} in the $T$-convex power bounded case, $\CK^{eq}$ eliminates $\exists^\infty$ and therefore $G$ is definably semisimple.

   By Fact \ref{F: existence of finite normla to get D-group}, there exists a $K_0$-definable finite normal subgroup $N'\trianglelefteq G$ such that in $G/N'$ the almost $K/\CO$-rank and the $K/\CO$-rank agree (they may be zero); by Lemma \ref{L:passage of D-group under finite-to-one}(4) this still holds if we further quotient by finite normal subgroups. Replace $G$ by $G/N'$ (using Corollary \ref{C: semisimple in quotients} which says it is still definably semisimple).

     Assume first that $G$ is locally almost strongly internal to $\bk$. 
    By Proposition \ref{P: k}, there is a finite normal subgroup $N_0\trianglelefteq G$ definable over $K_0$, and $K_0$-definable normal subgroups $H_1,H_2\trianglelefteq G/N_0$  such that $H_1\cap H_2=\{e\}$,  $H_1\cdot H_2$ has finite index in $G/N_0$ and $H_1, H_@$ centralize each other. Furthermore, $H_2$ is $K_0$-definably isomorphic to a $\bk$-linear definably semisimple group and the almost $\bk$-rank of $H_1$ is $0$. Since $G$ is definably semisimple, so is $G/N_0$ (Corollary \ref{C: semisimple in quotients}). Replace $G$ by $G/N_0$.

    If the almost $\bk$-rank of $G$ is $0$ then we just take $H_1=G$ and $H_2=\{e\}$. 

\begin{claim}
        The almost $K/\CO$-rank of $H_1$ is $0$.
    \end{claim}
    \begin{claimproof}
Assume towards contradiction that $H_1$ is almost locally strongly internal to $K/\CO$. By Fact \ref{F: existence of finite normla to get D-group}, there exists a finite $N_1\trianglelefteq H_1$,  invariant under conjugation in $G$ (namely normal in $G$), such that $H_1/N_1$ is locally strongly internal to $K/\CO$. Notice that $G$ acts on $H_1/N_1$ by $\sigma_g(hN_1):=h^gN_1$.

Since the almost $\bk$-rank of $H_1$ is $0$, so  is the almost $\bk$-rank of $H_1/N_1$.
  We now apply Proposition \ref{P: K/O} to $H_1/N_1$ and the definable  family of automorphisms $\CA=\{\sigma_g:g\in G\}$, and obtain a definable infinite normal abelian subgroup of $H_1/N_1$ which is $\CA$-invariant. By Corollary \ref{C: semisimple in quotients}, $H_1$ contains a definable infinite normal abelian subgroup which is invariant under conjugation in $G$, namely normal in $G$. This contradicts the semisimplicity of $G$.
\end{claimproof}

By Theorem \ref{T:johnson question}(2a), the almost $\Gamma$-rank of  $G$ is $0$ and therefore the same is true for $H_1$. So $H_1$ is $K$-pure, as claimed.

This completes the proof of our approximation to the theorem, when $\CK$ is either $V$-minimal or power bounded $T$-convex.

\underline{(b) Assume now that $\CK$ is $p$-adically closed.}

\vspace{.2cm}

In this case, we just need to show that $G$ is $K$-pure (and then we take $H_1=G$). However, since $\CK$ does not eliminate $\exists^\infty$ we cannot assume a-priori that it is definably semisimple. 

Again, by Theorem \ref{T:johnson question}(2a), the almost $\Gamma$-rank of $G$ is $0$, for otherwise $G$ would have a $K_0$-definable infinite normal abelian subgroup, whose $K_0$-points would contradict the definable semisimplicity of $G(\CK_0)$. 

Since the almost $\bk$-rank of $G$ is obviously $0$, it follows from Theorem \ref{T:johnson question} 2(b), that the almost $K/\CO$-rank of $G$ must be $0$. Indeed, if not, then once again $G$ would contain an infinite $K_0$-definable normal abelian subgroup whose $K_0$-points would contradict the semisimplicity of $G(\CK_0)$.

We therefore showed, in the $p$-adically closed case, that $G$ is $K$-pure. This ends the proof of the approximated statement in all cases.
\\

We now proceed with the proof of Theorem \ref{T: main}. As we showed above, we have a finite $N\trianglelefteq G$, and  $H_1, H_2\trianglelefteq G/N$. all defined over $K_0$, satisfying (1), (2), (4), with $H_1$ being $K$-pure (in particular, $H_1$ is locally strongly internal to $K$). In the $p$-adically closed case, we take $H_1=G/N$ and $H_2=\{e\}$.

By Corollary  \ref{C: semisimple in quotients}, $G/N$ is still definably semisimple. For clarity of notation, we replace $G$ by $G/N$.

Note that $\dim G=\dim H_1+\dim H_2$, and since $\dim H_2=0$, we have $\dim G=\dim H_1$. By Lemma \ref{L:nu of subgroup}, $\nu_K(G)=\nu_K(H_1)$. By Lemma \ref{L:existence of local group},  $G$ contains a definable, differentiable normal local subgroup, with respect to $K$, which -- as $\dim(G)=\dim(H_1)$ -- we may assume to be contained in $H_1$. Thus we have  an associated $K_0$-definable map $\ad_K: G\to \gl_n(K)$, with $n=\dim H_1$. Let $\ad_K^{H_1}=\ad_K\restriction H_1$.

    \begin{claim}\label{C: last claim}
        $\ker(\ad_K^{H_1})$ is a finite normal subgroup of $G$.
    \end{claim}
    \begin{claimproof}
        Since $H_1$ is $K$-pure, by Corollary \ref{C:dugaldetal}, for every $h\in \ker(\ad_K^{H_1})$, $C_G(h)$ has finite index in $H_1$. By Corollary \ref{Baldwin Saxl}, there exists a $\CK_0$-definable subgroup $\widetilde H_1\trianglelefteq H_1$ of finite index, that is also normal in $G$, such that $\widetilde H_1\leq C_{H_1}(\ker \ad_K^{H_1})$ and thus $\widetilde H_1 \cap \ker(\ad_K^{H_1})\subseteq Z(\widetilde H_1)$.
        Since   $\ker(\ad_K^{H_1})=\ker(\ad_K)\cap H_1$ it is obviously normal in $G$.

        Thus, $\widetilde H_1\cap \ker(\ad_K^{H_1})$ is a $\CK_0$-definable normal abelian subgroup of $G$, so it must be finite by semisimplicity of $G(\CK_0)$.
        
        Finally, since $\widetilde H_1$ has finite index in $H_1$ it follows that $\ker(\ad_K^{H_1})$ is finite, as claimed.\footnote{The argument given in the claim shows that for $K$-pure groups, the kernel of $\ad$ has a (relatively) open normal abelian subgroup of finite index. This is true in particular for  $p$-adic Lie groups definable in the $p$-adic field. Recently,  \cite{helge-example},  Gl\"ockner constructed an example of a $1$-dimensional $p$-adic Lie group $G$ for which this fails. In fact, in his example $\ker (\ad_K)=G$, but $G$ contains no open normal abelian subgroup.}
    \end{claimproof}

Clearly, $H_1/\ker(\ad_K^{H_1})$ is definably isomorphic, over $K_0$, to a subgroup of $\gl_n(K)$, with $n=\dim H_1$. Since $\ker(\ad_K^{H_1})\cap H_2=\{e\}$, we can replace $G$ by $G/\ker(\ad_K^{H_1})$ and obtain $H_1,H_2$ as needed.

Since all the groups and maps are defined over $K_0$ the theorem now descends to $G(\CK_0)$ as well.  This ends the proof of Theorem \ref{T: main}.
\end{proof}

\begin{remark}\label{R: result is ss}
In Theorem \ref{T: main} it is not claimed that $H_1$ is definably semisimple, though we believe it is true. We expect a standard proof using the tools developed in the unpublished paper \cite{GisHalMac} (and \cite[\S 6]{AcHa}). Note, however, that if $G$ in the theorem is definably connected  or has almost $\bk$-rank  $0$ then it follows easily that $H_1$ is definably semisimple.
\end{remark}

As a special case, we get: 
\begin{corollary}
    Let $\CK_0$ be as above. If a group $G$, definable in $\CK_0$,  is definably simple, then it is definably isomorphic to either a $K_0$-linear group or a  $\bk_0$-linear $H$. 
\end{corollary}

We also have the following.

\begin{corollary}\label{C: ss is fo} 
Let $\CK_0\prec  \CK$ be as above.
    Let  $G$ be a $K_0$-definable group. Then $G(K_0)$ is definably semisimple  if and only if $G(K)$ is. 
\end{corollary}
\begin{proof} 
    By Proposition \ref{P:exists-infty V-min or T-conv} and  \cite[\S 3]{JohCminimalexist}, we may assume that $\CK_0$ is $p$-adically closed.

    If $G(K)$ is definably semisimple, then so is $G(K_0)$.  So we assume that $G(K_0)$ is  definably semisimple and show that so is $G(K)$.

     By Theorem \ref{T:johnson question}(2),  $G$ is $K$-pure; so by Theorem \ref{T: main}, there exists a finite normal subgroup $H_0\trianglelefteq G$ with $G/H_0 (\CK_0)$ definably isomorphic to a $K_0$-linear group. Note that $(G/H_0)(\CK_0)$ is  definably semisimple by Corollary \ref{C: semisimple in quotients}. As $\CK_0$ eliminates $\exists^\infty$ it follows that $(G/H_0)(\CK)$ is definably semisimple as well. However, since $H_0$ is finite,  $G(\CK)$ is definably semisimple.  
\end{proof}

\appendix
\section{Auxiliary results on power-bounded $T$-convex valued fields}
In this appendix, we prove two results on power bounded $T$-convex valued fields. The first, stating that definable subsets of $K$ are finite boolean combinations of ball cuts, is due to Holly \cite[Theorem 4.8]{holly} in the case of RCVF. In full generality it was proved by Tyne,  \cite[Page 94]{tynephd}, but never published. Tyne's proof builds on a deep result, dubbed the valuation property (also not published in the required generality). As a service to the community, we provide an alternative, more direct  proof. The second result shows, using a theorem of Johnson's \cite{JohCminimalexist}, uniform finiteness for all imaginary sorts.

From now on, $\CK$ denotes a power bounded $T$-convex valued field. We remind some standard notation. 

\subsection{Definable subsets of $K$}
If $C\subseteq K$ is any convex set, by $x<C$ we mean that $x<y$ for all $y\in C$ and  $x\leq C$ is defined similarly. For convex sets $C_1, C_2$ we write  $C_1<C_2$ if $x<y$ for any $x\in C_1$ and $y\in C_2$, similarly $C_1\leq C_2$.

By a \emph{definable cut} in $K$ we mean a pair of disjoint definable convex sets $\CC=(C_1,C_2)$, such that $C_1<C_2$ and $C_1\cup C_2=K$. A cut  $\CC$ is \emph{realized} if either $C_1$ has a maximum or $C_2$ has minimum.

For a definable function $f$ from $C_1$ (or some open interval containing it) to either $K$ or $\Gamma$ we say that  $\lim_{x\to \CC^-}f(x)=t_0$, if for every $t_1<t_0<t_2$  there exists $x\in C_1$ such that for all $x'>x$ in $C_1$, $t_1<f(x')<t_2$ (and likewise $\lim_{x\to \CC^+}$). 

Following \cite{holly}, we define: 
\begin{definition}
    A definable cut $\CC=(C_1,C_2)$ in $K$ is {\em a ball cut} if there is a  ball $B$ (possibly a point) such that either $C_1=\{x\in K: x<B\}$ (and then $C_2=\{x\in K: B\leq x\}$), or $C_2=\{x\in K: B<x\}$ (and then $C_1=\{x\in K: x\leq B\}$.
\end{definition}

By o-minimality of $\Gamma$, for every definable set $X$, bounded above or below, and $x\in X$,  there exists a maximal ball around $x$ which is contained in $X$. We leave the following easy observation to the reader.

\begin{lemma}\label{L:at most two closed max balls}
    Let $C\subseteq K$ be a convex definable subset and let $b_1,b_2,b_3$ be maximal balls in $C$ with $b_1<b_2<b_3$. Then $b_2$ is necessarily an open ball.
\end{lemma}

\begin{proposition}\label{ball interval} If $\CC=(C_1,C_2)$ is definable cut with $C_1,C_2 \neq \emptyset$, then $\CC$ is a ball cut. As a corollary, every definable subset of $K$ is a boolean combination of balls and intervals.
\end{proposition}
\begin{proof}
Since every definable subset of $K$ is a finite union of convex sets \cite[Corollary 3.14]{TconvexI}, it will suffice to prove the first clause of the statement. So assume that  $\CC=(C_1,C_2)$ as given is an unrealized cut (if realized then $\CC$  is a ball cut with a trivial ball).  For every $x\in C_1$, let  $B_x$  denote the maximal ball in $C_1$ containing $x$ (since $C_2\neq \emptyset $ such a ball exists) and let $r(x)\in \Gamma$ be its radius. Note that $r(x)$ is (weakly) increasing with $x$. We start with the following.
\begin{claim}\label{C: ball cut}
 Keeping the above notation, if $r(x)$ stabilizes as $x\to \mathcal C^-$ then $\CC$ is a ball cut.
\end{claim}
\begin{claimproof}
    Notice that $r(x)$ is (possibly weakly) increasing.  Assume that $r(x)=r_0$ for sufficiently large $x$ in $C_1$. After re-scaling, assume that $r_0=0$.
    
    If $B_x$ is the same ball for all sufficiently large $x\in C_1$ then $\mathcal C$ is a ball cut, so assume that for every $x\in C_1$ there is some $x'>x$  in $C_1$ such that $B_x\neq B_{x'}$. By Lemma \ref{L:at most two closed max balls}, for all sufficiently large $x$, all the $B_x$ are open. Thus, for any $x\in C_1$, the closed ball $B_{\geq 0}(x)$ intersects $C_2$. As every ball is convex, we have $B_{\geq 0}(x_1)=B_{\geq 0}(x_2)$ for all sufficiently large elements of $C_1$; let $B$ be this closed ball. After translating, we may assume that $B=\CO$. 

    As a result, the map $x\mapsto x+\m$ maps $(B\cap C_1, B\cap C_2)$ into a cut in $\bk$. By o-minimality of $\bk$, this cut is realized, namely either the left side has a maximum or the ride side has a minimum. In the first case, $C_1$ has a right side ball and in the second case $C_2$ has a left side ball.
\end{claimproof}

By the claim, we may assume that $r(x)$ does not stabilize, as $x$ increases in $C_1$. 

Using definable Skolem functions, \cite[Remark 2.7]{vdDries-Tconvex}, we find a definable $h:C_1\to K$ such that for all $x\in C_1$, $r(x)=v(h(x))$. Let $\CL_{omin}$ be the language of the underlying o-minimal reduct (i.e., $\CL_{omin}=\CL(T)$). By \cite[Corollary 2.8]{vdDries-Tconvex}, there exists an $\CL_{omin}$-definable 
function $\widehat h:I\to K$ such that $h=\hat h$ on an end segment of $C_1^-$, which we may assume equals to $I\cap C_1$.  Since $\mathcal C$ is an unrealized cut and $I$ is an $\CL_{omin}$-definable interval containing an end segment of $C_1$ then necessarily $I\cap C_2\neq \0$. Shrinking $I$ (without losing the property that $I\cap C_i\neq 0$ for $i=1,2$) we may assume that $h$ is strictly monotone and continuous. 

By replacing, if needed,  $h$ by $-h$ (and $\widehat h$ by $-\widehat h$) we may assume that $\widehat h$ is strictly decreasing.\\

 \noindent\textbf{Case 1: } $\lim\limits_{x\to \CC^-} r(x)=\infty$. In this case $\lim\limits_{x\to \CC^-}\widehat h(x)=0$. Thus, the function $\widehat h$, which is strictly decreasing and continuous, takes a convex set of the form 
 $\{x\in C_1:x>c\}$, for some $c\in C_1\cap I$, onto an open interval $(0,d)$, with $d=\widehat h(c)$.

 Since $\widehat h$ is  $\CL_{omin}$-definable, so is its inverse function $\widehat h^{-1}\restriction(0,d)$. By o-minimality, and since $\hat h^{-1}$ is strictly decreasing and bounded, it takes the interval $(0,d)$ to an interval of the form $(c,a)$, for some $a\in K$,  and therefore $a$ realizes the cut $\mathcal C$, contradicting our assumption. \\

 \noindent\textbf{Case 2: } $\lim\limits_{x\to \CC^-} r(x)=r_0\in \Gamma$. 
Since $r(x)$ does not stabilize, then $r(x)=v(h(x))<r_0$ for all $x\in C_1$. After re-scaling, we may assume that $r_0=0$, so $v(\widehat h(x))<0$ for all $x\in C_1\cap I$ and   $\lim\limits_{x\in \CC^-}v(\widehat h(x))=0$. Thus,  for all $x\in C_2\cap I$, we  have $v(h_1(x))\geq 0$, and by continuity there must be an element $x\in C_2\cap I$ with $v(\widehat h(x))=0$. Hence, there is some $x_2\in C_2\cap I$ such that for all $x\in C_2$, if $x<x_2$ then $v(\widehat h(x))=0$. 

Consequently, $x\in C_2\cap I\iff \widehat h(x)\in \CO$. Let $(C_1',C_2')$ be the ball cut  $C_1'=\{y\in K:y\leq \CO\}$ and let $J=\widehat h(I)$. Then $J\cap C_i'\neq \emptyset$, for $i=1,2$, and  $\widehat h^{-1}$ is strictly decreasing (from $J$ to $I$). 
For simplicity, let $g=\widehat h^{-1}$.

For any $y\in \CO\cap J$, let $B_y\subseteq C_2$ be the maximal ball containing $g(y)\in C_2$, and denote its radius by $r'(y)$. We may assume that $y\mapsto B_y$ does not stabilize as $y\to (J\cap \CO)^+$ (otherwise $\CC$ is a ball cut, and we are done) and thus, by Lemma \ref{L:at most two closed max balls}  the $B_y\sub C_2$ are open. By \cite[Proposition 4.2]{vdDries-Tconvex}, $r'(y)$ stabilizes for sufficiently large $y\in  J$. Since $g$
sends $\CO\cap J$ to $C_2\cap I$, it follows that for some $c\in C_2$, all maximal balls $B\sub C_2$, with $B<c$, have the same radius. We can now conclude that $\mathcal C$ is a ball cut, using Claim \ref{C: ball cut} (with the roles of $C_1$ and $C_2$ interchanged), thus finishing the proof of Proposition \ref{ball interval}.
\end{proof}

The fact that $\CK$ is definably spherically complete is a consequence of $0$-h-minimality of $\CK$, \cite[Lemma 2.7.1]{hensel-min}. The proof there is not hard, though it implicitly uses Tyne's theorem. We give here a different proof using the previous proposition.
\begin{corollary} 
$\CK$ is definably spherically complete.
\end{corollary}
\begin{proof} 
Let $\{B_t:t\in T\}$ be a definable chain of balls in $K$. Assume towards contradiction that  $\bigcap_{t\in T} B_t= \emptyset$. Let $r(B_t)\in \Gamma$ be the valuative radius of $B_t$.

We define two definable convex sets $C_1,C_2$ by

\[C_1=\{x\in K: \exists t \,\, x<B_t\}\,\,;\,\, C_2=\{x\in K: \exists t \,\, B_t<x\}.\]

Since balls are convex,  our assumption implies that $\CC=(C_1,C_2)$ is a definable, unrealized, cut. By Proposition \ref{ball interval}, this is a ball cut.  For simplicity (the other cases are similar), we assume that $C_1=\{ x\in K: x\leq B\}$ for some ball $B$. Translating and re-scaling, we may assume that $B$ is either $\CO$ or $\bm$.

Let  $B_0=\bigcup\limits_{t\in T} B_t$. We define a function $r:B_0\to \Gamma$ 
by $r(x)=\sup\{r(B_t):x\in B_t\}$. Using definable Skolem functions, we find a definable function $h:B_0\to K$, such that $v(h(x))=r(x)$. 

Assume that $B=\CO$. By \cite[Proposition 4.2]{vdDries-Tconvex}, the function $v(h(x))$, restricted  to $\CO$, eventually stabilizes as $x\to \CC^-$. This implies that the  chain of balls $B_t$ has a minimal element (there is a bijection between the balls and their radii), contradicting our assumption that the intersection of the chain is empty.

Assume that  $B=\bm$ and  consider $h \restriction C_2$. Let $\CC'=(C_1',C_2')$, where $C_1'=\{x\in K : x\leq \CO\}$. As $x\to \CC^+$, we get that $x^{-1}\to \CC^-$, so applying \cite[Proposition 4.2]{vdDries-Tconvex} to $h(x^{-1})$, we conclude that $v(h(x))$ must stabilize as $x\to \CC^+$, again reaching a contradiction.
\end{proof}

\subsection{Elimination of $\exists^\infty$ in the $T$-convex power bounded case}

We now show that $\CK^{eq}$ eliminates $\exists^\infty$; the proof utilizes a criterion used by Johnson to prove a parallel result for $C$-minimal valued fields, see \cite{JohCminimalexist}.
\begin{proposition}\label{P:exists-infty V-min or T-conv}
    $\CK^{\eq}$ eliminates $\exists^\infty$.
\end{proposition}
\begin{proof} 
We shall apply Johnson's criterion for eliminating $\exists^\infty$, \cite{JohCminimalexist}.   By \cite[Theorem 2.3]{JohCminimalexist}, it suffices to prove that if $X$ is a definable set in $\CK^{eq}$ such that  there exists a definable set $S\subseteq X\times K$ with the function $a\mapsto S_a:=\{b\in K: (a,b)\in S\}$  injective on $X$, then $\exists^{\infty}$ is eliminated on $X$. Namely, if $\{Y_t:t\in T\}$ is a definable family of subsets of $X$ then there is a bound on the size of those $Y_t$ that are finite.

Let $X$ be such a definable set (with $S\sub X\times K$ as in the assumption). As $\CK$ is weakly o-minimal (and saturated), there exists $k\in \Nn$ such that each $S_a$ is a finite union of at most $k$ convex sets. By partitioning $X$, we may assume that each $S_a$ consists of exactly $k$ convex sets.  Let $X'=X\times \{1,\dots,k\}$ and let  $S'\subseteq X'\times K$ the set satisfying that $S'_{a,i}$ is the $i$-th convex component of $S_a$. 

    It is sufficient to prove that $\exists^\infty$ is eliminated on $X'$:  Indeed, if $\exists^\infty$ is not eliminated on $X$  then there exists a definable family of subsets $\{Y_t: t\in T\}$ of $X$ and a sequence $\{t_n\}$, such that $|Y_{t_n}|$ is finite and tends to $\infty$. We then define a family of finite subsets of $X'$ as follows: For $i=1,\dots, k$, let 
    \[Y'_{t,i}=\{\text{ the $i$-th convex component of $S_a: a\in Y_t$}\}.\] 
    Since $|Y_{t_n}|\to \infty$  one of the $|Y_{ t_n,i}|$ must tend to $\infty$, thus $X'$ does not eliminate $\exists^\infty$.

    We now replace $X$ by $X'$ and $S$ by $S'$, so we may assume that each $S_a$ is a convex subset of $K$. By Proposition \ref{ball interval}, 
        every $S_a$ is a boolean combination of intervals and balls; so by convexity it must be of the form $B_1\square_1 x\square_2 B_2$, where each $B_i$ is either a point or a ball and $\square_i\in \{<,=,\leq\}$.
        Thus, every $S_a$ is coded by a pair of balls (for simplicity, we consider singletons as balls), so it is sufficient to treat the case where each $S_a$ is a ball , namely we may assume that $X$ is a set of balls.
        Let $\{Y_t:t\in T\}$ be a definable family of subsets of $X$. We claim that there is a bound on the size of the finite $Y_t$ in the family. We reduce the problem to the bound, in families, on the number of convex components of subsets of $K$, as well as the o-minimality of $\Gamma$.

We conclude the proof as in \cite[\S 3]{JohCminimalexist}.
If a ball $b$ belongs to a finite $Y_t$ then it contains a ball $b'\in Y_t$ which is minimal with respect to inclusion. Thus, we may assume that for every $t\in T$, every ball in $Y_t$ contains a minimal ball in $Y_t$ (the set of all such $t$ is definable).

We first note that whenever $Y_t$ is finite, each convex component of the definable set $\bigcup\{b\in Y_t: b \mbox{ minimal} \}$ consists of a single minimal ball in $Y_t$. Indeed, the union of finitely many (but more than one), necessarily pairwise disjoint, balls is not a convex set. 

Thus, we may assume now that for each $Y_t$ in the family, each convex component of the definable set $\bigcup\{b\in Y_t: b \mbox{ minimal} \}$ consists of a single minimal ball in $Y_t$ (this is a definable property of $t$). By the bound on the number of convex components, it follows that there is a bound on the number of minimal balls in each $Y_t$.



Assume towards contradiction that the number of balls in those finite $Y_t$ is not uniformly bounded.  Then, by the bound on the number of minimal balls in $Y_t$,  there are chains of balls in $Y_t$, as $t$ varies,  of unbounded size. This is impossible, as this would imply that the sets  $\{r(B): B\in Y_t\}$ (where $r(B)$ is the valuative radius of $B$) are finite of unbounded size (as $t$ ranges over $T$). Since $\Gamma$ is o-minimal and stably embedded,  definable families of finite subsets of unbounded size do not exist. 
\end{proof}

Let us conclude with an example demonstrating that general weakly o-minimal expansions of groups do not necessarily eliminate $\exists^\infty$ in the  imaginary sorts: 
\begin{example}
    Our goal is to construct an ordered $\Qq$-vector space with a discretely ordered definable family   of  convex subgroups. 
    
    Let $\CR_\Zz$ be a real closed valued field $R$ with value group $\Qq$ together with a predicate $Z\sub \Qq$ for the set of integers. Let $z: \Qq\to \Zz$ be the upper integer value.  Let  $\CM$ be the $2$-sorted structure reduct of $\CR_\Zz$ consisting of the ordered  $\mathbb Q$-vector space $R^{\mathbb Q}=(R,<,+,\{\lambda_{q}\}_{q\in \mathbb Q})$, the sort $(\mathbb Z,<)$ and the function $\zeta: R\to \Zz$ given by $z\circ v$. 

    It is not hard to check that, after adding the function symbols for the successor and predecessor on $\Zz$, the structure $\CM$ has quantifier elimination. It follows that the induced structure on $R$ is weakly o-minimal. It is also not hard to see that $\CM$ is inter-definable with the expansion of the $1$-sorted structure  $R^{\mathbb Q}$ by a binary relation $B$ on $R$, defined by  $B(x,y) \Leftrightarrow \zeta(x)\geq \zeta(y)$. Since $(\mathbb Z,<)$ is interpretable then $\exists^\infty$ cannot be eliminated in the imaginary sorts.

\end{example}
    We expect that also weakly o-minimal expansions of fields do not necessarily eliminate $\exists^\infty$ in their imaginary sorts (although T-convex structures, even if not power bounded, do eliminate $\exists^\infty$).

\subsection*{Conflicts of Interests:} None.

\subsection*{Financial Support:} The first author was partially supported by ISF grants No. 555/21 and 290/19. The second author was supported by ISF grant No. 555/21. The third author was supported by ISF grant No. 290/19.

\bibliographystyle{plain}
\bibliography{harvard}

\end{document}